\documentclass{amsart}
\usepackage{tikz,tikz-cd}
\usepackage{url}
\usepackage{graphicx}
\usepackage{color}
\usepackage{times}
\usepackage{cite}
\usepackage{enumitem,latexsym}
\usepackage{amsmath,amssymb}
\usepackage{amsthm}
\usepackage{verbatim}
\usepackage{tabularx}
\usepackage{mathrsfs}

\newtheorem{thm}{Theorem}[section]
\newtheorem*{theorem*}{Theorem}
\newtheorem*{acknowledgement*}{Acknowledgement}
\newtheorem{cor}[thm]{Corollary}

\newtheorem{lem}[thm]{Lemma}
\newtheorem{prop}[thm]{Proposition}
\theoremstyle{definition}
\newtheorem{defn}[thm]{Definition}
\theoremstyle{remark}
\newtheorem{rem}[thm]{Remark}

\numberwithin{equation}{section}

\newcommand{\set}[1]{\left\{#1\right\}}
\newcommand{\Real}{\mathbb R}

\newcommand{\func}[1]{\ensuremath{\operatorname{#1}} }

\newcommand{\Div}[0]{\func{div}}

\newcommand{\dist}[0]{\mathrm{dist}}

\newcommand{\spt}[0]{\func{spt}}

\newcommand{\cC}[0]{\mathcal{C}}

\title{A Mountain-Pass Theorem for Asymptotically Conical Self-Expanders}

\author{Jacob Bernstein}
\address{Department of Mathematics, Johns Hopkins University, 3400 N. Charles Street, Baltimore, MD 21218}
\email{bernstein@math.jhu.edu}

\author{Lu Wang}
\address{Department of Mathematics, California Institute of Technology, 1200 E. California Boulevard, Pasadena, CA 91125}
\email{drluwang@caltech.edu}

\begin{document}

\begin{abstract}
We develop a min-max theory for asymptotically conical self-expanders of mean curvature flow. In particular, we show that given two distinct strictly stable self-expanders that are asymptotic to the same cone and bound a domain, there exists a new asymptotically conical self-expander trapped between the two.
\end{abstract}

\maketitle

\section{Introduction} \label{IntroSec}

\subsection{Main results}
A \emph{hypersurface}, i.e., a properly embedded codimension-one submanifold, $\Sigma\subset\mathbb{R}^{n+1}$, is a \emph{self-expander} if
\begin{equation} \label{ExpanderEqn}
\mathbf{H}_\Sigma=\frac{\mathbf{x}^\perp}{2}.
\end{equation}
Here 
$$
\mathbf{H}_\Sigma=\Delta_\Sigma\mathbf{x}=-H_\Sigma\mathbf{n}_\Sigma=-\mathrm{div}_\Sigma(\mathbf{n}_\Sigma)\mathbf{n}_\Sigma
$$
is the mean curvature vector, $\mathbf{n}_\Sigma$ is the unit normal and $\mathbf{x}^\perp$ is the normal component of the position vector. Self-expanders arise naturally in the study of mean curvature flow. Indeed, $\Sigma$ is a self-expander if and only if the family of homothetic hypersurfaces
$$
\set{\Sigma_t}_{t>0}=\set{\sqrt{t}\Sigma}_{t>0}
$$
is a \emph{mean curvature flow} (MCF), that is, a solution to the flow
\begin{equation}
\left(\frac{\partial \mathbf{x}}{\partial t}\right)^\perp=\mathbf{H}_{\Sigma_t}.
\end{equation}
Self-expanders model the behavior of a MCF as it emerges from a conical singularity \cite{AIC} and also model possible long time behavior of the flow \cite{EHAnn}. Self-expanders arise variationally as stationary points, with respect to compactly supported variations, of the functional
$$
E[\Sigma]=\int_{\Sigma} e^{\frac{|\mathbf{x}|^2}{4}}\, d\mathcal{H}^n
$$
where $\mathcal{H}^n$ is the $n$-dimensional Hausdorff measure.

There are no closed self-expanders, instead the natural class to consider are those that are asymptotically conical. More precisely, given an integer $l\geq 2$, a hypersurface $\Sigma\subset\mathbb{R}^{n+1}$ is \emph{$C^{l}$-asymptotically conical} if there is a $C^l$-regular cone $\cC$ -- i.e., a dilation invariant $C^l$-hypersurface in $\mathbb{R}^{n+1}\setminus\set{\mathbf{0}}$ -- so that $\lim_{\rho\to0^+}\rho\Sigma=\cC$ in $C^l_{loc}(\mathbb{R}^{n+1}\setminus\set{\mathbf{0}})$. When this occurs write $\mathcal{C}(\Sigma)=\cC$. Fix a choice of unit normal to $\cC$. For any $C^2$-asymptotically conical hypersurface $\Sigma$ with $\cC(\Sigma)=\cC$, we fix a choice of unit normal to $\Sigma$ that is compatible with that of $\cC$. Using this normal, let $\Omega_-(\Sigma)$ be the open set whose boundary is $\Sigma$ and whose outward normal agrees with that of $\Sigma$. For $C^2$-asymptotically conical hypersurfaces $\Sigma$ and $\Sigma^\prime$ with $\cC(\Sigma)=\cC(\Sigma^\prime)=\cC$, define $\Sigma\preceq\Sigma^\prime$ provided $\Omega_-(\Sigma)\subseteq\Omega_-(\Sigma^\prime)$.

The main result of this paper is the following:

\begin{thm}\label{MainThm}
For $n\geq 2$, let $\cC$ be a $C^3$-regular cone in $\mathbb{R}^{n+1}$. Suppose $\Gamma_-$ and $\Gamma_+$ (not necessarily connected) are distinct strictly stable $C^2$-asymptotically conical self-expanders with $\mathcal{C}(\Gamma_-)=\mathcal{C}(\Gamma_+)=\cC$ and $\Gamma_-\preceq\Gamma_+$. Then there exists a $C^2$-asymptotically conical (possibly singular) self-expander $\Gamma_0\neq\Gamma_\pm$ with $\cC(\Gamma_0)=\cC$ and $\Gamma_-\preceq\Gamma_0\preceq\Gamma_+$ and that has codimension-$7$ singular set.
\end{thm}	

Repeatedly applying Theorem \ref{MainThm} we obtain some refined properties for $\Gamma_0$ up to dimension six.

\begin{cor} \label{RefinedMainCor}
For $2\leq n\leq 6$, let $\cC$ and  $\Gamma_\pm$ be given as in Theorem \ref{MainThm}. Let $\Omega=\Omega_-(\Gamma_+)\setminus\overline{\Omega_-(\Gamma_-)}$ be the open region between $\Gamma_-$ and $\Gamma_+$. Then there exists a smooth self-expander $\Gamma_0$ that is $C^2$-asymptotic to $\cC$ and so $\Gamma_-\preceq\Gamma_0\preceq\Gamma_+$ and $\Gamma_0\cap\Omega\neq \emptyset$. Moreover, if $\cC$ is also generic, that is, there is {\bf no} self-expander $C^2$-asymptotic to $\cC$ with nontrivial Jacobi fields that fix the infinity, then $\Gamma_0$ may be taken to be an unstable self-expander.
\end{cor}

In \cite{BWBanach, BWProper, BWDegree} we adapt ideas of White \cite{WhiteDegree} and develop a degree-theoretic method to produce asymptotically conical self-expanders of prescribed topological type. In particular, we show that there is an open set of cones in $\mathbb{R}^3$ so that for each cone in the set there exist three distinct self-expanders asymptotic to the cone, two of which are topological annuli and the third is the union of two disks. Another application of Theorem \ref{MainThm} is a generalization of this fact to higher dimensions.

\begin{cor} \label{TopologyCor}
For $2\leq n \leq 6$, there exists an open set $\mathcal{U}$ of $C^{3}$-regular cones so that for every $\cC\in \mathcal{U}$ there exist at least two distinct connected smooth self-expanders $C^2$-asymptotic to $\cC$ and at least one disconnected smooth self-expander $C^2$-asymptotic to $\cC$.
\end{cor}

\subsection{Overview of the proof of Theorem \ref{MainThm}}
To prove Theorem \ref{MainThm}, we establish a mountain-pass theorem, i.e., Theorem \ref{MinMaxThm}, for asymptotically conical self-expanders. Min-max theory for minimal hypersurfaces has seen a lot of recent development -- see, e.g., \cite{AMN}, \cite{CaD}, \cite{ChaL}, \cite{ChoM}, \cite{CGK}, \cite{DeLellisRamic}, \cite{DeLellisTasnady}, \cite{GaG}, \cite{GuLWZ}, \cite{IMN}, \cite{KLS}, \cite{LioMN}, \cite{MNWillmore},  \cite{MNMorse}, \cite{MNYau}, \cite{MNS}, \cite{Montezuma}, \cite{MontezumaNoncompact}, \cite{Riviere}, \cite{Song}, \cite{ZWang}, \cite{Zhou}, \cite{ZhouZhu} and references therein.  We in particular note \cite{KZ} wherein the authors develop the min-max theory for self-shrinking solutions to mean curvature flow in $\Real^3$.  Our approach is inspired by De Lellis-Tasnady's \cite{DeLellisTasnady} (see also \cite{ColdingDeLellis}) reformulation of work of Almgren-Pitts \cite{Pitts} and Simon-Smith \cite{Smith}. One may think of Theorem \ref{MinMaxThm} as an analog, in the non-compact setting, of work of De Lellis-Ramic \cite{DeLellisRamic} on min-max theory for compact minimal surfaces with fixed boundary (see also \cite{Montezuma} for an approach more closely using Almgren-Pitts' work). However, we emphasize two differences. The first is that in \cite{DeLellisRamic} the critical point produced by the min-max procedure does not need to be trapped between the two strictly stable critical points. As this kind of property plays a crucial role in our later application \cite{BWUniqueExpander}, we take care to establish it. The second, and more essential, point of difference is that, as the expander functional is infinite valued, we work directly with a certain relative functional (whose existence and properties were established in \cite{BWExpanderRelEnt}) to produce a critical point. While it may be possible to produce a min-max critical point by taking a limit of compact critical points (as is commonly done to produce local minimizers; see also \cite{CaD} for a min-max construction of geodesic lines) it seems difficult to guarantee that will produce a geometrically distinct asymptotically conical self-expander. 

A key point of our arguments is an appropriate choice of function space $\mathfrak{Y}$ on Grassmann $n$-plane bundles -- See Section \ref{FuncSpaceSec} for the precise definition. The space $\mathfrak{Y}$ is motivated by a notion of relative expander entropy $E_{rel}$ which is introduced in our earlier work \cite{BWExpanderRelEnt} (See also Section \ref{ExpanderRelEntSec}) and satisfies the following properties: First, to each hypersurface $\Sigma$ with finite relative expander entropy one can associate a unique element $V_{\Gamma}$ of the dual space $\mathfrak{Y}^*$ -- See Proposition \ref{YEstProp}. Second, the relative expander entropy is well defined on the subspace $\overline{\mathfrak{Y}^*_\cC}\subseteq\mathfrak{Y}^*$ consisting of elements that are obtained by taking limits of sequences $V_{\Gamma_i}$  in the weak-* topology, and every element of $\overline{\mathfrak{Y}^*_\cC}$ has a (weighted) varifolds decomposition -- See Lemma \ref{RelEntropyBndLem}. Third, any element of $\overline{\mathfrak{Y}^*_\cC}$ that are $E_{rel}$-minimizing to first order (which is defined in Section \ref{RelMinimizingSec}) has no concentration of relative expander entropy at infinity -- See Item (1) of Proposition \ref{StationaryProp} -- this uses a non-compact vector field in a similar manner as is done by Ketover-Zhou \cite{KZ} to deal with the same issue in their setting. These properties ensure that a pull-tight procedure (cf. \cite{ColdingDeLellis} and \cite{DeLellisRamic}) can be carried out for elements of $\overline{\mathfrak{Y}^*_\cC}$ to produce a minimizing sequence so that any min-max sequence converges to an element of $\overline{\mathfrak{Y}^*_\cC}$ that is $E_{rel}$-minimizing to first order -- See Proposition \ref{PullTightProp}. 

Another key point is that, to address the local regularity of the min-max limit, it is convenient to consider an open domain $\Omega^\prime$ slightly thicker than $\tilde{\Omega}=\overline{\Omega_+(\Gamma_-)}\cap\overline{\Omega_-(\Gamma_+)}$, the closed region between $\Gamma_-$ and $\Gamma_+$, so that $\overline{\Omega^\prime}\cap\Omega_\pm(\Gamma_\pm)$ is foliated in a certain manner by asymptotically conical hypersurfaces with expander mean curvature pointing towards $\Gamma_\pm$ -- See Proposition \ref{ThickeningProp}. By the strong maximum principle (see \cite{SolomonWhite} and \cite{WhiteMax}), the varifold associated to the min-max limit is supported in $\tilde{\Omega}\subset\Omega^\prime$. Thus, with slight modifications, the arguments of \cite{DeLellisRamic} (see also \cite{ColdingDeLellis} and \cite{DeLellisTasnady}) for the regularity of min-max minimal surfaces can be adapted to show the varifold is supported on some $E$-stationary hypersurface trapped between $\Gamma_-$ and $\Gamma_+$ that has codimension-$7$ singular set -- See Propositions \ref{AMProp} and \ref{InteriorRegProp}.

Moreover, we prove the varifold associated to the min-max limit has multiplicity one. This is a stronger statement than what is shown for more general min-max results. Indeed, it is an simple consequence of the finiteness of relative expander entropy that the tangent cone of the varifold at infinity is the multiplicity-one cone $\cC$ -- See Item (4) of Proposition \ref{StationaryProp}. As there are no compact self-expanders, the constancy theorem immediately implies that the varifold has multiplicity one. 

Lastly, to ensure the min-max method produces a new self-expander, we show a uniform lower bound on relative expander entropy max-value of any sweepout of $\tilde{\Omega}$, i.e., a path of hypersurfaces connecting $\Gamma_-$ and $\Gamma_+$. Namely, the maximum of relative entropy of slices in a sweepout is strictly larger than that of $\Gamma_-$ and $\Gamma_+$ -- See Proposition \ref{StrictLowBndRelEntProp}. Compare with \cite[Lemma 11.1]{DeLellisRamic} for sweepouts of compact surfaces, however, our approach is different from there and uses calibration type arguments and observations on properties of $E$-minimizers. 

\subsection{Organization}
In Section \ref{PrelimSec} we fix the notation for the remainder of the paper and recall the main results from \cite{BWExpanderRelEnt} that will be used in this paper. In Section \ref{FoliationThickenSec} we construct the open domain $\Omega^\prime$ that is slightly thicker than the region $\tilde{\Omega}$ between $\Gamma_-$ and $\Gamma_+$ and satisfies good properties. In Section \ref{FuncSpaceSec} we introduce the space $\mathfrak{Y}(\overline{\Omega^\prime})$ of functions on the Grassman $n$-plane bundle over $\overline{\Omega^\prime}$. We generalize estimates in \cite{BWExpanderRelEnt} to elements of $\mathfrak{Y}(\overline{\Omega^\prime})$ and show the relative expander entropy is well defined for elements of the subset $\overline{\mathfrak{Y}^*_\cC}(\overline{\Omega^\prime};\Lambda)\subset\mathfrak{Y}^*(\overline{\Omega^\prime})$ consisting of those that are obtained by taking limits of hypersurfaces with relative expander entropy bounded by $\Lambda$. In Section \ref{FlowSec} we study the action of flows generated by a suitable class of vector fields on elements of $\mathfrak{Y}^*(\overline{\Omega^\prime})$. In particular, we derive the first variation formula and prove the continuous dependence of the flow action on vector fields. In Section \ref{StationarySec} we introduce the modified action of flows on elements of $\overline{\mathfrak{Y}^*_\cC}(\overline{\Omega^\prime};\Lambda)$ and collect properties for elements of $\overline{\mathfrak{Y}^*_\cC}(\overline{\Omega^\prime};\Lambda)$ that are $E_{rel}$-minimizing to first order in $\overline{\Omega^\prime}$. In Section \ref{MountainPassSec} we define parametrized families of asymptotically conical hypersurfaces in $\overline{\Omega^\prime}$, in particular, sweepouts of $\tilde{\Omega}$,  and the relative entropy min-max value for parametrized families. We then establish a mountain-pass theorem for a certain homotopically closed set of parametrized families. In Section \ref{LowerBndSec} we use calibration type arguments to obtain a uniform lower bound on relative entropy max-value of any sweepout of $\tilde{\Omega}$. In Section \ref{ExistSweepoutSec} we show the existence of at least one sweepout of $\tilde{\Omega}$ and finish the proof of the results.

\subsection{Acknowledgements}
The first author was partially supported by the NSF Grants DMS-1609340 and DMS-1904674 and the Institute for Advanced Study with funding provided by the Charles Simonyi Endowment. The second author was partially supported by the NSF Grants DMS-2018221(formerly DMS-1811144) and DMS-2018220 (formerly DMS-1834824), the funding from the Wisconsin Alumni Research Foundation and a Vilas Early Career Investigator Award by the University of Wisconsin-Madison, and a von Neumann Fellowship by the Institute for Advanced Study with funding from the Z\"{u}rich Insurance Company and the NSF Grant DMS-1638352. The second author would like to thank Brian White and Jonathan Zhu for helpful discussions.

\section{Preliminaries} \label{PrelimSec}
We fix notation and certain conventions we will use throughout the remainder of the paper. We also recall certain background and facts we will need from \cite{BWExpanderRelEnt}.

\subsection{Basic notions} \label{NotionSec}
Here is the list of notations that we use throughout the paper:

\begin{tabularx}{0.95\textwidth}{lX}
$B_R(p)$ & the open ball in $\mathbb{R}^{n+1}$ centered at $p$ with radius $R$; \\
$\bar{B}_R(p)$ & the closed ball in $\mathbb{R}^{n+1}$ centered at $p$ with radius $R$; \\
$B_R^{\mathcal{X}}(p)$ & the open ball in the Banach space $\mathcal{X}$ centered at $p$ with radius $R$; \\
$\bar{B}_R^{\mathcal{X}}(p)$ & the closed ball in the Banach space $\mathcal{X}$ centered at $p$ with radius $R$; \\
$\mathcal{T}_\delta(U)$ & the $\delta$-tubular (open) neighborhood of a set $U$; \\
$\overline{U}$ or $\mathrm{cl}(U)$ & the closure of a set $U$; \\
$\partial U$ & the topological boundary of a set $U$; \\
$\partial^* U$ & the reduced boundary of a Caccioppoli set $U$; \\
$\nabla_{\Sigma}$ & the gradient on a Riemannian manifold $\Sigma$; \\
$\Div_{\Sigma}$ & the divergence on a Riemannian manifold $\Sigma$; \\
$\Delta_{\Sigma}$ & the Laplacian on a Riemannian manifold $\Sigma$;\\
$\Gamma_-, \Gamma_+$ & two asymptotically conical self-expanders with $\Gamma_+$ lying ``above" $\Gamma_-$;\\
$\Omega, \tilde{\Omega}$ & the open and closed regions bounded by $\Gamma_-$ and $\Gamma_+$;\\
$\Omega', \Omega''$ & certain open regions ``thickening" $\tilde{\Omega}$.
\end{tabularx}

We omit the center of a ball when it is the origin. We also omit the subscript, $\Sigma$, in the gradient, divergence and Laplacian when it is Euclidean space.  See Section \ref{ConventionSec} for precise definitions of $\Gamma_\pm$, $\Omega$, and $\tilde{\Omega}$.  See Proposition \ref{ThickeningProp} for properties of $\Omega'$ and $\Omega''$.

\subsection{Partial ordering of asymptotically conical hypersurfaces} \label{PartialOrderSec}
Let $\mathcal{C}$ be a $C^2$-regular cone in $\mathbb{R}^{n+1}$ so the link $\mathcal{L}(\mathcal{C})$ is an embedded codimension-one $C^2$ submanifold of $\mathbb{S}^n$. Clearly, $\mathcal{L}(\mathcal{C})$ separates $\mathbb{S}^n$ and we fix a closed set $\omega\subset\mathbb{S}^n$ so that $\partial\omega=\mathcal{L}(\mathcal{C})$. Given a Caccioppoli set $U$, the boundary $\Sigma=\partial^* U$ is \emph{asymptotic to $\mathcal{C}$} if 
$$
\lim_{\rho\to 0^+} \mathcal{H}^n\lfloor(\rho\Sigma)=\mathcal{H}^n\lfloor\mathcal{C}.
$$
When this happens set $\mathcal{C}(\Sigma)=\mathcal{C}$. For such $\Sigma$, let $\Omega_-(\Sigma)$ be the subset of $\mathbb{R}^{n+1}\setminus\overline{\Sigma}$ so that $\partial\Omega_-(\Sigma)=\Sigma$ and
$$
\lim_{\rho\to 0^+} \mathrm{cl}(\rho \Omega_-(\Sigma))\cap\mathbb{S}^n=\omega \mbox{ as closed sets}.
$$
Let $\Omega_+(\Sigma)=\mathbb{R}^{n+1}\setminus\overline{\Omega_-(\Sigma)}$. 

For boundaries $\Sigma_0$ and $\Sigma_1$ both asymptotic to $\mathcal{C}$, write
$$
\Sigma_0\preceq\Sigma_1 \mbox{ provided $\Omega_-(\Sigma_0)\subseteq\Omega_-(\Sigma_1)$}.
$$
If $\Sigma_0\preceq\Sigma_1$, then define 
$$
\mathcal{C}(\Sigma_0,\Sigma_1)=\set{U\colon\mbox{$U$ is a Caccioppoli set with $\Omega_-(\Sigma_0)\subseteq U\subseteq\Omega_-(\Sigma_1)$}}.
$$

\subsection{Relative expander entropy} \label{ExpanderRelEntSec}
Let $U\in\mathcal{C}(\Sigma_0,\Sigma_1)$ with $\Sigma=\partial^* U$, and let $Z=\Omega_+(\Sigma_0)\cap\Omega_-(\Sigma_1)$. For a function $\psi\in C^0_c (\overline{Z}\times \mathbb{S}^n)$ define
$$
E[\Sigma, \Sigma_0; \psi]=\int_{\Sigma} \psi(p, \mathbf{n}_\Sigma(p)) e^{\frac{|\mathbf{x}(p)|^2}{4}} \, d\mathcal{H}^n-\int_{\Sigma_0} \psi(p,\mathbf{n}_{\Sigma_0}(p)) e^{\frac{|\mathbf{x}(p)|^2}{4}} \, d\mathcal{H}^n.
$$
We remark that $E[\Sigma,\Sigma_0; \psi]$ is linear in $\psi$ and that, when $\psi$ is {even} (see \eqref{eveneqn}), $E[\Sigma, \Sigma_0; \psi]$ is independent of the choice of $\mathbf{n}_{\Sigma}$ or $\mathbf{n}_{\Sigma_0}$.  

For a function $\psi\in C^0_{loc}(\overline{Z}\times\mathbb{S}^n)$ let
$$
E_{rel}[\Sigma, \Sigma_0; \psi; \bar{B}_R]=\int_{\Sigma\cap \bar{B}_R} \psi(p, \mathbf{n}_\Sigma(p)) e^{\frac{|\mathbf{x}|^2}{4}} \, d\mathcal{H}^n -\int_{\Sigma_0\cap \bar{B}_R} \psi(p, \mathbf{n}_{\Sigma_0}(p) )e^{\frac{|\mathbf{x}|^2}{4}} \, d\mathcal{H}^n
$$
and
$$
E_{rel}[\Sigma, \Sigma_0; \psi]=\lim_{R\to \infty} E_{rel}[\Sigma, \Sigma_0; \psi; \bar{B}_R]
$$
when this limit exists. Observe that if $\psi$ has compact support, then the limit is defined.  

\subsection{The space $\mathfrak{X}$} \label{XSpaceSec}
Let $Y\subseteq\mathbb{R}^{n+1}$. For a function $\psi\in Lip(Y\times \mathbb{S}^n)$ and any $p\in Y$, define $\hat{\psi}_p(\mathbf{v})=\psi(p,\mathbf{v})$ and
$$
\nabla_{\mathbb{S}^n}\psi(p,\mathbf{v})=\nabla_{\mathbb{S}^n}\hat{\psi}_p(\mathbf{v}).
$$
Consider the Banach space
$$
\mathfrak{X}(Y)=\set{\psi\in Lip(Y\times\mathbb{S}^n) \colon \Vert \psi\Vert_{\mathfrak{X}}<\infty}
$$
where
$$
\Vert \psi \Vert_{\mathfrak{X}}=\Vert\psi\Vert_{Lip}+\Vert\nabla_{\mathbb{S}^n}\psi\Vert_{Lip}+\sup_{(p,\mathbf{v})\in Y\times\mathbb{S}^n} (1+|\mathbf{x}(p)|)|\nabla_{\mathbb{S}^n}\psi(p,\mathbf{v})|.
$$

\subsection{Conventions}\label{ConventionSec}
We now set conventions we will use in the remainder of the paper. Fix a $C^3$-regular cone $\mathcal{C}$ in $\mathbb{R}^{n+1}$ and pick a closed set $\omega\subset\mathbb{S}^n$ so $\partial\omega=\mathcal{L}(\mathcal{C})$. Using $\omega$, if $\Gamma$ is a hypersurface that is $C^2$-asymptotic to $\mathcal{C}$, then we will always choose the unit normal $\mathbf{n}_\Gamma$ so it points into $\Omega_+(\Gamma)$ and out of $\Omega_-(\Gamma)$. Let $\Gamma_-$ and $\Gamma_+$ be two distinct strictly stable self-expanders that are both $C^2$-asymptotic to $\mathcal{C}$ and $\Gamma_-\preceq\Gamma_+$.  Denote by $\Omega=\Omega_+(\Gamma_-)\cap\Omega_-(\Gamma_+)$ the open region between $\Gamma_-$ and $\Gamma_+$ and by $\tilde{\Omega}=\overline{\Omega_+(\Gamma_-)}\cap\overline{\Omega_-(\Gamma_+)}$ the closed region. In general $\overline{\Omega}$ may not equal $\tilde{\Omega}$ as $\Gamma_-$ and $\Gamma_+$ may have common components. Let $\Gamma_-^\prime$ and $\Gamma_+^\prime$ be two hypersurfaces, not necessarily self-expanders, both asymptotic to $\mathcal{C}$ and assume $\Gamma_-^\prime\preceq\Gamma_-\preceq\Gamma_+\preceq\Gamma_+^\prime$. If $\Omega^\prime=\Omega_+(\Gamma_-^\prime)\cap\Omega_-(\Gamma_+^\prime)$, then $\Omega\subseteq\Omega^\prime$. We further assume $\Omega^\prime$ is thin at infinity relative to $\Gamma_-$ in the sense that there are constants $C_0=C_0(\Omega^\prime,\Gamma_-)>0$ and $\mathcal{R}_0=\mathcal{R}_0 (\Omega^\prime,\Gamma_-)>1$ so that, for all $R>\mathcal{R}_0$,
$$
\Omega^\prime\backslash B_R \subset \mathcal{T}_{C_0 R^{-n-1}e^{-\frac{R^2}{4}}} (\Gamma_-).
$$
We will fix a choice of such $\Omega^\prime$ in Section \ref{FoliationThickenSec}.

If $\Gamma$ is a $C^2$-asymptotically conical self-expander, then it follows from the interior estimates for MCF (see, e.g., Theorem 3.4 and Remark 3.6 (ii) of \cite{EHInterior}) that 
\begin{equation} \label{LinearDecayEqn}
C_{\Gamma,l}=\sup_{p\in\Gamma} \left((1+|\mathbf{x}(p)|\sum_{i=1}^l |\nabla_{\Gamma}^i \mathbf{n}_\Gamma (p)|)\right)<\infty.
\end{equation}

We also introduce the following test functions. Let
$$
\phi_{R,\delta}(p)=\left\{\begin{array}{ll} 1 & \mbox{if $p\in B_R$} \\ 1-\frac{|\mathbf{x}(p)|-R}{\delta} & \mbox{if $p\in \bar{B}_{R+\delta}\setminus B_R$} \\
0 & \mbox{if $p\in\mathbb{R}^{n+1}\setminus\bar{B}_{R+\delta}$} \end{array} \right.
$$
be a cutoff. Let
$$
\alpha_{R_1, R_2, \delta}(p)= \phi_{R_2,\delta}(p)-\phi_{R_1-\delta, \delta}(p)
$$
be the cutoff adapted to the closed annulus $\bar{B}_{R_2}\setminus B_{R_1}$.

Next, a set $Y\subseteq\mathbb{R}^{n+1}$ is \emph{quasi-convex} if there is a constant $C>0$ so that any pair of points $p,q\in Y$ can be joined by a curve $\beta$ in $Y$ with 
$$
\mathrm{Length}(\beta)\leq C|\mathbf{x}(p)-\mathbf{x}(q)|.
$$ 
We will always assume any set $Y$ under consideration to be quasi-convex. By \cite[Theorem 4.1]{Heinonen}, the space of Lipschitz functions on $Y$ is the same as the $W^{1,\infty}$ space and the norms are equivalent. 

Finally, a function $\psi\colon Y\times\mathbb{S}^n\to\mathbb{R}$ is \emph{even} if 
\begin{equation}\label{eveneqn}
\psi(p,\mathbf{v})=\psi(p,-\mathbf{v}) \mbox{ for all }(p,\mathbf{v})\in Y\times\mathbb{S}^n.
\end{equation}
Observe that an even function is naturally identified with a function on the Grassman $n$-plane bundle of $Y$. We will always assume functions on $Y\times\mathbb{S}^n$ to be even. 

\subsection{Main results of \cite{BWExpanderRelEnt}} \label{ExpanderRelEntSec}
Follow the conventions of Section \ref{ConventionSec}. We will need several results from \cite{BWExpanderRelEnt} in this paper. The first of theses is the existence of the relative expander entropy in the obstacle setting for domains that are thin at infinity; see Proposition 3.4 and Theorem 3.1 of \cite{BWExpanderRelEnt}.

\begin{prop} \label{RelEntropyProp}
There are constants $\bar{R}_0=\bar{R}_0(\Omega^\prime,\Gamma_-)>1$ and $\bar{C}_0=\bar{C}_0(\Omega^\prime, \Gamma_-)>0$ so that if $\Gamma=\partial^* U$ for some $U\in\mathcal{C}(\Gamma_-^\prime,\Gamma_+^\prime)$ and $\psi\in Lip(\overline{\Omega^\prime})$ satisfies $\psi\geq 0$, then, for any $R_2>R_1+\delta>R_1>\bar{R}_0$, 
$$
E[\Gamma,\Gamma_-;\phi_{R_2,\delta}\psi] \geq E[\Gamma,\Gamma_-; \phi_{R_1,\delta}\psi]-\bar{C}_0 R_1^{-1} \Vert \psi \Vert_{Lip}.
$$
As consequences, one has
\begin{enumerate}
\item For any $R_2>R_1>\bar{R}_0$,
$$
E_{rel}[\Gamma, \Gamma_-; \bar{B}_{R_2}]\geq E_{rel}[\Gamma, \Gamma_-; \bar{B}_{R_1}]-\bar{C}_0 R_1^{-1}.
$$
In particular, $E_{rel}[\Gamma, \Gamma_-]=\lim_{R\to\infty} E_{rel}[\Gamma,\Gamma_-; \bar{B}_R]$ exists (possibly positive infinite) and, for any $R> \bar{R}_0$, satisfies the estimate
$$
E_{rel}[\Gamma, \Gamma_-]\geq E_{rel}[\Gamma, \Gamma_-; \bar{B}_R]-\bar{C}_0 R^{-1}.
$$
\item For any $\delta>0$, $\lim_{R\to\infty} E[\Gamma,\Gamma_-; \phi_{R,\delta}]=E_{rel}[\Gamma,\Gamma_-]$.
\end{enumerate}
\end{prop}

We also need the following weighted estimate; see Propositions 4.6 and 4.8 as well as Theorem 4.1 of \cite{BWExpanderRelEnt}.

\begin{prop} \label{WeightRelProp}
There are constants $\bar{R}_1=\bar{R}_1(\Omega^\prime,\Gamma_-)>1$ and $\bar{C}_1=\bar{C}_1(\Omega^\prime,\Gamma_-)>0$ so that if $\Gamma=\partial^* U$ for some $U\in\mathcal{C}(\Gamma_-^\prime,\Gamma_+^\prime)$ and $\psi\in\mathfrak{X}(\overline{\Omega^\prime})$, then the following is true:
\begin{enumerate}
\item \label{WeightAnnulusItem} For any $R_2>R_1>R_1-\delta>\frac{1}{2}R_1>\bar{R}_1$, 
$$
\left|E[\Gamma,\Gamma_-; \alpha_{R_1,R_2,\delta}\psi] \right| \leq \bar{C}_1 \left(R_1^{-1}+\left|E[\Gamma,\Gamma_-; \alpha_{R_1,R_2,\delta}]\right|\right)\Vert \psi \Vert_{\mathfrak{X}};
$$
\item \label{WeightBallItem} For any $0<\delta<1$ and $R>\bar{R}_1$,
$$
\left|E[\Gamma,\Gamma_-; \phi_{R,\delta}\psi]\right| \leq \bar{C}_1 \left(1+\left|E[\Gamma,\Gamma_-; \phi_{R,\delta}]\right| \right)\Vert \psi \Vert_{\mathfrak{X}};
$$
\item \label{WeightGlobalItem} If, in addition, $E_{rel}[\Gamma,\Gamma_-]<\infty$, then $E_{rel}[\Gamma,\Gamma_-; \psi]$ exists and 
$$
\left| E_{rel}[\Gamma, \Gamma_-; \psi] \right| \leq  \bar{C}_1 \left(1+ \left|E_{rel}[\Gamma, \Gamma_-]\right|\right)\Vert \psi \Vert_{\mathfrak{X}}.
$$
\end{enumerate}
\end{prop}

A key ingredient in the proofs of Propositions \ref{RelEntropyProp} and \ref{WeightRelProp} is the existence of a ``good" vector field near infinity compatible with the self-expander $\Gamma_-$; see \cite[Proposition 3.3]{BWExpanderRelEnt}.

\begin{lem} \label{VectorFieldLem}
There are constants $\bar{R}_2=\bar{R}_2(\Omega^\prime,\Gamma_-)>1$ and $\bar{C}_2=\bar{C}_2(\Omega^\prime,\Gamma_-)>0$ and a smooth vector field $\mathbf{N}\colon \overline{\Omega^\prime}\backslash \bar{B}_{\bar{R}_2} \to \mathbb{R}^{n+1}$ that satisfies:
\begin{enumerate}
\item \label{UnitLengthItem} $|\mathbf{N}|=1$;
\item \label{InitialItem} $\mathbf{N}|_{\Gamma_-}=\mathbf{n}_{\Gamma_-}$;
\item \label{C3EstimateItem} $|\mathbf{x}\cdot \mathbf{N}|+\sum_{i=1}^3 |\nabla^i\mathbf{N}|\leq  \bar{C}_2 (1+|\mathbf{x}|)^{-1}$.
\end{enumerate}
\end{lem}

Finally, we need the following result -- see \cite[Proposition 6.3]{BWExpanderRelEnt} -- which implies that for (possibly singular) hypersurfaces that are ``reasonable" near infinity the relative entropy is finite.

\begin{prop} \label{RelEntropyAnnuliProp}
Fix $\bar{K}_0>0$ and $\bar{\mathcal{R}}_0>1$. There is a radius $\bar{\mathcal{R}}_1=\bar{\mathcal{R}}_1(\Omega^\prime,\Gamma_-,\bar{K}_0,\bar{\mathcal{R}}_0)>\bar{\mathcal{R}}_0$ and a constant $\bar{K}_1=\bar{K}_1(\Omega^\prime,\Gamma_-,\bar{K}_0)>0$ so that if $\Gamma$ is a hypersurface in $\mathbb{R}^{n+1}\setminus\bar{B}_{\bar{\mathcal{R}}_0}$ trapped between $\Gamma_-^\prime\setminus\bar{B}_{\bar{\mathcal{R}}_0}$ and $\Gamma_+^\prime\setminus\bar{B}_{\bar{\mathcal{R}}_0}$ that is asymptotic to $\cC$ and satisfies 
$$
\sup_{p\in\Gamma} (1+|\mathbf{x}(p)|) |A_\Gamma(p)| \leq \bar{K}_0,
$$
then, for any $R_2>R_1>\bar{\mathcal{R}}_1$ and $0<\delta<1$,
$$
\left| E[\Gamma,\Gamma_-; \alpha_{R_1,R_2,\delta}] \right| \leq \bar{K}_1 R_1^{-2}.
$$
\end{prop}

\section{Expander mean convex foliations and a thickening of $\Omega$} \label{FoliationThickenSec}
Continue to use the conventions of Section \ref{ConventionSec}. In Theorem \ref{MainThm} we are asked to find a new self-expander lying in the closed region $\tilde{\Omega}$ between the two strictly stable self-expanders $\Gamma_-$ and $\Gamma_+$. For technical reasons it is more convenient to work with a slightly ``thicker" region $\Omega^\prime$ that has good properties: Namely, it is thin at infinity relative to $\Gamma_-$,  a certain modification of the radial vector field $\mathbf{x}$ points out of the region near infinity, and $\Omega^\prime\cap\Omega_-(\Gamma_-)$ and $\Omega^\prime\cap\Omega_+(\Gamma_+)$ can be foliated in a certain way by expander mean convex hypersurfaces. The purpose of this section is to establish the existence of such a region $\Omega^\prime$.

We use the vector field $\mathbf{N}$ given by Lemma \ref{VectorFieldLem} to define the following modification of the radial vector field near infinity of $\Gamma_-$
\begin{equation} \label{X0VectorFieldEqn}
\mathbf{X}_0= \mathbf{x}-\left(\mathbf{x}\cdot \mathbf{N}\right) \mathbf{N}.
\end{equation}
Observe $\mathbf{X}_0$ is tangent to $\Gamma_-$.  The main result of this section is the following:

\begin{prop} \label{ThickeningProp} 
There exist open subsets $\Omega^\prime$ and $\Omega^{\prime\prime}$ of $\mathbb{R}^{n+1}$ so that $\tilde{\Omega}\subset\Omega^\prime$ and $\overline{\Omega^\prime}\subset\Omega^{\prime\prime}$ and with the following properties:
\begin{enumerate}
\item $\overline{\Omega^{\prime\prime}}\cap \overline{\Omega_\pm(\Gamma_\pm)}$ is foliated by $C^2$-asymptotically conical hypersurfaces, $ \set{\Gamma_s^\pm}_{s\in [0,1]}$ with $\Gamma_0^\pm=\Gamma_\pm$ and $\mathcal{C}(\Gamma_s^\pm)=\mathcal{C}(\Gamma_\pm)=\mathcal{C}$ and, for each $s>0$, $\Gamma_s^\pm$ has expander mean curvature pointing toward $\Gamma_\pm$;
\item $\Omega^\prime$ is thin at infinity relative to $\Gamma_-$ so there are constants $\mathcal{R}_0>1$ and $C_0>0$ so that, for all $R>\mathcal{R}_0$, $\Omega^\prime\setminus B_R\subset\mathcal{T}_{C_0 R^{-n-1} e^{-\frac{R^2}{4}}}(\Gamma_-)$;
\item For each $s\in (0,1]$ there is a radius $\mathcal{R}(s)>0$ so that $\Gamma_s^\pm \cap \Omega^\prime\subset B_{\mathcal{R}(s)}$;
\item There is a radius $\mathcal{R}_1>1$ so that, outside of $B_{\mathcal{R}_1}$,  $\mathbf{X}_0$ points out of $\Omega^\prime$. 
\end{enumerate}
\end{prop}

We start with some basic estimates for the lowest eigenfunction of the stability operator on a connected self-expander. For a self-expander $\Gamma$, let 
$$
L_\Gamma=\Delta_\Gamma+\frac{\mathbf{x}}{2}\cdot\nabla_\Gamma+|A_{\Gamma}|^2-\frac{1}{2}
$$
be the stability operator on $\Gamma$. For integer $l\geq 0$ define the weighted Sobolev space on $Y\subseteq\Gamma$ by
$$
W^l(Y)=\set{f\in H^{l}_{loc}(Y) \colon \Vert f \Vert_{W^l}<\infty}
$$
where 
$$
\Vert f \Vert_{W^l} =\left( \int_{Y} \sum_{0\leq i\leq l} |\nabla_{\Gamma}^i f|^2 e^{\frac{|\mathbf{x}|^2}{4}} \, d\mathcal{H}^n \right)^{\frac{1}{2}}.
$$
Observe $W^l$ is the same as the Banach space $W^l_{\frac{1}{4}}$ introduced in our earlier work \cite{BWDegree}.

\begin{prop} \label{EigenfuncProp}
If $\Gamma$ is a $C^2$-asymptotically conical connected self-expander in $\mathbb{R}^{n+1}$, then there is a unique $\mu>-\infty$ and a unique function $f>0$ on $\Gamma$ so that
$$
\left(L_\Gamma+\mu\right)f=0 \mbox{ with $\Vert f\Vert_{W^0}=1$}.
$$
Moreover, there is a constant $C^\prime_0=C^\prime_0(\Gamma,\mu)>0$ so that 
\begin{equation} \label{C0BndEigenfuncEqn}
\frac{1}{C^\prime_0} \left(1+r^2\right)^{-\frac{1}{2}(n+1-2\mu)} e^{-\frac{r^2}{4}} \leq f \leq C^\prime_0 \left(1+r^2\right)^{-\frac{1}{2}(n+1-2\mu)} e^{-\frac{r^2}{4}}
\end{equation}
and, for any $\delta>0$,  there are constants $C^\prime_m=C^\prime_m(\Gamma,\mu,\delta)>0$ for $m\geq 1$ so that 
\begin{equation} \label{DerBndEigenfuncEqn}
\Vert e^{\frac{r^2}{4+\delta}}\nabla_\Gamma^m f\Vert_{C^0}\leq C^\prime_m
\end{equation}
where $r(p)=|\mathbf{x}(p)|$ for $p\in\Gamma$. We call $\mu$ and $f$ the first eigenvalue and eigenfunction, respectively, of $L_\Gamma$.
\end{prop}

\begin{proof}
As $\Gamma$ is $C^2$-asymptotically conical 
\begin{equation} \label{QuadCurvEqn}
\sup_{p\in\Gamma} (1+|\mathbf{x}(p)|^2)|A_\Gamma(p)|^2<\infty.
\end{equation}
Thus, by standard spectral theory (e.g., \cite[Proposition 4.1]{BWDegree}), $L_\Gamma$ is formally self-adjoint in $W^0(\Gamma)$ and has a discrete spectrum with a finite lower bound. Hence, there is a unique $\mu>-\infty$ and, as $\Gamma$ is connected, a unique positive function $f\in W^1(\Gamma)$ so that
$$
\left(L_\Gamma+\mu\right)f=0 \mbox{ distributionally and $\Vert f\Vert_{W^0}=1$}.
$$
As $\Gamma$ is a self-expander, it is smooth and properly embedded and so standard elliptic regularity theory implies $f\in C^\infty_{loc}(\Gamma)$.
	
Let
$$
\underline{g}=\left(r^{-n-1+2\mu}+r^{-n-2+2\mu}\right) e^{-\frac{r^2}{4}} \mbox{ and } \overline{g}=\left(r^{-n-1+2\mu} -r^{-n-2+2\mu}\right)e^{-\frac{r^2}{4}}.
$$
By the curvature decay \eqref{QuadCurvEqn}, a simple computation in \cite[Lemma A.1]{BWExpanderRelEnt} and a slight modification of the proof of \cite[Proposition A.1]{BWDegree}, there are constants $R_0=R_0(\Gamma, \mu)>1$ and $C=C(\Gamma,f, \mu)>1$ so that:
\begin{itemize}
\item In $\Gamma\setminus B_{R_0}$, $-\frac{1}{2}+\mu+|A_\Gamma|^2<\mu$;
\item In $\Gamma\setminus B_{R_0}$, 
$$
\frac{1}{2} (1+r^2)^{-\frac{1}{2}(n+1-2\mu)}e^{-\frac{r^2}{4}}\leq \underline{g}<\overline{g}\leq 2 (1+r^2)^{-\frac{1}{2}(n+1-2\mu)} e^{-\frac{r^2}{4}};
$$
\item In $\Gamma\setminus B_{R_0}$, $\left(L_\Gamma+\mu\right)\overline{g}<0<\left(L_\Gamma+\mu\right)\underline{g}$;
\item On $\Gamma\cap\partial B_{R_0}$, $C^{-1}\underline{g}\leq f \leq C\overline{g}$;
\item If $h\in C^2(\Gamma\setminus B_{R_0})\cap W^1(\Gamma\setminus B_{R_0})$ satisfies $h= 0$ on $\Gamma\cap \partial B_{R_0}$, then
$$
\int_{\Gamma\setminus B_{R_0}} h^2 e^{\frac{r^2}{4}} \, d\mathcal{H}^n \leq \frac{1}{|\mu|+1} \int_{\Gamma\setminus B_{R_0}} |\nabla_\Gamma h|^2 e^{\frac{r^2}{4}} \, d\mathcal{H}^n.
$$ 
\end{itemize}
Choose a sequence of numbers, $R_i>R_0$, so that $R_i\to\infty$. As $f\in C^\infty_{loc}(\Sigma)$, the Dirichlet problem
$$
\left\{
\begin{array}{cc}
\left(L_\Gamma+\mu\right)g_i=0 & \mbox{in $\Gamma\cap (B_{R_i}\setminus \bar{B}_{R_0})$}\\
g_i=f & \mbox{on $\Gamma\cap\partial B_{R_0}$} \\
g_i=C\overline{g} & \mbox{on $\Gamma\cap\partial B_{R_i}$}
\end{array}
\right.
$$
has a unique smooth solution $g_i$. To see this observe that if $v$ satisfies
$$
0=\mathscr{L}_\Gamma^0 v +2\nabla_\Gamma \log \overline{g}\cdot \nabla_\Gamma v +\frac{(L_\Gamma+\mu)\overline{g}}{\overline{g}} v= \mathscr{L}_\Gamma^0 v+\overline{\mathbf{b}}\cdot  \nabla_\Gamma v+\overline{c} v,
$$
where $\mathscr{L}_{\Gamma}^0=\Delta_{\Gamma}+\frac{\mathbf{x}}{2}\cdot\nabla_\Gamma$, then $(L_\Gamma+\mu) (\overline{g} v) =0$. As $\overline{c}=\overline{g}^{-1}(L_\Gamma+\mu)\overline{g}<0$, it follows from standard elliptic PDE theory that
$$
\left\{
\begin{array}{cc}
\mathscr{L}_\Gamma^0 v_i +\overline{\mathbf{b}}\cdot \nabla_\Gamma v_i +\overline{c} v_i=0 & \mbox{in $\Gamma\cap (B_{R_i}\setminus \bar{B}_{R_0})$}\\
v_i=\overline{g}^{-1} f & \mbox{on $\Gamma\cap\partial B_{R_0}$} \\
v_i=C & \mbox{on $\Gamma\cap\partial B_{R_i}$}
\end{array}
\right.
$$
has a unique solution $v_i$. Hence, $g_i=\overline{g} v_i$ are the claimed solutions. Moreover, by the maximum principle, as $\overline{c}\leq 0$,  $0<v_i\leq C$ and so $0<g_i\leq C \overline{g}$.
	
In a similar fashion, $w_i=\underline{g}^{-1} g_i>0$ satisfies
$$
0=\mathscr{L}_\Gamma^0 w_i +2\nabla_\Gamma\log \underline{g} \cdot \nabla_\Gamma w_i +\frac{(L_{\Gamma} +\mu)\underline{g}}{\underline{g}} w_i.
$$
As $\underline{g}^{-1}((L_{\Gamma} +\mu)\underline{g})>0$, there are no interior minima for $w_i$ and so $w_i\geq C^{-1}$. Hence, 
$$
C^{-1}\underline{g}\leq g_i \leq C \overline{g}.
$$
	
It follows from the Schauder estimates \cite[Theorem 6.2]{GT} and the Arzel\`{a}-Ascoli theorem that, up to passing to a subsequence, the $g_i$ converges in $C^{\infty}_{loc}(\Gamma)$ to a function $g$ which satisfies
$$
\left\{
\begin{array}{cc}
\left(L_\Gamma+\mu\right)g=0 & \mbox{in $\Gamma\setminus\bar{B}_{R_0}$} \\
g=f & \mbox{on $\Gamma\cap\partial B_{R_0}$}
\end{array}
\right.
$$
and $C^{-1}\underline{g} \leq g \leq C\overline{g}$ in $\Gamma\setminus B_{R_0}$.
	
Next we show $g=f$ in $\Gamma\setminus\bar{B}_{R_0}$, from which the $C^0$ estimate of $f$ follows easily. Observe that, by the Schauder estimates, one has $g\in W^1(\Gamma\setminus B_{R_0})\cap C^2(\Gamma\setminus B_{R_0})$. Set $h=g-f$. Thus, $h\in W^1(\Gamma\setminus B_{R_0})\cap C^2(\Gamma\setminus B_{R_0})$ with $h=0$ on $\Gamma\cap\partial B_{R_0}$, and 
$$
\left(L_\Gamma+\mu\right)h=e^{-\frac{r^2}{4}}\mathrm{div}_\Gamma\left(e^{\frac{r^2}{4}}\nabla_\Gamma h\right)+\left(|A_\Gamma|^2-\frac{1}{2}+\mu\right)h=0.
$$
Hence, multiplying the above equation by $h e^{\frac{r^2}{4}}$ and integrating by parts (which is justified by our hypotheses on $h$) give
$$
\int_{\Gamma\setminus B_{R_0}} \left(|\nabla_\Gamma h|^2+\left(\frac{1}{2}-\mu-|A_\Gamma|^2\right)h^2\right) e^{\frac{r^2}{4}} \, d\mathcal{H}^n=0.
$$
The choice of $R_0$ ensures $\frac{1}{2}-\mu-|A_\Gamma|^2>-|\mu|$ in $\Gamma\setminus B_{R_0}$ and
$$
\int_{\Gamma\setminus B_{R_0}} h^2 e^{\frac{r^2}{4}}\, d\mathcal{H}^n\leq \frac{1}{|\mu|+1}\int_{\Gamma\setminus B_{R_0}} |\nabla_\Gamma h|^2 e^{\frac{r^2}{4}}\, d\mathcal{H}^n.
$$
It follows that
$$
0\geq \int_{\Gamma\setminus B_{R_0}} |\nabla_\Gamma h|^2 e^{\frac{r^2}{4}} \, d\mathcal{H}^n-|\mu| \int_{\Gamma\setminus B_{R_0}} h^2 e^{\frac{r^2}{4}} \, d\mathcal{H}^n\geq \int_{\Gamma\setminus B_{R_0}} h^2 e^{\frac{r^2}{4}} \, d\mathcal{H}^n
$$
and so $h=0$. Hence, setting
$$
\overline{C}=\sup_{B_{R_0}\cap \Gamma} \left((1+r^2)^{\frac{1}{2}(n+1-2\mu)}e^{\frac{r^2}{4}} f \right) \; \mbox{ and }\; \underline{C}=\inf_{B_{R_0}\cap \Gamma} \left( (1+r^2)^{\frac{1}{2}(n+1-2\mu)}e^{\frac{r^2}{4}} f\right),
$$
the first estimate holds with
$$
C_0^\prime=\max\set{2C, \overline{C}, \underline{C}^{-1}}.
 $$	
  
Finally, in view of \eqref{LinearDecayEqn}, the claimed estimates on derivatives of $f$ follow from the $C^0$ estimate of $f$ and the Schauder estimates. 
\end{proof}

We then use the first eigenfunction of the stability operator and its estimates to produce good foliations on either side of a strictly stable self-expander.

\begin{lem}\label{MeanConvexLem}
Let $\Gamma\subset \Real^{n+1}$ be a strictly stable self-expander $C^2$-asymptotic to $\cC$. There are positive constants $\epsilon_0=\epsilon_0(\Gamma)$ and $c_0=c_0(\Gamma)$ and a family of hypersurfaces $\set{\Gamma_s}_{s\in [-\epsilon_0,\epsilon_0]}$ so that:
\begin{enumerate}
\item $\Gamma_0=\Gamma$;
\item $\set{\Gamma_s}_{s\in [-\epsilon_0,\epsilon_0]}$ is a foliation;
\item Each $\Gamma_s$ is $C^2$-asymptotically conical with $\cC(\Gamma_s)=\cC$; 
\item For $s\neq 0$, $\Gamma_s$ has expander mean curvature pointing toward $\Gamma_0$;
\item For $p\in \Gamma_s$, 
$$
\dist(p, \Gamma)\geq c_0|s| (1+|\mathbf{x}(p)|^2)^{-\frac{1}{2}(n+1)+2c_0} e^{-\frac{|\mathbf{x}(p)|^2}{4}}.
$$
\end{enumerate}
\end{lem}

\begin{proof}
Suppose $\Gamma=\bigcup_{1\leq j\leq M} \Gamma^j$ where the $\Gamma^j$ are disjoint connected components of $\Gamma$. Let $\mu_j$ and $f_j$ be the first eigenvalue and eigenfunction, respectively, of $L_{\Gamma^j}$ that are given by Proposition \ref{EigenfuncProp}. As $\Gamma$ is strictly stable, so is each $\Gamma^j$ and so $\mu_j>0$. Define positive functions $\mu$ and $f$ on $\Gamma$ by, for $p\in\Gamma^j$, $\mu(p)=\mu_j$ and $f(p)=f_j(p)$. There is an $\epsilon_0\in (0,1)$ so that, for all $s\in [-\epsilon_0,\epsilon_0]$, 
$$
\Gamma_s=\set{\mathbf{f}_s(p)=\mathbf{x}(p)+s f(p)\mathbf{n}_\Gamma(p)\colon p\in\Gamma}
$$
is a hypersurface in $\mathbb{R}^{n+1}$. As $f>0$ one has a foliation $\set{\Gamma_s}_{s\in [-\epsilon_0,\epsilon_0]}$ and so Items (1) and (2) hold. Using estimates of $f$ one readily checks Item (3) holds. Moreover, choosing $c_0=\min\set{1/C_0^{\prime},\mu,1}$, one immediately has that Item (5) holds.

Finally, shrinking $\epsilon_0$, if needed, and appealing to \cite[Lemma A.2]{BWExpanderRelEnt}, one has that the expander mean curvature of $\Gamma_s$ is given by, at $p\in\Gamma$,
$$
\left(H_{\Gamma_s}+\frac{\mathbf{x}}{2}\cdot\mathbf{n}_{\Gamma_s}\right)\circ\mathbf{f}_s=-s L_{\Gamma} f+s^2Q(f,\mathbf{x}\cdot\nabla_{\Gamma} f, \nabla_\Gamma f, \nabla^2_\Gamma f)
$$
where, for some $C=C(\Gamma)>0$, 
$$
|Q| \leq C \left(|f|+|\mathbf{x}\cdot\nabla_\Gamma f|+|\nabla_\Gamma f|+|\nabla_\Gamma^2 f|\right) \left( |f|+|\nabla_\Gamma f|\right).
$$
By the properties of $f$, there is a constant $C^{\prime}=C^{\prime}(\Gamma,\mu)>0$ so that 
$$
|Q| \leq C^{\prime} e^{-\frac{|\mathbf{x}|^2}{3}} \mbox{ and } f \geq \frac{1}{C^\prime} e^{-\frac{|\mathbf{x}|^2}{3}}.
$$
Thus, as $\mu>0$, if $s\in (0,\epsilon_0]$, then
$$
\left(H_{\Gamma_s}+\frac{\mathbf{x}}{2}\cdot\mathbf{n}_{\Gamma_s}\right)\circ\mathbf{f}_s(p) \geq \frac{s \mu}{C^\prime}  e^{-\frac{|\mathbf{x}(p)|^2}{3}}-s^2C^{\prime}  e^{-\frac{|\mathbf{x}(p)|^2}{3}};
$$
if $s\in [-\epsilon_0,0)$, then
$$
\left(H_{\Gamma_s}+\frac{\mathbf{x}}{2}\cdot\mathbf{n}_{\Gamma_s}\right)\circ\mathbf{f}_s(p) \leq \frac{s \mu}{C^\prime}  e^{-\frac{|\mathbf{x}(p)|^2}{3}}+s^2C^{\prime}  e^{-\frac{|\mathbf{x}(p)|^2}{3}}.
$$
Hence, up to further shrinking $\epsilon_0$ so that $\frac{\mu}{C^\prime}\geq 2\epsilon_0 C^{\prime}$, for all $s\in (0,\epsilon_0]$,
$$
\left(H_{\Gamma_s}+\frac{\mathbf{x}}{2}\cdot\mathbf{n}_{\Gamma_s}\right)\circ\mathbf{f}_s(p) \geq \frac{s \mu}{2C^\prime}  e^{-\frac{|\mathbf{x}(p)|^2}{3}}>0
$$
and, for all $s\in [-\epsilon_0,0)$, 
$$
\left(H_{\Gamma_s}+\frac{\mathbf{x}}{2}\cdot\mathbf{n}_{\Gamma_s}\right)\circ\mathbf{f}_s(p) \leq \frac{s \mu}{2C^\prime}  e^{-\frac{|\mathbf{x}(p)|^2}{3}} <0.
$$
That is, Item (4) holds.
\end{proof}

We next show we can perturb an asymptotically conical self-expander, $\Gamma$, on both sides to produce a region enclosed by the two perturbations that is thin at infinity relative to $\Gamma$ and so the outward unit normal of the region points asymptotically more in the radial direction.

\begin{lem} \label{RadialLem}
Let $\Gamma\subset \Real^{n+1}$ be a $C^2$-asymptotically conical self-expander. Given a constant $\kappa>0$, there is a radius $\mathcal{R}_2=\mathcal{R}_2(\Gamma,\kappa)>1$ and a constant $C_1=C_1(\Gamma)>0$ so that
if
$$
\varphi(p)= \kappa |\mathbf{x}(p)|^{-n-1} e^{-\frac{|\mathbf{x}(p)|^2}{4}},
$$ 
then
$$
\Sigma^\pm=\set{\mathbf{x}(p)\pm \varphi(p) \mathbf{n}_{\Gamma}(p)\colon p\in \Gamma\backslash \bar{B}_{\mathcal{R}_2}}
$$
are hypersurfaces and, for $p\in \Sigma^\pm$,
$$
\left| \mathbf{n}_{\Sigma^\pm}(p) -\mathbf{n}_{\Gamma}(\Pi_{\Gamma}(p))\mp\frac{\mathbf{x}(p)}{2} \varphi(p)\right|\leq C_1 |\mathbf{x}(p)|^{-1} \varphi(p)
$$
where $\Pi_\Gamma$ is the nearest point projection to $\Gamma$.
\end{lem}

\begin{proof}
As $\Gamma$ is $C^2$-asymptotically conical, there is an $\epsilon=\epsilon(\Gamma)\in (0,1)$ so that $\mathcal{T}_{\epsilon}(\Gamma)$ is a regular neighborhood of $\Gamma$. Pick an $\mathcal{R}_2=\mathcal{R}_2(\Gamma, \kappa)>1$ large enough so that
$$
\sup_{p\in \Gamma\backslash \bar{B}_{\mathcal{R}_2}}\left( \varphi +|\nabla_{\Gamma} \varphi|\right)< \epsilon.
$$
It follows that
$$
\Sigma^\pm =\set{\mathbf{x}(p)\pm \varphi(p) \mathbf{n}_{\Gamma}\colon p\in \Sigma\backslash \bar{B}_{\mathcal{R}_2}}
$$
are smooth hypersurfaces.  There is a constant $K_0=K_0(\Gamma)$ so that if $p\in \Sigma^\pm$ and $q=\Pi_{\Gamma}(p)$, then 
\begin{equation} \label{DiffNormalEqn}
\begin{split}
\left|	\mathbf{n}_{\Sigma^\pm}(p) - \left(\mathbf{n}_{\Gamma}(q) \mp \nabla_{\Gamma}\varphi(q)\right) \right|&=\left| \mathbf{n}_{\Sigma^\pm}(p) - \mathbf{n}_{\Gamma}(q) \pm \nabla_{\Gamma}\varphi(q) \right| \\
&\leq K_0 \left( |\varphi(q)|^2+ |\nabla_{\Gamma} \varphi(q)|^2 \right).
\end{split}
\end{equation}
See \cite[(2.27)-(2.28)]{WangUniqueness} for a derivation of this estimate. 
   
Up to increasing $\mathcal{R}_2$, one may ensure that, for $p\in \Sigma^\pm$,
\begin{equation} \label{DiffDistVarphiEqn}
|\Pi_{\Gamma}(p)-\mathbf{x}(p)|+|\varphi(\Pi_\Gamma(p))|+|\nabla_\Gamma\varphi(\Pi_\Gamma(p))| \leq |\mathbf{x}(p)|^{-3}<1<\frac{1}{2}|\mathbf{x}(p)|.
\end{equation}
It follows that there is a constant $K_1=K_1(n)>0$ so that for such $p$ 
\begin{equation} \label{DiffVarphiEqn}
|\varphi(p)-\varphi(\Pi_{\Gamma}(p))| \leq K_1 |\mathbf{x}(p)|^{-2} \varphi(p)
\end{equation}
and thus
\begin{equation} \label{SquareVarphiEqn}
|\varphi(\Pi_\Gamma(p))|^2 \leq |\mathbf{x}(p)|^{-3} \left(1+K_1 |\mathbf{x}(p)|^{-2} \right) \varphi(p).
\end{equation}
Moreover, as $\Gamma$ is a $C^2$-asymptotically conical self-expander, there is a constant $K_2=K_2(\Gamma)$ so that, for all $q\in \Gamma\backslash \bar{B}_{\mathcal{R}_2}$,
$$
|\mathbf{x}(q)\cdot \mathbf{n}_{\Gamma}(q)|=2|H_{\Gamma}(q)| \leq K_2 |\mathbf{x}(q)|^{-1}.
$$
Hence, as
$$
\nabla_{\Gamma} \varphi =-\frac{\mathbf{x}^\top}{2} \varphi -(n+1) \mathbf{x}^\top |\mathbf{x}|^{-2} \varphi 
$$
it follows that there is a $K_3=K_3(\Gamma)$ so that, for $q\in \Gamma\backslash \bar{B}_{\mathcal{R}_2}$,
$$
\left|  \nabla_{\Gamma} \varphi(q) +\frac{\mathbf{x}(q)}{2} \varphi(q) \right|\leq K_3 |\mathbf{x}(q)|^{-1}\varphi(q).
$$
This together with \eqref{DiffDistVarphiEqn}-\eqref{DiffVarphiEqn} implies that, for $p\in\Sigma^\pm$,
\begin{equation} \label{GradientVarphiEqn}
\left| \nabla_{\Gamma} \varphi(\Pi_\Gamma(p))+\frac{\mathbf{x}(p)}{2} \varphi(p) \right| \leq K_4 |\mathbf{x}(p)|^{-1} \varphi(p)
\end{equation}
for some $K_4=K_4(\Gamma)$.

Therefore, combining \eqref{DiffNormalEqn}-\eqref{GradientVarphiEqn} implies that there is a $K_5=K_5( \Gamma)$ so that if $p\in \Sigma^\pm$, then
$$
\left| \mathbf{n}_{\Sigma^\pm}(p)-\mathbf{n}_{\Gamma}(\Pi_{\Gamma}(p))\mp\frac{\mathbf{x}(p)}{2} \varphi(p)\right|\leq K_5 |\mathbf{x}(p)|^{-1} \varphi(p).
$$
The result follows by setting $C_1=K_5$.
\end{proof}

\begin{proof}[Proof of Proposition \ref{ThickeningProp}]
Let $\set{\Upsilon^\pm_s}_{s\in [-\epsilon^\pm_0,\epsilon^\pm_0]}$ be the foliation given by Lemma \ref{MeanConvexLem} using $\Gamma=\Gamma_\pm$ with the corresponding constants $\epsilon_0^\pm$ and $c_0^\pm$. Set $\epsilon_0=\min\set{\epsilon^+_0, \epsilon_0^-}$ and $c_0=\min\set{c_0^+, c_0^-}$. Set
$$
\Gamma_s^\pm=\Upsilon^\pm_{\pm\epsilon_0 s}.
$$
Let
$$
\Omega^{\prime\prime}=\Omega_-(\Gamma_1^+)\cap \Omega_+(\Gamma_1^-),
$$
and  this region satisfies Item (1) by construction.
    
By \cite[Proposition 2.1]{BWExpanderRelEnt} there is a radius $R_0=R_0(\Gamma_-,\Gamma_+)>1$ and a constant $K_0=K_0(\Gamma_-,\Gamma_+)>0$ and a function $u \colon \Gamma_-\backslash \bar{B}_{R_0}\to \Real$ so that
$$
\Gamma_+\setminus\bar{B}_{2 R_0}\subset\set{\mathbf{x}(p)+u(p)\mathbf{n}_{\Gamma_-}(p)\colon p\in\Gamma_-\setminus\bar{B}_{R_0}}\subset\Gamma_+
$$
and $u$ satisfies 
$$
|u|\leq K_0 |\mathbf{x}|^{-n-1} e^{-\frac{|\mathbf{x}|^2}{4}}.
$$

With $\Gamma=\Gamma_-$ and $\kappa=2K_0$ apply Lemma \ref{RadialLem} to produce a function $\varphi$, a radius $\mathcal{R}_2>1$ and hypersurfaces $\Sigma^-$ and $\Sigma^+$.  Observe that outside $\bar{B}_{\mathcal{R}}$ with $\mathcal{R}=2\max\set{R_0, \mathcal{R}_2}$,  the choice of $\varphi$ ensures that $\Sigma^-$ is the graph of $-\varphi$ over $\Gamma_-$ lying entirely inside $\Omega_-(\Gamma_-)$, while $\Sigma^+$ is the graph of $\varphi$ over $\Gamma_-$ and $\Sigma^+$ lies, by construction, entirely within $\Omega_+(\Gamma_+)$. Observe that the growth rate of $\varphi$ and Item (5) of Lemma \ref{MeanConvexLem} ensure that, up to increasing $\mathcal{R}$, one may take $\Sigma^\pm \subset \Omega^{\prime\prime}$. For the same reason, $\Omega_+(\Gamma_s^+)\cap \Sigma^+$ is contained in a compact set for all $s\neq 0$ and the same is true of $\Omega_-(\Gamma_s^-)\cap \Sigma^-$.
    
Pick $\mathcal{R}^\prime\geq \mathcal{R}$, so there is a function $v_+ \colon \Gamma_+\backslash \bar{B}_{\mathcal{R}^\prime}\to \Real$ so that
$$
\Sigma^+ \setminus \bar{B}_{2\mathcal{R}^\prime} \subset \set{ \mathbf{x}(p)+v_+(p)\mathbf{n}_{\Gamma_+}(p)\colon p\in \Gamma_+ \setminus \bar{B}_{\mathcal{R}^\prime}}.
$$
Set $v_-=-\varphi$ so the graph of $v_-$ over $\Gamma_-$ is a subset of $\Sigma^-$. 
    
By using cutoffs appropriately, one may extend $v_\pm$ to functions $\hat{v}_\pm\colon \Gamma_\pm \to \Real$ so that $\hat{v}_\pm=v_\pm$ outside $\bar{B}_{2\mathcal{R}^\prime}$, $\hat{v}_+>0>\hat{v}_-$ and 
$$
\Gamma_\pm^\prime=\set{\mathbf{x}(p) +\hat{v}_\pm \mathbf{n}_{\Gamma_\pm}(p) \colon p\in \Gamma_\pm}
$$
are smooth hypersurfaces contained in $\Omega^{\prime\prime}$. In particular, up to increasing $\mathcal{R}^\prime$, one has $\Gamma_-^\prime \setminus B_{\mathcal{R}^\prime}=\Sigma^-\setminus B_{\mathcal{R}^\prime}$ and $\Gamma_+^\prime\setminus B_{\mathcal{R}^\prime}=\Sigma^+\setminus B_{\mathcal{R}^\prime}$. Clearly, $\Gamma_\pm^\prime$ are asymptotically conical with asymptotic cone $\cC$. Let 
$$
\Omega^\prime=\Omega_+(\Gamma_-^\prime)\cap \Omega_-(\Gamma_+^\prime).
$$
By construction one has $\tilde{\Omega} \subset \Omega^\prime$ and $\overline{\Omega^\prime}\subset\Omega^{\prime\prime}$ and $\Omega^\prime$ is thin at infinity relative to $\Gamma_-$. 
    
Finally, as $\mathbf{X}_0\cdot \mathbf{N}=0$,  outside of $B_{\mathcal{R}_1}$ where $\mathcal{R}_1\geq 2\mathcal{R}^\prime>2$ one has, by the construction of $\Sigma^\pm$, 
\begin{align*}
\left|\mathbf{X}_0\cdot \mathbf{n}_{\Sigma^\pm } \mp \frac{1}{2}|\mathbf{x}|^2 \varphi \right| &=	\left|\mathbf{X}_0\cdot (\mathbf{n}_{\Sigma^\pm }-\mathbf{N}) \mp \frac{1}{2}|\mathbf{x}|^2 \varphi \right|\\
&=\left|\mathbf{X}_0\cdot (\mathbf{n}_{\Sigma^\pm }-\mathbf{N}) \mp \frac{1}{2}\left(\mathbf{X}_0\cdot \mathbf{x} +(\mathbf{x}\cdot \mathbf{N})^2\right)\varphi \right|\\
&\leq |\mathbf{X}_0| \left|\mathbf{n}_{\Sigma^\pm }- \mathbf{N}\mp\frac{1}{2}\mathbf{x} \varphi\right|+\frac{1}{2}|\mathbf{X}_0| |\mathbf{x}\cdot\mathbf{N}|^2 \varphi \\
&\leq 2|\mathbf{x}| \left|\mathbf{n}_{\Sigma^\pm}-\mathbf{N}\mp\frac{1}{2}\mathbf{x} \varphi\right|+\bar{C}_2 |\mathbf{x}|^{-2} \varphi\\
&\leq  (2 C_1+\bar{C}_2) \varphi
\end{align*}
where $\bar{C}_2$ is the constant given by Lemma \ref{VectorFieldLem}. It follows that, up to increasing $\mathcal{R}_1$, one has
$$
\mathbf{X}_0\cdot \mathbf{n}_{\Sigma^+ }\geq \frac{1}{4} |\mathbf{x}|^2 \varphi>0 \mbox{ while } \mathbf{X}_0\cdot \mathbf{n}_{\Sigma^- }\leq - \frac{1}{4} |\mathbf{x}|^2 \varphi<0.
$$
Hence, outside of $B_{\mathcal{R}_1}$, $\mathbf{X}_0$ points out of $\Omega^\prime$.
\end{proof}

\section{Function spaces} \label{FuncSpaceSec}
We introduce several function spaces that extend the space $\mathfrak{X}$. In the next sections we will study deformation properties in these spaces. In what follows we use the conventions of Section \ref{ConventionSec}.

\subsection{The space $\mathfrak{Y}$} \label{YSpaceSec}
First of all we introduce the following norm on $C^0(Y\times \mathbb{S}^n)$ where $Y$ is a quasi-convex unbounded domain in $\mathbb{R}^{n+1}$. For $\psi\in C^0(Y\times \mathbb{S}^n)$ let
$$
\Vert \psi\Vert_{\mathfrak{W}}= \sup_{(p, \mathbf{v})\in Y\times \mathbb{S}^n} (|\mathbf{x}(p)|+1)^{n+1} e^{\frac{|\mathbf{x}(p)|^2}{4}} \left| \psi (p, \mathbf{v}) \right|.
$$
Let 
$$
\mathfrak{W}(Y)=\set{\psi\in C^0(Y\times \mathbb{S}^n) \colon \Vert \psi\Vert_{\mathfrak{W}}<\infty}
$$ 
be the space of rapidly decaying continuous functions. It is readily checked that $\mathfrak{W}(Y)$ is a Banach space.

Notice that $\mathfrak{X}(Y)\cap \mathfrak{W}(Y)$ is non-empty, and $\mathfrak{X}(Y)$ and $\mathfrak{W}(Y)$ are both continuously embedded in $C^0(Y\times\mathbb{S}^n)$. We introduce the following natural norm on the vector space $\mathfrak{Y}^\prime(Y)=\mathfrak{X}(Y)+ \mathfrak{W}(Y)$:
$$
\Vert \psi \Vert_\mathfrak{Y}=\inf\set{ \Vert \zeta\Vert_{\mathfrak{X}}+\Vert \xi\Vert_{\mathfrak{W}} \colon \psi=\zeta+\xi}.
$$
It follows from the interpolation theory, \cite[Chapter 3, Theorem 1.3]{BennettSharpley}, that $\mathfrak{Y}^\prime (Y)$ with this norm is a Banach space. Although in general $\mathfrak{Y}^\prime(Y)$ is not separable, we have the following result:

\begin{prop} \label{YSpaceProp}
There is a closed subspace $\mathfrak{Y}(Y)\subseteq \mathfrak{Y}^\prime (Y)$ defined by
$$
\mathfrak{Y}(Y)=\mathrm{span}_{\mathbb{R}}\set{\mathbf{1}}\oplus\mathfrak{Y}_0(Y)=\set{c \mathbf{1}+\psi\colon c\in \Real, \psi\in\mathfrak{Y}_0(Y)}
$$
where $\mathbf{1}$ is the constant function equal to $1$ and $\mathfrak{Y}_0(Y)$ is the closure of $C_c^0(Y\times\mathbb{S}^n)$ in $\mathfrak{Y}^\prime(Y)$. The space $\mathfrak{Y}(Y)$ satisfies
\begin{enumerate}
\item $C^0_c(Y\times \mathbb{S}^n)\subseteq \mathfrak{Y}(Y)$;
\item  $\mathbf{1}\in \mathfrak{Y}(Y)$ and $\Vert \mathbf{1}\Vert_{\mathfrak{Y}} =1$;
\item For $\psi \in \mathfrak{Y}_0(Y)$,  $\frac{1}{4}\left( |c|+\Vert \psi\Vert_{\mathfrak{Y}}\right)\leq \Vert c \mathbf{1}+\psi\Vert_{\mathfrak{Y}}\leq |c|+\Vert \psi\Vert_{\mathfrak{Y}}$;
\item $\mathfrak{Y}(Y)$ is an algebra and $\Vert \psi_1 \psi_2 \Vert_{\mathfrak{Y}} \leq \Vert \psi_1 \Vert_{\mathfrak{Y}}\Vert \psi_2 \Vert_{\mathfrak{Y}}$;
\item  $\mathfrak{Y}(Y)$ is separable;
\item $C^\infty_c(Y\times \mathbb{S}^n)$ is dense in $\mathfrak{Y}_0(Y)$.
\end{enumerate}
\end{prop}

\begin{proof}
The first property is immediate from the definition. The second is a consequence of the fact that $\Vert \mathbf{1}\Vert_{\mathfrak{X}}=1$, $Y$ is an unbounded domain and elements of $\mathfrak{W}$ decay rapidly. In particular, if $\mathbf{1}=\xi+\zeta$ for $\xi\in \mathfrak{W}$ and $\zeta\in \mathfrak{X}$, then $\lim_{p\to \infty} \zeta(p,\mathbf{v})=1$ and so $\Vert \zeta \Vert_{\mathfrak{X}}\geq 1$. For the third item, observe that, by the triangle inequality and the second item,
$$
\Vert c\mathbf{1}+\psi\Vert_{\mathfrak{Y}}\leq |c|+\Vert \psi \Vert_{\mathfrak{Y}}.
$$
This verifies the second inequality. As elements of $\mathfrak{Y}_0(Y)$ must decay as one approaches infinity, arguing as in the second item gives $\Vert c\mathbf{1}+\psi\Vert_{\mathfrak{Y}}\geq |c|$. If $\Vert \psi \Vert_{\mathfrak{Y}}\leq 2|c|$, then the first inequality follows. Suppose $\Vert \psi \Vert_{\mathfrak{Y}}> 2|c|$. By definition, for every $\epsilon>0$ there are $\xi\in \mathfrak{W}(Y)$ and $\zeta \in \mathfrak{X}(Y)$ so that $c\mathbf{1}+\psi=\xi +\zeta$ and
$$
\Vert \xi\Vert_{\mathfrak{W}}+ \Vert \zeta \Vert_{\mathfrak{X}}\leq \Vert c\mathbf{1}+\psi\Vert_{\mathfrak{Y}}+\epsilon.
$$
As $\psi = \xi +(\zeta-c\mathbf{1})$, one has
$$
\Vert \psi \Vert_{\mathfrak{Y}}\leq \Vert \xi \Vert_{\mathfrak{W}}+ \Vert \zeta-c\mathbf{1}\Vert_{\mathfrak{X}}\leq \Vert \xi \Vert_{\mathfrak{W}}+ \Vert \zeta\Vert_{\mathfrak{X}}+|c|\leq |c|+ \Vert c\mathbf{1}+\psi\Vert_{\mathfrak{Y}}+\epsilon.
$$
Hence, 
$$
\frac{1}{4}\left (|c|+ \Vert \psi \Vert_{\mathfrak{Y}}\right)\leq \frac{1}{2}\Vert \psi \Vert_{\mathfrak{Y}} \leq 
 \Vert c\mathbf{1}+\psi\Vert_{\mathfrak{Y}}+\epsilon.
$$
Sending $\epsilon\to 0$ verifies the first inequality and proves the third item.

The fourth follows from the definition of the $\mathfrak{Y}$ norm and $\mathfrak{Y}(Y)$ and the fact that both $\mathfrak{W}(Y)$ and $\mathfrak{X}(Y)$ are algebras while $\mathbf{1}$ is the multiplicative identity. Indeed, one readily checks that if $\xi\in \mathfrak{X}(Y)$ and $\zeta \in \mathfrak{W}(Y)$, then $\xi \zeta\in \mathfrak{W}(Y)$ and with estimate $\Vert \xi \zeta \Vert_{\mathfrak{W}} \leq \Vert \xi \Vert_{\mathfrak{X}} \Vert \zeta \Vert_{\mathfrak{W}}$. By definition, for $\psi_i\in \mathfrak{Y}(Y)$ and $\epsilon>0$, there are $\xi_i\in \mathfrak{X}(Y)$ and $\zeta_i\in \mathfrak{W}(Y)$ so that $\psi_i=\xi_i+\zeta_i$ and $ \Vert \xi_i \Vert_{\mathfrak{X}}+\Vert \zeta_i\Vert_{\mathfrak{W}}\leq \Vert \psi_i \Vert_{\mathfrak{Y}}+\epsilon$, for $i=1,2$.  Hence, 
\begin{align*}
\Vert \psi_1 \psi_2 \Vert_{\mathfrak{Y}} & \leq \Vert \xi_1 \xi_2 \Vert_{\mathfrak{Y}}+\Vert \xi_1 \zeta_2 \Vert_{\mathfrak{Y}}+ \Vert \xi_2 \zeta_1 \Vert_{\mathfrak{Y}}+ \Vert \zeta_1 \zeta_2 \Vert_{\mathfrak{Y}}\\
& \leq \Vert \xi_1 \xi_2 \Vert_{\mathfrak{X}}+\Vert \xi_1 \zeta_2 \Vert_{\mathfrak{W}}+ \Vert \xi_2 \zeta_1 \Vert_{\mathfrak{W}}+ \Vert \zeta_1 \zeta_2 \Vert_{\mathfrak{W}}\\
&\leq \Vert \xi_1 \Vert_{\mathfrak{X}} \Vert \xi_2 \Vert_{\mathfrak{X}}+\Vert \xi_1 \Vert_{\mathfrak{X}} \Vert \zeta_2 \Vert_{\mathfrak{W}}+ \Vert \xi_2 \Vert_{\mathfrak{X}} \Vert \zeta_1 \Vert_{\mathfrak{W}}+ \Vert \zeta_1 \Vert_{\mathfrak{W}} \Vert \zeta_2 \Vert_{\mathfrak{W}}\\
&\leq (\Vert \psi_1\Vert_{\mathfrak{Y}}+\epsilon)(\Vert \psi_2\Vert_{\mathfrak{Y}}+\epsilon).
\end{align*}
Sending $\epsilon$ to zero gives desired estimate.  This immediately implies $\mathfrak{Y}_0(Y)$ is an algebra.  Finally,  if $\psi_i=c_i \mathbf{1}+\zeta_i $ where $c_i\in\mathbb{R}$ and $\zeta_i\in\mathfrak{Y}_0(Y)$, then $\psi_1\psi_2=c_1 c_2 \mathbf{1}+c_1 \zeta_2 +c_2\zeta_1 + \zeta_1 \zeta_2\in \mathfrak{Y}(Y)$ as claimed.  

For the fifth and sixth items observe that
$$
C_c^0(Y\times \mathbb{S}^n)=\bigcup_{i=1}^\infty C_c^0((B_{i}\cap Y)\times\mathbb{S}^n).
$$
By the Stone-Weierstrass theorem, each $X_i=C_c^0((B_{i}\cap Y)\times \mathbb{S}^n )$ is separable with respect to the uniform topology and $X_i^\prime=C^\infty_c((B_{i}\cap Y)\times \mathbb{S}^n)$ is dense with respect to the uniform topology. As the $\mathfrak{W}$ norm is equivalent to the uniform norm on $X_i$, each $X_i$ is separable when equipped with the $\mathfrak{W}$ norm and $X_i^\prime$ is dense in $X_i$ with respect to this norm. As $\Vert \phi \Vert_{\mathfrak{Y}}\leq \Vert \phi \Vert_{\mathfrak{W}}$ for any $\phi \in X_i$ by the definition of the norms, one has immediately that $X_i$ with the $\mathfrak{Y}$ norm is separable and $X_i^\prime$ is dense with respect to this norm. As $C_c^0(Y\times \mathbb{S}^n)$ is the union of countably many $X_i$, it follows that $C_c^0(Y\times \mathbb{S}^n)$ with the $\mathfrak{Y}$ norm is separable. In a similar fashion as $C^\infty_c(Y\times \mathbb{S}^n)=\bigcup_{i=1}^\infty X_i^\prime$, this space is dense in $C_c^0(Y\times \mathbb{S}^n)$ with the $\mathfrak{Y}$ norm and hence also in $\mathfrak{Y}_0(Y)$. As $\mathfrak{Y}_0(Y)$ is the completion of a separable space it is separable. The separability of $\mathfrak{Y}(Y)$ is immediate.
\end{proof}

\subsection{Weighted estimates for elements of $\mathfrak{Y}$} \label{YEstSec}
Given $\mathfrak{Y}(Y)$ as in Proposition \ref{YSpaceProp}, let $\mathfrak{Y}^*(Y)$ denote its continuous dual space with the dual norm
$$
\Vert V \Vert_{\mathfrak{Y}^*} =\sup \set{ V[\psi]\colon \psi \in \mathfrak{Y}(Y), \Vert \psi \Vert_{\mathfrak{Y}}\leq 1}.
$$
We extend Proposition \ref{WeightRelProp} to elements of $\mathfrak{Y}$ to obtain the following:

\begin{prop} \label{YEstProp}
There is a constant $C_2=C_2(\Omega^\prime,\Gamma_-)>0$ so that if $\Gamma=\partial^* U$ for some $U\in\mathcal{C}(\Gamma_-^\prime,\Gamma_+^\prime)$, then, for any $\psi \in \mathfrak{Y}(\overline{\Omega^\prime})$, 
\begin{enumerate}
\item For any $R_2>R_1>R_1-\delta>\frac{1}{2}R_1>\bar{R}_1$,
$$
|E[\Gamma,\Gamma_-; \alpha_{R_1,R_2,\delta}\psi]| \leq C_2 \left(R_1^{-1}+|E[\Gamma,\Gamma_-;\alpha_{R_1,R_2,\delta}]|\right) \Vert \psi \Vert_{\mathfrak{Y}};
$$
\item For any $0<\delta<1$ and $R>\bar{R}_1$,
$$
|E[\Gamma,\Gamma_-; \phi_{R,\delta}\psi]| \leq C_2 \left(1+|E[\Gamma,\Gamma_-; \phi_{R,\delta}]|\right) \Vert \psi \Vert_{\mathfrak{Y}};
$$
\item If, in addition, $E_{rel}[\Gamma,\Gamma_-]<\infty$, then $E_{rel}[\Gamma, \Gamma_-; \psi]$ exists and 
$$
\left| E_{rel}[\Gamma, \Gamma_-; \psi] \right| \leq  C_2 \left(1+ \left|E_{rel}[\Gamma, \Gamma_-]\right|\right)\Vert \psi \Vert_{\mathfrak{Y}}.
$$
In particular, for such $\Gamma$ there is a well defined element $V_\Gamma\in \mathfrak{Y}^*(\overline{\Omega^\prime})$ given by
$$
V_{\Gamma}[\psi]=E_{rel}[\Gamma, \Gamma_-; \psi]
$$
that satisfies
$$
\Vert V_{\Gamma}\Vert_{\mathfrak{Y}^*}\leq C_2\left(1+ \left|E_{rel}[\Gamma, \Gamma_-]\right|\right).
$$
\end{enumerate}
Here $\bar{R}_1=\bar{R}_1(\Omega^\prime,\Gamma_-)$ is the constant given by Proposition \ref{WeightRelProp}.
\end{prop}

\begin{proof}
By linearity and Proposition \ref{WeightRelProp}, it suffices to show the claims for elements $\xi\in\mathfrak{W}(\overline{\Omega^\prime})$ in the $\mathfrak{W}$ norm. It is convenient to set
$$
\gamma(p)=(1+|\mathbf{x}(p)|)^{-n-1} e^{-\frac{|\mathbf{x}(p)|^2}{4}}>0.
$$
Observe that 
$$
\Vert \xi \Vert_{\mathfrak{W}}=\Vert\xi\gamma^{-1}\Vert_{C^0}.
$$
We compute, for a continuous compactly supported function $\phi\geq 0$,
\begin{equation} \label{ComputationEqn}
\begin{split}
\left| E[\Gamma, \Gamma_-; \phi\xi] \right| & =\left| \int_{\Gamma} \phi \xi e^{\frac{|\mathbf{x}|^2}{4}} \, d\mathcal{H}^n-\int_{\Gamma_-} \phi\xi e^{\frac{|\mathbf{x}|^2}{4}} \, d\mathcal{H}^n \right| \\
& \leq \Vert \xi \Vert_{\mathfrak{W}} \int_{\Gamma} \phi\gamma e^{\frac{|\mathbf{x}|^2}{4}} \, d\mathcal{H}^n+\Vert\xi\Vert_{\mathfrak{W}}\int_{\Gamma_-} \phi\gamma e^{\frac{|\mathbf{x}|^2}{4}} \, d\mathcal{H}^n\\
& =\Vert \xi \Vert_{\mathfrak{W}} E[\Gamma, \Gamma_-; \phi\gamma]+ 2\Vert \xi \Vert_{\mathfrak{W}} \int_{\Gamma_-} \phi\gamma e^{\frac{|\mathbf{x}|^2}{4}} \, d\mathcal{H}^n.
\end{split}
\end{equation}
One readily checks that there is a constant $K=K(\Gamma_-)>1$ so that
$$
\Vert\gamma\Vert_{\mathfrak{X}}+\sup_{R>0} (1+R)\int_{\Gamma_-\setminus B_R} \gamma e^{\frac{|\mathbf{x}|^2}{4}} \, d\mathcal{H}^n \leq K.
$$

Plugging $\phi=\alpha_{R_1,R_2,\delta}$ into \eqref{ComputationEqn}, it follows from Item (1) of Proposition \ref{WeightRelProp} that, for any $R_2>R_1>R_1-\delta>\frac{1}{2} R_1>\bar{R}_1$,
\begin{equation} \label{WeightAnnuliEqn}
|E[\Gamma,\Gamma_-; \alpha_{R_1,R_2,\delta}\xi]| \leq  C_2 \left(R_1^{-1}+|E[\Gamma,\Gamma_-;\alpha_{R_1,R_2,\delta}]|\right)\Vert\xi\Vert_{\mathfrak{W}}
\end{equation}
where $C_2=\bar{C}_1K+4K$. That is, the first item holds.

Similarly, plugging $\phi=\phi_{R,\delta}$ into \eqref{ComputationEqn} and appealing to Item (2) of Proposition \ref{WeightRelProp} give that, for any $\delta\in (0,1)$ and $R>\bar{R}_1$,
\begin{equation} \label{WeightBallEqn}
\left| E[\Gamma,\Gamma_-; \phi_{R,\delta}\xi]\right| \leq C_2 \left(1+|E[\Gamma,\Gamma_-; \phi_{R,\delta}]|\right)\Vert\xi \Vert_{\mathfrak{W}}.
\end{equation}
This proves the second item.

To see the last item, sending $\delta\to 0$ in \eqref{WeightAnnuliEqn}, it follows from the dominated convergence theorem that
$$
\left| E_{rel}[\Gamma,\Gamma_-; \xi; \bar{B}_{R_2}\setminus\bar{B}_{R_1}]\right| \leq C_2 \left(R_1^{-1}+\left|E_{rel}[\Gamma,\Gamma_-; \bar{B}_{R_2}\setminus\bar{B}_{R_1}]\right|\right)\Vert \xi \Vert_{\mathfrak{W}} .
$$
As $E_{rel}[\Gamma,\Gamma_-]<\infty$, this implies that $E_{rel}[\Gamma,\Gamma_-;\xi]$ exists. Finally, sending $\delta\to 0$ and $R\to\infty$ in \eqref{WeightBallEqn} gives that
$$
\left| E_{rel}[\Gamma,\Gamma_-; \xi] \right| \leq C_2 \left(1+\left|E_{rel}[\Gamma,\Gamma_-]\right|\right)\Vert\xi\Vert_{\mathfrak{W}}.
$$
This completes the proof.
\end{proof}

\subsection{Relative expander entropy of elements of $\mathfrak{Y}^*$} \label{RelEntropySec}
We extend the notion of relative expander entropy to elements of $\mathfrak{Y}^*(\overline{\Omega^\prime})$. Notice that $C^0_c((B_R\cap\overline{\Omega^\prime})\times \mathbb{S}^n)$ with the $C^0$ norm is continuously embedded in $\mathfrak{Y}(\overline{\Omega^\prime})$. Thinking of continuous functions on the Grassmanian bundle of $n$-planes $G_n(B_R)$ as even elements of $C^0_c((B_R\cap\overline{\Omega^\prime})\times\mathbb{S}^n)$, it follows from the Riesz representation theorem that, for any $V\in\mathfrak{Y}^*(\overline{\Omega^\prime})$, 
$$
V[\bar{B}_R]=V[\bar{B}_R\times\mathbb{S}^n]=\lim_{\delta\to 0} V[\phi_{R,\delta}].
$$
Furthermore, one can define
$$
\overline{E}_{rel}[V]=\limsup_{R\to \infty} V[\bar{B}_R]
$$
and
$$
\underline{E}_{rel}[V]=\liminf_{R\to \infty} V[\bar{B}_R].
$$
If $\overline{E}_{rel}[V]=\underline{E}_{rel}[V]$, then we set $E_{rel}[V]=\lim_{R\to \infty}V[\bar{B}_R]$.

Using Proposition \ref{YEstProp}, one may define, for $\Lambda\geq 0$,
$$\mathfrak{Y}_{\mathcal{C}}^*(\overline{\Omega^\prime}; \Lambda)=\set{V_\Gamma\colon \Gamma=\partial^* U, U\in \mathcal{C}(\Gamma_-^\prime, \Gamma_+^\prime), \left| E_{rel}[\Gamma, \Gamma_-] \right|\leq \Lambda}\subseteq \mathfrak{Y}^*(\overline{\Omega^\prime}).
$$
Let $\overline{\mathfrak{Y}_{\mathcal{C}}^*}(\overline{\Omega^\prime};\Lambda)$ be the closure of $\mathfrak{Y}_{\mathcal{C}}^*(\overline{\Omega^\prime};\Lambda)$ in the weak-* topology of $\mathfrak{Y}^*(\overline{\Omega^\prime})$. Observe that if $V_{\Gamma_i}\in \mathfrak{Y}_{\mathcal{C}}^*(\overline{\Omega^\prime};\Lambda)$ satisfy $V_{\Gamma_i}\to V $ in the weak-* topology, then
$$
\lim_{i\to \infty} E_{rel}[V_{\Gamma_i}]=\lim_{i\to \infty} V_{\Gamma_i}[\mathbf{1}]=V[\mathbf{1}].
$$
However, for $V\in \overline{\mathfrak{Y}_{\mathcal{C}}^*}(\overline{\Omega^\prime}; \Lambda)$ one, in principle, may have
$$
\underline{E}_{rel}[V]< \bar{E}_{rel}[V]\neq  V[\mathbf{1}].
$$
In fact, one has that

\begin{lem} \label{RelEntropyBndLem}
Given $V\in  \overline{\mathfrak{Y}_{\mathcal{C}}^*}(\overline{\Omega^\prime};\Lambda)$, the following holds:
\begin{enumerate}
\item There is a (weighted) varifold $V^E_+$ so that, for any (even) $\psi\in C^0_c(\overline{\Omega^\prime}\times\mathbb{S}^n)$,
$$
V[\psi]=V^E_+[\psi]-V_{\Gamma_-}^E[\psi]=V^E_+[\psi]-\int_{\Gamma_-} \psi (p, \mathbf{n}_{\Gamma_-}(p)) e^{\frac{|\mathbf{x}|^2}{4}}\, d\mathcal{H}^n;
$$
\item There is a constant $E_-=E_-(\Omega^\prime,\Gamma_-)\leq 0$ so that
$$
-\infty<E_-\leq \underline{E}_{rel}[V]= \overline{E}_{rel}[V]\leq V[\mathbf{1}]\leq \Lambda.
$$
In particular, $E_{rel}[V]$ exists and is finite. Moreover, for any $R_2>R_1>\bar{R}_0$,
$$
V[\bar{B}_{R_2}] \geq V[\bar{B}_{R_1}]-\bar{C}_0 R_1^{-1},
$$
where $\bar{R}_0=\bar{R}_0(\Omega^\prime, \Gamma_-)$ and $\bar{C}_0=\bar{C}_0(\Omega^\prime,\Gamma_-)$ are the constants given by Proposition \ref{RelEntropyProp}.
\end{enumerate}
\end{lem}

\begin{proof}
Let $V_{\Gamma_i}\in \mathfrak{Y}_{\mathcal{C}}^*(\overline{\Omega^\prime};\Lambda)$ satisfy $V_{\Gamma_i}\to V$ in the weak-* topology of $\mathfrak{Y}^*(\overline{\Omega^\prime})$. By Proposition \ref{RelEntropyProp}, for any $R>\bar{R}_0$,
$$
\Lambda\geq E_{rel}[\Gamma_i,\Gamma_-]\geq E_{rel}[\Gamma_i,\Gamma_-;\bar{B}_R]-\bar{C}_0 R^{-1}
$$
and so
$$
V_{\Gamma_i}^E[\bar{B}_R] =\int_{\Gamma_i\cap\bar{B}_R} e^{\frac{|\mathbf{x}|^2}{4}} \, d\mathcal{H}^n \leq C
$$
where $C$ depends on $\Omega^\prime,\Gamma_-,\Lambda$ and $R$. Thus, up to passing to a subsequence, $V^E_{\Gamma_i}\to V^E_+$ in the sense of varifolds so, for any (even) $\psi\in C^0_c(\overline{\Omega^\prime}\times\mathbb{S}^n)\subseteq\mathfrak{Y}(\overline{\Omega^\prime})$,
$$
V[\psi]=\lim_{i\to\infty} V_{\Gamma_i}[\psi]=\lim_{i\to\infty} V_{\Gamma_i}^E[\psi]-V_{\Gamma_-}^E[\psi]=V^E_+[\psi]-V_{\Gamma_-}^E[\psi]
$$ 
proving the first item.

To prove the inequalities in the second item, observe that the upper bound of $V[\mathbf{1}]$ is immediate from the weak-* convergence. To see the lower bound, appealing to Proposition \ref{RelEntropyProp} gives that, for any $R_2>R_1+\delta>R_1>\bar{R}_0$ and all $i$,
\begin{equation} \label{RelEntAnnEqn}
V_{\Gamma_i}[\phi_{R_2,\delta}] \geq V_{\Gamma_i}[\phi_{R_1,\delta}]-\bar{C}_0 R_1^{-1}.
\end{equation}
Sending $R_2\to\infty$ and invoking Proposition \ref{RelEntropyProp}, one sees 
$$
V_{\Gamma_i}[\mathbf{1}] \geq V_{\Gamma_i}[\phi_{R_1,\delta}]-\bar{C}_0 R_1^{-1}.
$$
It then follows from the weak-* convergence that 
$$
V[\mathbf{1}] \geq V[\phi_{R_1,\delta}]-\bar{C}_0 R_1^{-1}.
$$
As $\phi_{R,\delta}\leq \phi_{R, \delta^\prime}$ for $\delta^\prime<\delta$ and $\lim_{\delta\to 0} \phi_{R,\delta}=\mathbf{1}_{\bar{B}_R}$, the dominated convergence theorem combined with the previous inequality yields
$$
V[\mathbf{1}] \geq V[\bar{B}_{R_1}]-\bar{C}_0R_1^{-1}.
$$
Hence, taking the limsup of both sides as $R_1\to\infty$, gives $V[\mathbf{1}] \geq \overline{E}_{rel}[V]$ proving the claimed lower bound.

To prove the middle equality, it follows from \eqref{RelEntAnnEqn} and the weak-* convergence that
$$
V[\phi_{R_2,\delta}] \geq V[\phi_{R_1,\delta}]-\bar{C}_0 R_1^{-1}.
$$
Thus, sending $\delta\to 0$, by the dominated convergence theorem,
\begin{equation} \label{RelEntAnn2Eqn}
V[\bar{B}_{R_2}] \geq V[\bar{B}_{R_1}]-\bar{C}_0R_1^{-1}.
\end{equation}
This implies that
$$
\liminf_{R\to\infty} V[\bar{B}_{R}] \geq \limsup_{R\to\infty} V[\bar{B}_{R}]
$$
and so $\underline{E}_{rel}[V]=\overline{E}_{rel}[V]$.
 
It remains only to show the uniform lower bound for $E_{rel}[V]$. To achieve this, by the existence of $E_{rel}[V]$ and the first item, fixing $R_1=2\bar{R}_0$ and sending $R_2\to\infty$ in \eqref{RelEntAnn2Eqn} imply that
$$
E_{rel}[V]\geq V[\bar{B}_{2\bar{R}_0}]-\frac{1}{2}\bar{C}_0\bar{R}_0^{-1} \geq -V_{\Gamma_-}^E[\bar{B}_{2\bar{R}_0}]-\frac{1}{2} \bar{C}_0\bar{R}_0^{-1}.
$$
This immediately gives the uniform lower bound. 
\end{proof}

For convenience we will use the following notation: For $\Gamma=\partial^*U$ where $U\in \mathcal{C}(\Gamma_-^\prime, \Gamma_+^\prime)$ we consider the pair $(U, V_{\Gamma})\in\mathcal{C}(\Gamma_-^\prime, \Gamma_+^\prime)\times \mathfrak{Y}^*_{\mathcal{C}}(\overline{\Omega^\prime};\Lambda)$. For a sequence $(U_i, V_{\Gamma_i})$ we say $(U_i, V_{\Gamma_i})\to (U_\infty, V_\infty)$ provided $\mathbf{1}_{U_i}\to \mathbf{1}_{U_\infty}$ in the weak-* topology of $BV_{loc}$ and $V_{\Gamma_i}\to V_\infty$ in the weak-* topology of $\mathfrak{Y}^*(\overline{\Omega^\prime})$. By Lemma \ref{RelEntropyBndLem} and Proposition \ref{YEstProp} together with the Banach-Alaoglu theorem, we have

\begin{cor} \label{PairConvergeCor}
If $(U_i, V_{\Gamma_i})\in  \mathcal{C}(\Gamma_-^\prime, \Gamma_+^\prime)\times \mathfrak{Y}^*_{\mathcal{C}}(\overline{\Omega^\prime}; \Lambda)$, then up to passing to a subsequence there is a $(U_\infty, V_\infty)\in  \mathcal{C}(\Gamma_-^\prime, \Gamma_+^\prime)\times \overline{\mathfrak{Y}^*_{\mathcal{C}}}(\overline{\Omega^\prime};\Lambda)$ so that $(U_i, V_{\Gamma_i})\to (U_\infty, V_{\infty})$. For $\Gamma_\infty =\partial^* U_\infty$,  $V_{\Gamma_\infty}\leq  V_{\infty}$ in the sense of measures.
\end{cor}
\section{Action of flows of vector fields} \label{FlowSec}
In this section we study the action of flows of a suitable class of vector fields on elements of $\mathfrak{Y}^*$ and, in particular, prove the first variation formula. We continue to use the conventions of Section \ref{ConventionSec}. 

If $\Phi\colon \overline{\Omega^\prime}\to\overline{\Omega^\prime}$ is a local $C^1$ diffeomorphism, then the \emph{Jacobian of $\Phi$ with respect to the expander metric $g^E_{ij}=e^{\frac{|\mathbf{x}|^2}{2n}}\delta_{ij}$} is given by
$$
J^E\Phi(p,\mathbf{v})=J\Phi(p,\mathbf{v})e^{\frac{1}{4}(|\Phi(p)|^2-|\mathbf{x}(p)|^2)}
$$
where $J\Phi$ is the Jacobian of $\Phi$ with respect to the Euclidean metric. For a function $\psi$ on $\overline{\Omega^\prime}\times\mathbb{S}^n$, the \emph{pullback of $\psi$ under $\Phi$} is given by
$$
\Phi^\#\psi(p,\mathbf{v})=\psi(\Phi(p),\nabla_{\mathbf{v}}\Phi(p)).
$$

Suppose $\Phi\colon\overline{\Omega^\prime}\to \overline{\Omega^\prime}$ has the property that for all $\psi\in \mathfrak{Y}(\overline{\Omega^\prime})$, $\Phi^\# \psi J^E \Phi\in \mathfrak{Y}(\overline{\Omega^\prime})$. For such $\Phi$ and any $V\in\mathfrak{Y}^*(\overline{\Omega^\prime})$, we may define the \emph{pushforward of $V$ under $\Phi$} as follows: for all $\psi\in\mathfrak{Y}(\overline{\Omega^\prime})$,
$$
\Phi_\#V[\psi]=V\left[\Phi^\#\psi J^E\Phi\right].
$$
One readily checks that if $\Phi$ is $C^1$ and is fixed outside a compact set, i.e.,  equals the identity outside a compact set, then $\Phi_\#V$ is well defined. 

We now introduce the class of vector fields whose flow will not be fixed outside a compact set and that give suitable pushforwards of elements of $\mathfrak{Y}^*(\overline{\Omega^\prime})$. To that end let $\mathbf{X}_0$ be the vector field of \eqref{X0VectorFieldEqn} and $\mathcal{R}_1$ the radius given by Proposition \ref{ThickeningProp}. Let $\chi\colon\mathbb{R}^{n+1}\to [0,1]$ be a smooth cut-off so that $\spt(\chi)\subset\mathbb{R}^{n+1}\setminus\bar{B}_{\mathcal{R}_1}$ and $\chi=1$ in $\mathbb{R}^{n+1}\setminus \bar{B}_{\mathcal{R}_1+1}$. We then let
$$
\mathbf{Y}_0=\chi |\mathbf{x}|^{-2} \mathbf{X}_0= \chi|\mathbf{x}|^{-2}\left(\mathbf{x}-(\mathbf{x}\cdot\mathbf{N})\mathbf{N}\right).
$$
Define
$$
\mathcal{Y}_t(\overline{\Omega^\prime})=\set{\mathbf{Y}\in C^\infty_{loc}(\overline{\Omega^\prime}; \mathbb{R}^{n+1})\colon \Vert \mathbf{Y}\Vert_{\mathcal{Y}_t}<\infty}
$$
where
$$
\Vert\mathbf{Y}\Vert_{\mathcal{Y}_t}=\Vert (1+|\mathbf{x}|)^3 \mathbf{Y}\Vert_{C^0}+\sum_{l=1}^3 \Vert (1+|\mathbf{x}|)^2\nabla^l\mathbf{Y}\Vert_{C^0}.
$$
Let 
$$
\mathcal{Y}(\overline{\Omega^\prime})=\mathrm{span}_{\mathbb{R}}\set{\mathbf{Y}_0}\oplus\mathcal{Y}_t(\overline{\Omega^\prime})=\set{\mathbf{Y}=\alpha\mathbf{Y}_0+\mathbf{Y}_1\colon \alpha\in\mathbb{R}, \mathbf{Y}_1\in\mathcal{Y}_t(\overline{\Omega^\prime})}
$$
with the norm
$$
\Vert\mathbf{Y}\Vert_{\mathcal{Y}}=|\alpha|+\Vert\mathbf{Y}_1\Vert_{\mathcal{Y}_t}.
$$
Consider the convex cone
$$
\mathcal{Y}^-(\overline{\Omega^\prime})=\set{\mathbf{Y}\in\mathcal{Y}(\overline{\Omega^\prime})\colon \mathbf{Y}\cdot\mathbf{n}_{\partial\Omega^\prime}\leq 0}\subseteq\mathcal{Y}(\overline{\Omega^\prime})
$$
where $\mathbf{n}_{\partial\Omega^\prime}$ points out of $\Omega^\prime$. 

The main result of this section is the following:

\begin{prop} \label{FirstVarProp}
Suppose $\mathbf{Y}=\alpha\mathbf{Y}_0+\mathbf{Y}_1\in\mathcal{Y}^-(\overline{\Omega^\prime})$ for $\alpha\in\mathbb{R}$ and $\mathbf{Y}_1\in\mathcal{Y}_t(\overline{\Omega^\prime})$ that satisfies, for some constant $M_0>0$,
$$
\Vert\mathbf{Y}\Vert_{\mathcal{Y}}=|\alpha|+\Vert\mathbf{Y}_1\Vert_{\mathcal{Y}_t}\leq M_0.
$$ 
Let $\set{\Phi(t)}_{t\geq 0}$ be the family of diffeomorphisms in $\overline{\Omega^\prime}$ generated by $\mathbf{Y}$. Then $\Phi(t)(\overline{\Omega^\prime})\subseteq\overline{\Omega^\prime}$ and, for any $V\in\mathfrak{Y}^*(\overline{\Omega^\prime})$, the following is true:
\begin{enumerate}
\item The map $t\mapsto \Phi(t)_\# V$ is continuous in the weak-* topology of $\mathfrak{Y}^*(\overline{\Omega^\prime})$. Moreover, given $T\geq 0$ there is a constant $C_3=C_3(\Omega^\prime,\Gamma_-, M_0, T)>0$ so that, for all $0\leq t\leq T$,
$$
\Vert\Phi(t)_\# V\Vert_{\mathfrak{Y}^*} \leq C_3\Vert V\Vert_{\mathfrak{Y}^*}.
$$
\item The function $t\mapsto \Phi(t)_\# V[\mathbf{1}]$ is differentiable with
$$
\delta V[\mathbf{Y}]={\frac{d}{dt}\vline}_{t=0} \Phi(t)_\# V[\mathbf{1}]=V\left[\Div\mathbf{Y}-Q_{\nabla\mathbf{Y}}+\frac{\mathbf{x}}{2}\cdot\mathbf{Y}\right]
$$
where $Q_{\nabla\mathbf{Y}}(p,\mathbf{v})=\nabla_\mathbf{v}\mathbf{Y}(p)\cdot\mathbf{v}$.
\end{enumerate}
\end{prop}

To prove this proposition, we will need several auxiliary lemmas which are of a rather technical nature and are included in Appendix \ref{FlowAppendix}. 

\begin{proof}[Proof of Proposition \ref{FirstVarProp}]
By Lemma \ref{DecayVectorLem} and Corollaries \ref{WeightCor} and \ref{JacobiCor} with $a=0$ and $0\leq t\leq T$, one has
$$
J^E\Phi(t,p,\mathbf{v})=1+t\left(\Div\mathbf{Y}(p)-Q_{\nabla\mathbf{Y}}(p,\mathbf{v})+\frac{1}{2}\mathbf{x}(p)\cdot\mathbf{Y}(p)\right)+t^2Q(t,p,\mathbf{v})
$$
is an element of $\mathfrak{Y}(\overline{\Omega^\prime})$ and, for some constant $C=C(\Omega^\prime,\Gamma_-,M_0,T)>0$,
$$
\sup_{0\leq t\leq T} \left(\Vert J^E\Phi(t,\cdot,\cdot)\Vert_{\mathfrak{Y}}+\Vert Q(t, \cdot,\cdot)\Vert_{\mathfrak{Y}}\right) \leq C.
$$
Next, appealing to Lemmas \ref{FlowEstLem} and \ref{ComposeLem} gives that, for any $0\leq t\leq T$, if $\psi\in\mathfrak{Y}(\overline{\Omega^\prime})$, then so is $\Phi(t)^\#\psi$ and 
$$
\Vert\Phi(t)^\#\psi\Vert_{\mathfrak{Y}} \leq \tilde{C}_4 \Vert\psi\Vert_{\mathfrak{Y}}
$$
where $\tilde{C}_4$ depends on $\Omega^\prime,\Gamma_-,M_0$ and $T$. Hence, combining these estimates and appealing to Item (4) of Proposition \ref{YSpaceProp}, one has, for any $0\leq t \leq T$ and $\psi\in\mathfrak{Y}(\overline{\Omega^\prime})$,
$$
\Vert\Phi(t)^\#\psi J^E\Phi(t)\Vert_{\mathfrak{Y}} \leq 16 C\tilde{C}_4 \Vert\psi\Vert_{\mathfrak{Y}}
$$
and so
$$
\left| \Phi(t)_\# V[\psi] \right| \leq \Vert V\Vert_{\mathfrak{Y}^*} \Vert \Phi(t)^\# \psi J^E\Phi(t)\Vert_{\mathfrak{Y}}\leq 16 C \tilde{C}_4 \Vert V\Vert_{\mathfrak{Y}^*} \Vert\psi\Vert_{\mathfrak{Y}}.
$$
That is, 
\begin{equation} \label{NormPushforwardEqn}
\Vert\Phi(t)_\# V\Vert_{\mathfrak{Y}^*} \leq 16 C \tilde{C}_4 \Vert V\Vert_{\mathfrak{Y}^*},
\end{equation}
proving the desired estimate with $C_3=16C\tilde{C}_4$. In particular, as $T$ is arbitrary, $\Phi(t)_\# V\in\mathfrak{Y}(\overline{\Omega^\prime})$ for all $t\geq 0$.

Next, it is a standard exercise that for any (even) $\psi\in C^\infty_c(\overline{\Omega^\prime}\times\mathbb{S}^n)$ the map $t\mapsto \Phi(t)_\# V[\psi]$ is continuous. We further show the continuity can be extended to $\psi\in\mathfrak{Y}_0(\overline{\Omega^\prime})$. To see this, by Item (6) of Proposition \ref{YSpaceProp}, there is a sequence $\psi_j\in C^\infty_c(\overline{\Omega^\prime}\times\mathbb{S}^n)$ so that $\psi_j\to \psi$ in the $\mathfrak{Y}$ norm. Fix any $t_0\geq 0$. By \eqref{NormPushforwardEqn}, given $\epsilon>0$ there is a $j_0$ so that, for any $t\leq t_0+1$,
$$
\left| \Phi(t)_\#V[\psi_{j_0}-\psi] \right|<\frac{\epsilon}{2}.
$$
By what we have shown, there is a $\delta>0$ so that if $|t-t_0|<\delta$, then 
$$
\left| \Phi(t)_\# V[\psi_{j_0}]-\Phi(t_0)_\# V[\psi_{j_0}]\right|<\frac{\epsilon}{2}.
$$
Thus, by the triangle inequality, for any $|t-t_0|<\delta$,
\begin{align*}
\left| \Phi(t)_\# V[\psi]-\Phi(t_0)_\# V[\psi]\right| & \leq \left| \Phi(t)_\#V[\psi_{j_0}-\psi] \right|+\left| \Phi(t)_\# V[\psi_{j_0}]-\Phi(t_0)_\# V[\psi_{j_0}] \right| \\
&<\frac{\epsilon}{2}+\frac{\epsilon}{2}=\epsilon.
\end{align*}
This shows the map $t\mapsto \Phi(t)_\# V[\psi]$ is continuous at $t_0$. As $t_0$ is arbitrary, the claim follows immediately. It remains only to prove the map $t\mapsto \Phi(t)_\# V[\mathbf{1}]$ is differentiable. This readily follows from combining Corollaries \ref{WeightCor} and \ref{JacobiCor} with $a=0$, and the algebra property of space $\mathfrak{Y}(\overline{\Omega^\prime})$. In particular, 
$$
{\frac{d}{dt}\vline}_{t=0} \Phi(t)_\# V[\mathbf{1}]=V \left[\Div\mathbf{Y}-Q_{\nabla\mathbf{Y}}+\frac{\mathbf{x}}{2}\cdot\mathbf{Y}\right].
$$
This completes the proof.
\end{proof}

To conclude this section we record some properties about pushforwards of elements of $\mathfrak{Y}^*_\cC(\overline{\Omega^\prime};\Lambda)$ and continuous dependence of pushforwards on vector fields. 

\begin{lem} \label{RegPushforwardLem}
Fix a $\mathbf{Y}\in\mathcal{Y}^-(\overline{\Omega^\prime})$ and let $\set{\Phi(t)}_{t\geq 0}$ be the family of diffeomorphisms in $\overline{\Omega^\prime}$ generated by $\mathbf{Y}$. If $U\in\mathcal{C}(\Gamma_-^\prime,\Gamma_+^\prime)$ and $\Gamma=\partial^* U$ with $E_{rel}[\Gamma,\Gamma_-]<\infty$, then 
$$
\Phi(t)_\# V_{\Gamma}=V_{\Phi(t)(\Gamma)}-V_{\Phi(t)(\Gamma_-)}.
$$
\end{lem}

\begin{proof}
By Lemma \ref{FlowEstLem} and Corollaries \ref{WeightCor} and \ref{JacobiCor} 
$$
\Vert\Phi(t)^\# \phi_{R,\delta} J^E\Phi(t)\Vert_{\mathfrak{Y}} \leq C
$$
for some $C$ independent of $R$, and $\Phi(t)^\#\phi_{R,\delta}\to 1$ pointwise as $R\to \infty$. Thus, appealing to Proposition \ref{YEstProp}, the dominated convergence theorem and a change of variables, one readily computes 
\begin{align*}
\Phi(t)_\# V_{\Gamma}[\mathbf{1}] & =\lim_{R\to\infty} V_{\Gamma}[\Phi(t)^\#\phi_{R,\delta} J^E\Phi(t)] \\
& =\lim_{R\to\infty} \left(E[\Phi(t)(\Gamma),\Gamma_-;\phi_{R,\delta}]-E[\Phi(t)(\Gamma_-),\Gamma_-; \phi_{R,\delta}]\right).
\end{align*}
By Proposition \ref{RelEntropyProp}
$$
\lim_{R\to\infty} E[\Phi(t)(\Gamma),\Gamma_-;\phi_{R,\delta}]=E_{rel}[\Phi(t)(\Gamma),\Gamma_-]
$$ 
and
$$
\lim_{R\to \infty}E[\Phi(t)(\Gamma_-),\Gamma_-; \phi_{R,\delta}]=E_{rel}[\Phi(t)(\Gamma_-),\Gamma_-].
$$
Moreover, by Lemma \ref{FlowEstLem} and Proposition \ref{RelEntropyAnnuliProp} one has $E_{rel}[\Phi(t)(\Gamma_-),\Gamma_-]<\infty$ and so is $E_{rel}[\Phi(t)(\Gamma),\Gamma_-]$. Hence, combining these equalities gives 
$$
\Phi(t)_\# V_{\Gamma}[\mathbf{1}]=V_{\Phi(t)(\Gamma)}[\mathbf{1}]-V_{\Phi(t)(\Gamma_-)}[\mathbf{1}].
$$

If $\psi\in C^\infty_c(\overline{\Omega^\prime}\times\mathbb{S}^n)$, then the usual change of variable gives
$$
\Phi(t)_\# V_{\Gamma}[\psi]=V_{\Phi(t)(\Gamma)}[\psi]-V_{\Phi(t)(\Gamma_-)}[\psi].
$$
As, by Item (6) of Proposition \ref{YSpaceProp}, $C^\infty_c(\overline{\Omega^\prime}\times\mathbb{S}^n)$ is dense in $\mathfrak{Y}_0(\overline{\Omega^\prime})$, one has the above equality holds for all $\psi\in\mathfrak{Y}_0(\overline{\Omega^\prime})$. Hence, by the definition of $\mathfrak{Y}(\overline{\Omega^\prime})$ and linearity, the claim follows immediately.
\end{proof}

\begin{lem} \label{ContPushforwardLem}
Fix any radius $R_0>0$. The map
$$
\mathcal{Y}^-(\overline{\Omega^\prime})\times [0,\infty)\times\bar{B}_{R_0}^{\mathfrak{Y}^*}\ni (\mathbf{Y},t,V)\mapsto \Phi_{\mathbf{Y}}(t)_\# V\in\mathfrak{Y}^*(\overline{\Omega^\prime})
$$
is continuous. Here $\set{\Phi_{\mathbf{Y}}(t)}$ is the family of diffeomorphisms in $\overline{\Omega^\prime}$ generated by $\mathbf{Y}$, and $\mathfrak{Y}^*(\overline{\Omega^\prime})$ is endowed with the weak-* topology and $\bar{B}_{R_0}^{\mathfrak{Y}^*}$ is a subspace of $\mathfrak{Y}^*(\overline{\Omega^\prime})$.
\end{lem}

\begin{proof}
Fix $\mathbf{Y}\in\mathcal{Y}^-(\overline{\Omega^\prime}), t \geq 0$ and $V\in\bar{B}_{R_0}^{\mathfrak{Y}^*}$. First, given $\psi\in C^\infty_c(\overline{\Omega^\prime}\times\mathbb{S}^n)$, by Lemma \ref{FlowEstLem}, $\Phi_{\mathbf{Y}^\prime}(t^\prime)^\# \psi$ is supported in a fixed compact set as long as $(\mathbf{Y}^\prime,t^\prime)$ is sufficiently close to $(\mathbf{Y},t)$. Thus, it is a standard exercise to check that
$$
\lim_{(\mathbf{Y}^\prime,t^\prime,V^\prime)\to (\mathbf{Y},t,V)} \Phi_{\mathbf{Y}^\prime}(t^\prime)_\# V^\prime [\psi]=\Phi_{\mathbf{Y}}(t)_\# V[\psi].
$$

Next we show that the above limit still holds true for functions in $\mathfrak{Y}_0(\overline{\Omega^\prime})$. Endow $\mathcal{Y}^-(\overline{\Omega^\prime})\subset\mathcal{Y}(\overline{\Omega^\prime})$ with the subspace topology. It follows from Proposition \ref{FirstVarProp} that, for any $\mathbf{Y}^\prime\in B^{\mathcal{Y}^-}_1(\mathbf{Y}), 0\leq t^\prime \leq t+1$ and $V^\prime\in\bar{B}_{R_0}^{\mathfrak{Y}^*}$, 
\begin{equation} \label{PushforwardBndEqn}
\Vert\Phi_{\mathbf{Y}^\prime}(t^\prime)_\# V^\prime \Vert_{\mathfrak{Y}^*} \leq C_3 R_0.
\end{equation}
Now take any $\psi\in\mathfrak{Y}_0(\overline{\Omega^\prime})$. By Item (6) of Proposition \ref{YSpaceProp}, there is a sequence $\psi_j\in C^\infty_c(\overline{\Omega^\prime}\times\mathbb{S}^n)$ so that $\psi_j\to\psi$ in the $\mathfrak{Y}$ norm. Thus given $\epsilon>0$ there is a $j_0\in\mathbb{N}$ so that, for any $\mathbf{Y}^\prime\in B^{\mathcal{Y}^-}_1(\mathbf{Y}), 0\leq t^\prime \leq t+1$ and $V^\prime \in\bar{B}_{R_0}^{\mathfrak{Y}^*}$,
$$
\left| \Phi_{\mathbf{Y}^\prime}(t^\prime)_\# V^\prime [\psi_{j_0}-\psi]\right| \leq C_3 R_0 \Vert\psi_{j_0}-\psi\Vert_{\mathfrak{Y}}<\frac{\epsilon}{3}
$$
and, in particular, so
$$
\left| \Phi_{\mathbf{Y}}(t)_\# V[\psi_{j_0}-\psi]\right| <\frac{\epsilon}{3}.
$$
By general topology $\bar{B}^{\mathfrak{Y}^*}_{R_0}$ with the weak-* topology is metrizable and denote by $\mathcal{D}$ a choice of such metric. By the previous discussion, there is a $\rho_0>0$ so that if 
$$
\Vert\mathbf{Y}^\prime-\mathbf{Y}\Vert_{\mathcal{Y}}+|t^\prime-t|+\mathcal{D}(V^\prime,V)<\rho_0,
$$
then 
$$
\left| \Phi_{\mathbf{Y}^\prime}(t^\prime)_\# V^\prime [\psi_{j_0}]-\Phi_{\mathbf{Y}}(t)_\# V[\psi_{j_0}]\right|<\frac{\epsilon}{3}.
$$
Thus, combining above estimates and applying the triangle inequality give that
\begin{align*}
\left| \Phi_{\mathbf{Y}^\prime}(t^\prime)_\# V^\prime [\psi]-\Phi_{\mathbf{Y}}(t)_\# V[\psi] \right| & \leq \left| \Phi_{\mathbf{Y}^\prime}(t^\prime)_\# V^\prime [\psi_{j_0}]-\Phi_{\mathbf{Y}}(t)_\# V[\psi_{j_0}]\right| \\
&+\left| \Phi_{\mathbf{Y}^\prime}(t^\prime)_\# V^\prime [\psi_{j_0}-\psi]\right| +\left| \Phi_{\mathbf{Y}}(t)_\# V[\psi_{j_0}-\psi]\right| \\
& \leq \frac{\epsilon}{3}+\frac{\epsilon}{3}+\frac{\epsilon}{3}=\epsilon.
\end{align*}
Hence we have shown the claim.

To conclude the first item, it remains only to prove that $\Phi_{\mathbf{Y}^\prime}(t^\prime)_\# V^\prime [\mathbf{1}]\to \Phi_{\mathbf{Y}}(t)_\# V[\mathbf{1}]$ as $(\mathbf{Y}^\prime,t^\prime,V^\prime)\to (\mathbf{Y},t,V)$. To see this, first appealing to Lemma \ref{DecayVectorLem} and Corollaries \ref{WeightCor} and \ref{JacobiCor}, one gets that, for any $\mathbf{Y}^\prime\in B^{\mathcal{Y}^-}_1(\mathbf{Y})$ and $0\leq t^\prime \leq t+1$, 
$$
J^E\Phi_{\mathbf{Y}^\prime}(t^\prime,p,\mathbf{v})=c(\alpha_{\mathbf{Y}^\prime},t^\prime)+P_{\mathbf{Y}^\prime}(t^\prime,p,\mathbf{v})
$$
where 
$$
c(\alpha_{\mathbf{Y}^\prime},t^\prime)=1+\frac{1}{2}t^\prime\alpha_{\mathbf{Y}^\prime}+\frac{1}{4}(t^\prime \alpha_{\mathbf{Y}^\prime})^2\int_0^1 e^{\frac{1}{2}\alpha_{\mathbf{Y}^\prime} t^\prime\rho} (1-\rho) \, d\rho
$$
and there is a constant $C=C(\Omega^\prime,\Gamma_-,t)>0$ so that, for any $R>0$,
\begin{equation} \label{QuadraticEqn}
\sup_{0\leq t^\prime \leq t+1}\Vert P_{\mathbf{Y}^\prime}(t^\prime,\cdot,\cdot)\Vert_{\mathfrak{Y}(\overline{\Omega^\prime}\setminus\bar{B}_R)} \leq \frac{C}{R+1}.
\end{equation}
As $\alpha_{\mathbf{Y}^\prime}$ continuously depends on $\mathbf{Y}^\prime$,
\begin{equation} \label{ContConstEqn}
\lim_{(\mathbf{Y}^\prime,t^\prime,V^\prime)\to (\mathbf{Y},t,V)} V^\prime [c(\alpha_{\mathbf{Y}^\prime},t^\prime)\mathbf{1}]= V[c(\alpha_{\mathbf{Y}},t)\mathbf{1}].
\end{equation}
Next, invoking \eqref{QuadraticEqn}, given $\epsilon>0$ there is a radius $R_\epsilon>1$ so that if $\mathbf{Y}^\prime \in B^{\mathcal{Y}^-}_1(\mathbf{Y}), 0\leq t^\prime \leq t+1$ and $V^\prime \in\bar{B}_{R_0}^{\mathfrak{Y}^*}$, then, for all $R>R_\epsilon$,
$$
V^\prime [(1-\phi_{R,1}) P_{\mathbf{Y}^\prime}(t^\prime,\cdot,\cdot)]<\epsilon.
$$
As one readily checks that
$$
\lim_{(\mathbf{Y}^\prime,t^\prime,V^\prime)\to (\mathbf{Y},t,V)} V^\prime [\phi_{R,1} P_{\mathbf{Y}^\prime}(t^\prime,\cdot,\cdot)] = V[\phi_{R,1} P_{\mathbf{Y}}(t,\cdot,\cdot)]
$$
combining this with the previous estimate gives
$$
\lim_{(\mathbf{Y}^\prime,t^\prime,V^\prime)\to (\mathbf{Y},t,V)} V^\prime [P_{\mathbf{Y}^\prime}(t^\prime,\cdot,\cdot)]=V[P_{\mathbf{Y}}(t,\cdot,\cdot)].
$$
Hence the claim follows by combining this and \eqref{ContConstEqn}.
\end{proof}

\section{Stationarity of relative expander entropy and its properties} \label{StationarySec}
In this section we introduce an appropriate notion of stationarity for the relative expander entropy. In particular, this notion admits some of the large scale deformations from the previous section as valid variations and not just those that are fixed outside a compact set.  Continue to use the conventions of Section \ref{ConventionSec}. 

\subsection{Modified pushforwards} \label{PlusPushforwardSec}
Let $\Phi$ be a $C^1$ diffeomorphism of $\overline{\Omega^\prime}$ that is fixed outside a compact set. For $V\in\mathfrak{Y}^*(\overline{\Omega^\prime})$ define a modified pushforward of $V$ under $\Phi$ by 
$$
\Phi_\#^+ V=\Phi_\# V+V_{\Phi(\Gamma_-)}\in\mathfrak{Y}^*(\overline{\Omega^\prime}).
$$
As in the previous section we will extend this to an appropriate class of diffeomorphisms that do not fix things outside a compact set. 

If $\mathbf{Y}\in\mathcal{Y}^-(\overline{\Omega^\prime})$ and $\set{\Phi(t)}_{t\geq 0}$ is the family of diffeomorphisms in $\overline{\Omega^\prime}$ generated by $\mathbf{Y}$, then Lemma \ref{FlowEstLem} and Proposition \ref{RelEntropyAnnuliProp} imply $E_{rel}[\Phi(t)(\Gamma_-),\Gamma_-]<\infty$. Thus, by Propositions \ref{YEstProp} and  \ref{FirstVarProp}, $\Phi(t)_\#^+ V$ is well defined. The advantage of this notion is that, by Lemma \ref{RegPushforwardLem}, if $V\in\overline{\mathfrak{Y}^*_\cC}(\overline{\Omega^\prime};\Lambda)$, then $\Phi(t)_\#^+ V\in\overline{\mathfrak{Y}^*_\cC}(\overline{\Omega^\prime};\Lambda^\prime)$ for some $\Lambda^\prime>0$.

We record the following property about continuous dependence of modified pushforwards on vector fields.

\begin{prop} \label{ContPlusPushforwardProp}
Fix any radius $R_0>0$. We have
\begin{enumerate}
\item The map
$$
\mathcal{Y}^-(\overline{\Omega^\prime})\times [0,\infty)\times\bar{B}_{R_0}^{\mathfrak{Y}^*} \ni (\mathbf{Y},t,V)\mapsto \Phi_{\mathbf{Y}}(t)_\#^+ V \in\mathfrak{Y}^*(\overline{\Omega^\prime})
$$
is continuous;
\item The map
$$
\mathcal{Y}^-(\overline{\Omega^\prime})\times\mathcal{Y}^-(\overline{\Omega^\prime})\times [0,\infty)\times\bar{B}_{R_0}^{\mathfrak{Y}^*}\ni (\mathbf{Z},\mathbf{Y},t,V)\mapsto \delta (\Phi^+_{\mathbf{Y}}(t)_\# V)[\mathbf{Z}]
$$
is continuous. 
\end{enumerate}
Here $\set{\Phi_{\mathbf{Y}}(t)}_{t\geq 0}$ is the family of diffeomorphisms in $\overline{\Omega^\prime}$ generated by $\mathbf{Y}$, and $\mathfrak{Y}^*(\overline{\Omega^\prime})$ is endowed with the weak-* topology and $\bar{B}_{R_0}^{\mathfrak{Y}^*}$ is a subspace of $\mathfrak{Y}^*(\overline{\Omega^\prime})$.
\end{prop}

\begin{proof}
To see the first item, by Lemma \ref{ContPushforwardLem} it suffices to show the map
$$
\mathcal{Y}^-(\overline{\Omega^\prime})\times [0,\infty)\ni (\mathbf{Y},t) \mapsto V_{\Gamma^{\mathbf{Y}}_t}\in\mathfrak{Y}^*(\overline{\Omega^\prime})
$$
is continuous, where $\Gamma^{\mathbf{Y}}_t=\Phi_{\mathbf{Y}}(t)(\Gamma_-)$. To see this, fix any $\mathbf{Y}\in\mathcal{Y}^-(\overline{\Omega^\prime})$ and $t\geq 0$. Appealing to Lemma \ref{FlowEstLem} and Proposition \ref{RelEntropyAnnuliProp}, one has that if $\mathbf{Y}^\prime\in B_1^{\mathcal{Y}^-}(\mathbf{Y})$ and $0\leq t^\prime \leq t+1$, then for all $R>\bar{\mathcal{R}}_1$
\begin{equation} \label{ExteriorRelEntEqn}
\left| E_{rel}[\Gamma^{\mathbf{Y}^\prime}_{t^\prime},\Gamma_-; \mathbb{R}^{n+1}\setminus\bar{B}_R] \right| \leq \bar{K}_1 R^{-2}
\end{equation}
where $\bar{\mathcal{R}}_1$ and $\bar{K}_1$ both depends on $\Omega^\prime,\Gamma_-,\mathbf{Y}$ and $t$. Thus, by  Proposition \ref{YEstProp}, one sees that if $\psi\in\mathfrak{Y}(\overline{\Omega^\prime})$, then for all $R>\max\set{2\bar{R}_1,\bar{\mathcal{R}}_1}>1$
\begin{align*}
\left|V_{\Gamma^{\mathbf{Y}^\prime}_{t^\prime}}[(1-\phi_{2R,1})\psi]\right|
 & \leq C_2 \left(R^{-1}+\left|E_{rel}[\Gamma^{\mathbf{Y}^\prime}_{t^\prime},\Gamma_-; \mathbb{R}^{n+1}\setminus\bar{B}_{R}]\right|\right) \Vert (1-\phi_{2R,1}) \psi\Vert_{\mathfrak{Y}} \\
 & \leq 2C_2(1+\bar{K}_1) R^{-1} \Vert \psi \Vert_{\mathfrak{Y}}.
\end{align*}
Hence, as one readily checks
$$
\lim_{(\mathbf{Y}^\prime,t^\prime)\to (\mathbf{Y},t)} V_{\Gamma^{\mathbf{Y}^\prime}_{t^\prime}}[\phi_{2R,1}\psi]=V_{\Gamma^{\mathbf{Y}}_t}[\phi_{2R,1}\psi],
$$
combining these gives
$$
\lim_{(\mathbf{Y}^\prime,t^\prime)\to (\mathbf{Y},t)} V_{\Gamma^{\mathbf{Y}^\prime}_{t^\prime}}[\psi]= V_{\Gamma^{\mathbf{Y}}_t}[\psi].
$$
As $\psi$ is arbitrary, the claim follows immediately.

To prove the second, write
\begin{multline*}
\delta (\Phi^+_{\mathbf{Y}^\prime}(t^\prime)_\# V^\prime)[\mathbf{Z}^\prime]-\delta (\Phi^+_{\mathbf{Y}}(t)_\# V)[\mathbf{Z}] 
= \delta (\Phi^+_{\mathbf{Y}^\prime}(t^\prime)_\# V^\prime)[\mathbf{Z}^\prime-\mathbf{Z}] \\
+\delta (\Phi^+_{\mathbf{Y}^\prime}(t^\prime)_\# V^\prime-\Phi^+_{\mathbf{Y}}(t)_\# V)[\mathbf{Z}].
\end{multline*}
As $\mathbf{Z}^\prime\to\mathbf{Z}$ in the $\mathcal{Y}$ norm, it is readily checked that
$$
\Div (\mathbf{Z}^\prime-\mathbf{Z})-Q_{\nabla (\mathbf{Z}^\prime-\mathbf{Z})}+\frac{\mathbf{x}}{2}\cdot (\mathbf{Z}^\prime-\mathbf{Z})\to 0
$$
in the $\mathfrak{Y}$ norm. By \eqref{ExteriorRelEntEqn} and Propositions \ref{YEstProp} and \ref{FirstVarProp} one has, for any $\mathbf{Y}^\prime\in B^{\mathcal{Y}^-}_1(\mathbf{Y}),0\leq t^\prime \leq t+1$ and $V^\prime\in\bar{B}^{\mathfrak{Y}^*}_{R_0}$, 
$$
\Vert\Phi_{\mathbf{Y}^\prime}(t)_\#^+ V^\prime \Vert_{\mathfrak{Y}^*} \leq C
$$
where $C=C(\Omega^\prime, \Gamma_-,\mathbf{Y},R_0, t)$. Thus it follows that 
$$
\lim_{(\mathbf{Z}^\prime,\mathbf{Y}^\prime,t^\prime,V^\prime)\to (\mathbf{Z},\mathbf{Y},t,V)} \delta (\Phi^+_{\mathbf{Y}^\prime}(t^\prime)_\# V^\prime)[\mathbf{Z}^\prime-\mathbf{Z}]=0.
$$
Invoking the first item and Proposition \ref{FirstVarProp} gives
$$
\lim_{(\mathbf{Y}^\prime,t^\prime,V^\prime)\to (\mathbf{Y},t,V)}\delta (\Phi^+_{\mathbf{Y}^\prime}(t^\prime)_\# V^\prime-\Phi^+_{\mathbf{Y}}(t)_\# V)[\mathbf{Z}]=0.
$$
Hence, combining these limits gives
$$
\lim_{(\mathbf{Z}^\prime,\mathbf{Y}^\prime,t^\prime,V^\prime)\to (\mathbf{Z},\mathbf{Y},t,V)}\delta (\Phi^+_{\mathbf{Y}^\prime}(t^\prime)_\# V^\prime)[\mathbf{Z}^\prime]=\delta (\Phi^+_{\mathbf{Y}}(t)_\# V)[\mathbf{Z}].
$$
This completes the proof.
\end{proof}

\subsection{$E_{rel}$-minimizing to first order} \label{RelMinimizingSec}
Let $\mathcal{Y}_c(\overline{\Omega^\prime})$ be the subset of elements of $\mathcal{Y}_t(\overline{\Omega^\prime})$ with compact support and let 
$$
\mathcal{Y}^-_0(\overline{\Omega^\prime})=\set{\alpha \mathbf{Y}_0+\mathbf{Y}_1\in\mathcal{Y}^-(\overline{\Omega^\prime})\colon  \mathbf{Y}_1\in \mathcal{Y}_c(\overline{\Omega^\prime})}\subset\mathcal{Y}^-(\overline{\Omega^\prime}).
$$ 
If $\mathbf{Y}\in\mathcal{Y}^-_0(\overline{\Omega^\prime})$ and $\set{\Phi(t)}_{t\geq 0}$ is the family of diffeomorphisms in $\overline{\Omega^\prime}$ generated by $\mathbf{Y}$, then, by the fact that $\mathbf{Y}_0$ is tangent to $\Gamma_-$,  one has $\Phi(t)(\Gamma_-)$ and $\Gamma_-$ agreeing outside a compact set and so $V_{\Phi(t)(\Gamma_-)}$ may be thought of as a measure with compact support. Thus, by Proposition \ref{FirstVarProp}, given $V\in\mathfrak{Y}^*(\overline{\Omega^\prime})$, $\Phi(t)^+_{\#} V$ is differentiable at $t=0$ and so, as $\Gamma_-$ is a self-expander, we can define
\begin{equation} \label{PlusFirstVarEqn}
\delta^+ V[\mathbf{Y}]={\frac{d}{dt}\vline}_{t=0} \Phi(t)_\#^+ V[\mathbf{1}]=\delta V[\mathbf{Y}]=V\left[\Div\mathbf{Y}-Q_{\nabla\mathbf{Y}}+\frac{\mathbf{x}}{2}\cdot\mathbf{Y}\right].
\end{equation}
An element $V\in\mathfrak{Y}^*(\overline{\Omega^\prime})$ is \emph{$E_{rel}$-minimizing to first order in $\overline{\Omega^\prime}$} if $\delta^+ V[\mathbf{Y}]\geq 0$ for all $\mathbf{Y}\in\mathcal{Y}_0^-(\overline{\Omega^\prime})$. 

If $V\in \overline{\mathfrak{Y}_\cC^*}(\overline{\Omega^\prime};\Lambda)$ has decomposition 
$$
V=V_+^E-V_{\Gamma_-}^E
$$
for a weighted varifold $V_+^E$, then, as $\Gamma_-$ is a self-expander and so $E$-stationary, $V$ being $E_{rel}$-minimizing to first order in $\overline{\Omega^\prime}$ means that $V_+=e^{-\frac{|\mathbf{x}|^2}{4}}V_+^E$ is \emph{$E$-minimizing to first order in $\overline{\Omega^\prime}$}. That is, for $\mathbf{Y}\in \mathcal{Y}_c^-(\overline{\Omega^\prime})$
$$
\delta_E V_+[\mathbf{Y}]\geq 0
$$
where
$$
\delta_E V_+[\mathbf{Y}]={\frac{d}{dt}\vline}_{t=0} \int e^{\frac{|\mathbf{x}|^2}{4}} \, d\Phi(t)_\# V_+=\int \left(\Div_S\mathbf{Y}+ \frac{\mathbf{x}}{2}\cdot \mathbf{Y}\right) e^{\frac{|\mathbf{x}|^2}{4}} \, dV_+
$$
is the usual first variation formula for the $E$-functional. 

Observe that if $\spt(\mathbf{Y})\subset \Omega^\prime$, then both $\mathbf{Y}$ and $-\mathbf{Y}$ lie in $\mathcal{Y}_0^-(\overline{\Omega^\prime})$. Hence, if $V$ is $E_{rel}$-minimizing to first order in $\overline{\Omega^\prime}$, then $\delta^+ V[\mathbf{Y}]= 0$. In particular, $V_+$ is $E$-stationary in $\Omega^\prime$.  

We now summarize some properties of elements of $\overline{\mathfrak{Y}^*_\cC}(\overline{\Omega^\prime};\Lambda)$ that are $E_{rel}$-minimizing to first order in $\overline{\Omega^\prime}$. 

\begin{prop}\label{StationaryProp}
Fix any positive number $\Lambda$ and let $V\in \overline{\mathfrak{Y}_\cC^*}(\overline{\Omega^\prime};\Lambda)$ have decomposition $V=e^{\frac{|\mathbf{x}|^2}{4}}V_+-V_{\Gamma_-}^E$ for some varifold $V_+$. If $V$ is $E_{rel}$-minimizing to first order in $\overline{\Omega^\prime}$, then  
\begin{enumerate}
\item $V[\mathbf{1}]=E_{rel}[V]$;
\item The support of $V_+$ lies in $\tilde{\Omega}$, the closed region between $\Gamma_-$ and $\Gamma_+$;
\item $V_+$ is an $E$-stationary varifold in $\Omega^\prime$;
\item If $V_+$ is integer rectifiable, then there is an $R>0$ so $V_+\lfloor (\Real^{n+1}\backslash \bar{B}_R)=\mathcal{H}^n\lfloor\Gamma$ where $\Gamma$ is a self-expanding end that is $C^2$-asymptotic to $\mathcal{C}(\Gamma_-)$.
\end{enumerate}
\end{prop}

To prove the proposition, we will need the following lemma.

\begin{lem} \label{EquicontLem}
Given $\epsilon>0$ and $V\in\overline{\mathfrak{Y}^*_\mathcal{C}}(\overline{\Omega^\prime};\Lambda)$, there is a radius $R_\epsilon=R_\epsilon (\Omega^\prime,\Gamma_-,V)>1$ so that if $\psi\in\mathfrak{Y}(\overline{\Omega^\prime})$, then, for all $R_2>R_1 \geq R_\epsilon$, 
$$
\left| V[\mathbf{1}_{\bar{B}_{R_2}}\psi]-V[\mathbf{1}_{\bar{B}_{R_1}}\psi]\right|<\epsilon\Vert\psi\Vert_{\mathfrak{Y}}.
$$
Here $\mathbf{1}_Y$ is the indicator function of a set $Y$.
\end{lem}

\begin{proof}
Let $V_{\Gamma_i}\in\mathfrak{Y}_\cC^*(\overline{\Omega^\prime};\Lambda)$ satisfy $V_{\Gamma_i}\to V$ in the weak-* topology. By Item (1) of Proposition \ref{YEstProp}, if $\psi\in\mathfrak{Y}(\overline{\Omega^\prime})$, then, for all $0<\delta<1$ and $R_2>R_1+\delta>R_1\geq 2\bar{R}_1$ and for all $i$,
$$
\left| V_{\Gamma_i}[\phi_{R_2,\delta}\psi]-V_{\Gamma_i}[\phi_{R_1,\delta}\psi] \right| \leq C_2 \left(\left| V_{\Gamma_i}[\phi_{R_2,\delta}]-V_{\Gamma_i}[\phi_{R_1,\delta}] \right|+R_1^{-1}\right) \Vert\psi\Vert_{\mathfrak{Y}}
$$
where $\bar{R}_1$ and $C_2$ both depend on $\Omega^\prime$ and $\Gamma_-$. Sending $i\to\infty$, the weak-* convergence gives
$$
\left| V[\phi_{R_2,\delta}\psi]-V[\phi_{R_1,\delta}\psi] \right| \leq C_2 \left( \left| V[\phi_{R_2,\delta}]-V[\phi_{R_1,\delta}] \right|+R_1^{-1}\right)\Vert \psi \Vert_{\mathfrak{Y}}.
$$
Next, letting $\delta\to 0$, the dominated convergence theorem implies
$$
\left| V[\mathbf{1}_{\bar{B}_{R_2}}\psi]-V[\mathbf{1}_{\bar{B}_{R_1}}\psi] \right| \leq C_2 \left( \left| V[\bar{B}_{R_2}]-V[\bar{B}_{R_1}] \right|+R_1^{-1}\right) \Vert\psi\Vert_{\mathfrak{Y}}.
$$
By Lemma \ref{RelEntropyBndLem}, there is a radius $R^\prime_\epsilon=R^\prime_\epsilon(C_2, V)$ so that, for any $R_2>R_1\geq R^\prime_\epsilon$,
$$
C_2 \left| V[\bar{B}_{R_2}]-V[\bar{B}_{R_1}] \right|+C_2R_1^{-1}<\epsilon.
$$
Hence, combining these estimates yields, for any $R_2>R_1\geq \max\set{2\bar{R}_1, R^\prime_\epsilon}$,
$$
\left| V[\mathbf{1}_{\bar{B}_{R_2}}\psi]-V[\mathbf{1}_{\bar{B}_{R_1}}\psi] \right|<\epsilon \Vert \psi \Vert_{\mathfrak{Y}}.
$$
The claim follows with $R_\epsilon=\max\set{2\bar{R}_1, R^\prime_\epsilon}$ which depends on $\Omega^\prime,\Gamma_-$ and $V$.
\end{proof}

We are now ready to prove Proposition \ref{StationaryProp}.

\begin{proof}[Proof of Proposition \ref{StationaryProp}]
We first prove Item (1).  Let $\mathbf{Y}_0$ be the vector field from Section \ref{FlowSec}. Consider the cutoff function
$$
\phi_R=\left(1-R^{-2}|\mathbf{x}|^2\right)^4_+ \in C^3_c(\mathbb{R}^{n+1}).
$$
As $V$ is $E_{rel}$-minimizing to first order in $\overline{\Omega^\prime}$ and $-(1-\phi_R)\mathbf{Y}_0\in\mathcal{Y}_0^-(\overline{\Omega^\prime})$, it follows from Proposition \ref{FirstVarProp} that 
$$
V\left[\Div \left((1-\phi_R)\mathbf{Y}_0\right)-Q_{\nabla((1-\phi_R)\mathbf{Y}_0)}+\frac{\mathbf{x}}{2}\cdot\left((1-\phi_R)\mathbf{Y}_0\right)\right]\leq 0.
$$
One appeals to Lemma \ref{DecayVectorLem} to check that as $R\to\infty$
$$
\Div \left((1-\phi_R)\mathbf{Y}_0\right)-Q_{\nabla((1-\phi_R)\mathbf{Y}_0)}+\frac{\mathbf{x}}{2}\cdot\left((1-\phi_R)\mathbf{Y}_0\right)-\frac{1}{2}(1-\phi_R)\to 0
$$
in the $\mathfrak{Y}$ norm. Thus it follows that
\begin{equation} \label{UpperBndEqn}
V[\mathbf{1}]\leq \lim_{R\to\infty} V[\phi_R].
\end{equation}

As $\Vert\phi_R\Vert_{\mathfrak{Y}} \leq 1+8R^{-1}$, it follows from Lemma \ref{EquicontLem} that given $\epsilon>0$ there is a radius $R_\epsilon$ so that, for any $R^\prime>R_\epsilon$,
\begin{equation} \label{EquiDecayEqn}
\left| V[\phi_R]-V[\mathbf{1}_{\bar{B}_{R^\prime}}\phi_R]\right|<\epsilon (1+8R^{-1}).
\end{equation}
As $|\phi_R|\leq 1$ and $\phi_R\to 1$ pointwise as $R\to\infty$, it follows from the dominated convergence theorem that 
$$
\lim_{R\to\infty} V[\mathbf{1}_{\bar{B}_{R^\prime}}\phi_R]=V[\bar{B}_{R^\prime}].
$$
Thus, taking the limit of both sides of \eqref{EquiDecayEqn} as $R\to\infty$, gives
$$
\left|\lim_{R\to\infty} V[\phi_R]-V[\bar{B}_{R^\prime}]\right| \leq \epsilon.
$$
Hence, letting $R^\prime\to\infty$ and appealing to Lemma \ref{RelEntropyBndLem}, gives
$$
\lim_{R\to\infty} V[\phi_R] \leq\lim_{R'\to \infty} V[\bar{B}_{R^\prime}]+\epsilon=E_{rel}[V]+\epsilon.
$$ 
Invoking \eqref{UpperBndEqn} and Lemma \ref{RelEntropyBndLem} again, implies that
$$
E_{rel}[V] \leq V[\mathbf{1}]\leq E_{rel}[V]+\epsilon.
$$
As $\epsilon>0$ is arbitrary, we have $V[\mathbf{1}]=E_{rel}[V]$ proving the claim. 
	
The second item follows from the strict maximum principle for stationary varifolds \cite{SolomonWhite} or \cite{WhiteMax}. Namely, as $V$ is $E_{rel}$-minimizing to first order in $\overline{\Omega^\prime}$, one has that $V_+$ is $E$-minimizing to first order in $\overline{\Omega^\prime}$. If $\spt (V_+)\setminus \tilde{\Omega}$ is non-empty, then, by Item (4) of Proposition \ref{ThickeningProp}, there is an $s_+ \in (0,1]$ or an $s_-\in (0,1]$ so that $\spt(V_+)\cap\Gamma^\pm_{s_\pm} \neq \emptyset$ but $\spt(V_+)\cap \Gamma^\pm_{s} = \emptyset$ for $s\in (s_\pm,1]$ -- here $\Gamma_s^\pm$ are the foliation of $\overline{\Omega^{\prime\prime}}\setminus \Omega^\prime$ given by Proposition \ref{ThickeningProp}. By the strict maximum principle of Solomon-White \cite{SolomonWhite} (see also \cite[Theorem 1]{WhiteMax}) and the fact that $\Gamma_{s_\pm}^\pm$ is strictly expander mean convex, this is impossible. Hence, $\spt(V_+)\subseteq \tilde{\Omega}$ and this completes the proof of the second item. The third item is an immediate consequence fact that $V_+$ is $E$-minimizing to first order in $\overline{\Omega^\prime}$, the fact that $\tilde{\Omega}\subset \Omega^\prime$ and Item (2).
	
Finally, to prove the fourth item, pick an (even) function $\psi\in C^2_c((\Real^{n+1}\backslash \set{\mathbf{0}})\times \mathbb{S}^n)$. Clearly, there is a constant $C>1$ so that the support of $\psi$ is contained in $B_C\setminus\bar{B}_{C^{-1}}$. Let
$$
\psi_\rho (p,\mathbf{v})=\rho^{n} e^{-\frac{|\mathbf{x}(p)|^2}{4}} \psi(\rho p,\mathbf{v}).
$$
One readily computes that, for $\rho\in (0,1)$,
$$
\Vert \psi_\rho\Vert_{\mathfrak{X}}\leq 10 C \rho^{n-1}  e^{-\frac{1}{4C^2\rho^2}}\Vert \psi\Vert_{C^2}.
$$
As $\psi_\rho$ has compact support, one immediately has (up to restricting) that $\psi_\rho\in \mathfrak{Y}_0(\overline{\Omega^\prime})$ and, by definition,
$$
\Vert \psi_\rho\Vert_{\mathfrak{Y}}\leq 	\Vert \psi_\rho\Vert_{\mathfrak{X}}\leq 10 C \rho^{n-1}  e^{-\frac{1}{4C^2\rho^2}}\Vert \psi\Vert_{C^2}.
$$
As $E_{rel}[V]=V[\mathbf{1}]<\infty$ by Item (1) and $V\in \overline{\mathfrak{Y}_\cC^*}(\overline{\Omega^\prime};\Lambda)$, we have, by Proposition \ref{YEstProp}, 
\begin{equation} \label{RescaleYBndEqn}
\left| V[\psi_\rho] \right|\leq 10 C C_2 (1+|V[\mathbf{1}]|) \rho^{n-1}  e^{-\frac{1}{4C^2\rho^2}}\Vert \psi\Vert_{C^2}.
\end{equation}

Let $\mathscr{D}_\rho\colon \Real^{n+1}\times\mathbb{S}^n\to \Real^{n+1}\times\mathbb{S}^n$ be the dilation map given by $\mathscr{D}_\rho(p,\mathbf{v})=(\rho p,\mathbf{v})$, and let $V_-$ be the usual varifold associated to $\Gamma_-$. It is straightforward to see that
$$
V[\psi_\rho]=\rho^{n} V_+[ \psi\circ\mathscr{D}_\rho] - \rho^{n} V_{-}[ \psi\circ\mathscr{D}_\rho] =(\mathscr{D}_{\rho})_\# V_+ [\psi]-(\mathscr{D}_{\rho})_\# V_{-} [\psi]
$$
where $(\mathscr{D}_\rho)_\# V_\pm$ are the usual pushforwards of varifolds $V_\pm$. Thus, by \eqref{RescaleYBndEqn},
$$
\lim_{\rho\to 0} \left( (\mathscr{D}_{\rho})_\# V_+ [\psi]-(\mathscr{D}_{\rho})_\# V_{-} [\psi] \right) =0.
$$
As $\lim_{\rho\to 0} \rho \Gamma_-=\cC(\Gamma_-)=\cC$ in $C^2_{loc}(\Real^{n+1}\backslash \set{\mathbf{0}})$, 
$$
\lim_{\rho\to 0} (\mathscr{D}_\rho)_\# V_{-} [\psi]=\int_{\mathcal{C}} \psi(p,\mathbf{n}_{\mathcal{C}}(p)) \, d\mathcal{H}^n.
$$
As a consequence, 
$$
\lim_{\rho\to 0} (\mathscr{D}_\rho)_\# V_+ [\psi]=\int_{\mathcal{C}} \psi(p,\mathbf{n}_{\mathcal{C}}(p)) \, d\mathcal{H}^n.
$$
As $C^2_c((\Real^{n+1}\backslash \set{\mathbf{0}})\times \mathbb{S}^n)$ is dense in $C^0_c((\Real^{n+1}\backslash \set{\mathbf{0}})\times \mathbb{S}^n)$, it follows that as $\rho\to 0$ the $(\mathscr{D}_\rho)_\# V_+$ converges to $\mathcal{C}$ in the sense of varifolds. Finally, as $\cC$ is a $C^3$-regular cone and $V_+$ is integral $E$-stationary varifold, one can appeal to \cite[Proposition 3.3]{BWProper} to get the $C^2$ convergence. 
\end{proof}
	
\section{Min-max theory for asymptotically conical self-expanders} \label{MountainPassSec}
In this section we adapt a notion of parametrized family of hypersurfaces in \cite{DeLellisRamic} (see also \cite{ColdingDeLellis} and \cite{DeLellisTasnady}) to the setting we are considering. Following the strategy of Almgren-Pitts \cite{Pitts} and Simon-Smith \cite{Smith} for the min-max construction of compact minimal hypersurfaces -- see also \cite{ColdingDeLellis}, \cite{DeLellisRamic} and \cite{DeLellisTasnady} -- we show that there is a min-max sequence that converges to an element of $\overline{\mathfrak{Y}^*_\cC}(\overline{\Omega^\prime};\Lambda)$ whose associated varifold is $E$-stationary, supported in $\tilde{\Omega}$ and has codimension-$7$ singular set. We continue to follow the conventions of Section \ref{ConventionSec}.

\subsection{Parameterized families} \label{ParamFamilySec}
Let $I^k=[0,1]^k$ be the $k$-dimensional cube. 

\begin{defn} \label{ParamFamilyDef} 
Fix $k\geq 1$. A \emph{generalized smooth family of hypersurfaces in $\overline{\Omega^\prime}$ parameterized by $I^k$} is a family of pairs, $\set{(U_\tau,\Sigma_\tau)}_{\tau\in I^k}$ where each $U_\tau\in\mathcal{C}(\Gamma_-^\prime,\Gamma_+^\prime)$ and $\Sigma_\tau=\partial^* U_\tau$ that satisfies
\begin{enumerate}
\item $E_{rel}[\Sigma_{\tau},\Gamma_-]<\infty$;
\item For each $\tau\in I^k$ there is a finite set $S_{\tau}\subset \overline{\Omega^\prime}$ so that $\mathrm{cl}(\Sigma_{\tau})$ is a smooth hypersurface in $\overline{\Omega^\prime}\backslash S_\tau$;
\item \label{ContinueItem} The map $\tau\mapsto V_{\Sigma_\tau}$ is continuous in the weak-* topology of $\mathfrak{Y}^*(\overline{\Omega^\prime})$;
\item As $\tau\to\tau_0$, $\mathrm{cl}(\Sigma_\tau)\to \mathrm{cl}(\Sigma_{\tau_0})$ in $C^{\infty}_{loc}(\Real^{n+1}\backslash S_{\tau_0})$;
\item The map $\tau\mapsto \mathbf{1}_{U_{\tau}}$ is continuous in $L^1_{loc}(\mathbb{R}^{n+1})$.
\end{enumerate}
	
The family $\set{(U_\tau, \Sigma_\tau)}_{\tau\in [0,1]}$ is a \emph{sweepout of $\tilde{\Omega}$} if $(U_0,\Sigma_0)=(\Omega_-(\Gamma_-),\Gamma_-)$ and $(U_1,\Sigma_1)=(\Omega_-(\Gamma_+),\Gamma_+)$.
\end{defn}

\begin{rem} \label{ContinueRem}
Notice that by Proposition \ref{YEstProp}, combined with the other requirements,  Item \eqref{ContinueItem} of Definition \ref{ParamFamilyDef} is equivalent to the condition that $E_{rel}[V_{\Sigma_\tau}]$ is continuous in $\tau$ and $\tau\mapsto\Sigma_\tau$ is continuous in the locally Hausdorff sense; c.f. \cite[Definition 1.2]{DeLellisRamic} and \cite[Definition 0.2]{DeLellisTasnady}.  We also emphasize that we don't demand the sweepout of $\tilde{\Omega}$ lies entirely within $\tilde{\Omega}$ only that it remains in $\overline{\Omega^\prime}$.  This is, nevertheless, more restrictive than the analogous hypotheses of \cite{DeLellisRamic} and ensures the element produced by the min-max procedure lies in $\overline{\Omega^\prime}$ and hence in $\tilde{\Omega}$. The reason the region is thickened to $\overline{\Omega^\prime}$ is that this gives more admissible variations and so simplifies the regularity theory.
\end{rem}

From now on we will refer to such objects in Definition \ref{ParamFamilyDef} as \emph{families parameterized by $I^k$}, and we will omit the parameter space, $I^k$. when it is clear from context.

\begin{defn} \label{HomotopyDef}
Two families $\set{(U_\tau,\Sigma_\tau)}$ and $\set{(U^\prime_\tau, \Sigma^\prime_\tau)}$ parameterized by $I^k$ are \emph{homotopic} if there is a family $\set{(W_{(\tau,\rho)},\Xi_{(\tau,\rho)})}$ parameterized by $I^{k+1}=I^k\times [0,1]$ so that
\begin{enumerate}
\item $(W_{(\tau,0)},\Xi_{(\tau,0)})=(U_\tau, \Sigma_\tau)$ for all $\tau\in I^k$;
\item $(W_{(\tau,1)}, \Xi_{(\tau,1)})=(U^\prime_\tau, \Sigma^\prime_\tau)$ for all $\tau\in I^k$;
\item $(W_{(\tau,\rho)}, \Xi_{(\tau,\rho)})=(U_{\tau}, \Sigma_{\tau})$ for all $\tau\in \partial I^k$ and all $\rho\in [0,1]$.
\end{enumerate} 
	
A set $X$ of families parameterized by $I^k$ is \emph{homotopically closed} if $X$ contains the homotopy class of each of its elements.
\end{defn}

\begin{defn} \label{RelEntWidthDef}
Let $X$ be a homotopically closed set of families parameterized by $I^k$. The \emph{relative expander entropy min-max value of $X$} denoted by $m_{rel}(X)$ is the value
$$
m_{rel}(X)=\inf\set{\max_{\tau\in I^k} E_{rel}[V_{\Sigma_\tau}]\colon \set{(U_\tau,\Sigma_\tau)}\in X}.
$$
The \emph{relative expander entropy boundary-max value of $X$} is 
$$
bM_{rel}(X)=\max \set{ E_{rel}[V_{\Sigma_\tau}]\colon \set{(U_\tau,\Sigma_\tau)}\in X, \tau\in \partial I^k}.
$$
A \emph{minimizing sequence} is a sequence of elements $\set{\set{(U_\tau^\ell,\Sigma_\tau^\ell)}_{\tau}}^\ell\subseteq X$ such that
$$
\lim_{\ell\to \infty} \max_{\tau\in I^k} E_{rel}[V_{\Sigma_\tau^\ell}]=m_{rel}(X).
$$
A \emph{min-max sequence} is obtained from a minimizing sequence by taking slices $\set{(U^\ell_{\tau_\ell},\Sigma^\ell_{\tau_{\ell}})}_\ell$ for $\tau_\ell \in I^k$ such that
$$
E_{rel}[V_{\Sigma_{\tau_\ell}^\ell}]\to m_{rel}(X).
$$
\end{defn}
It is obvious that any subsequence of a min-max sequence is a min-max sequence.

\subsection{Min-max construction for $E_{rel}$} \label{MinMaxSec}
We adapt the classical min-max theory for compact minimal surfaces to $E_{rel}$ in our setting. The main result of this section is the following:

\begin{thm} \label{MinMaxThm}
Let $X$ be a homotopically closed set of families in $\overline{\Omega^\prime}$ parametrized by $I^k$ with $m_{rel}(X)>bM_{rel}(X)$. There is a minimizing sequence $\set{\set{(U_\tau^\ell,\Sigma_\tau^\ell)}_\tau}^\ell$ in $X$ so that there is a min-max sequence $\set{(U^\ell_{\tau_\ell},\Sigma^\ell_{\tau_\ell})}_\ell$ and a pair $(U_0,\Gamma_0)$ with $U_0\in\mathcal{C}(\Gamma_-,\Gamma_+)$ and $\Gamma_0=\partial^* U_0$ so that 
\begin{enumerate}
\item $\Gamma_0$ is $E$-stationary, has codimension-$7$ singular set and $E_{rel}[\Gamma_0,\Gamma_-]=m_{rel}(X)$;
\item $\Sigma^\ell_{\tau_\ell}$ converges in the sense of varifolds to $\Gamma_0$ and $\mathbf{1}_{U_{\tau_\ell}^\ell}$ converges in $L^1_{loc}$ to $\mathbf{1}_{U_0}$.
\end{enumerate}
\end{thm}

\begin{rem}
This is stronger than what is achieved in the more general situation considered in \cite{DeLellisTasnady} as the geometry of the expander problem implies the limit is with multiplicity one.
\end{rem}

We now prove Theorem \ref{MinMaxThm} by following the strategy in \cite{ColdingDeLellis}, \cite{DeLellisRamic} and \cite{DeLellisTasnady}. The proof is divided into several parts.

\subsubsection{Pull-tight procedure for $E_{rel}$} \label{PullTightSec}
Set 
$$
\Lambda_0=\max\set{|E_-|,4|m_{rel}(X)|}
$$
where $E_-$ is the uniform lower bound for $E_{rel}$ given by Lemma \ref{RelEntropyBndLem}. Let 
$$
\mathcal{V}=\overline{\mathfrak{Y}^*_\cC}(\overline{\Omega^\prime}; \Lambda_0)
$$
and let
$$
\mathcal{V}_s=\set{V\in\mathcal{V}\colon \mbox{$V$ is $E_{rel}$-minimizing to first order in $\overline{\Omega^\prime}$}}.
$$
By Proposition \ref{YEstProp}, $\mathcal{V}\subseteq \bar{B}^{\mathfrak{Y}^*}_{R_0}$ for $R_0=C_2(1+\Lambda_0)$. Endow $\mathcal{V}$ with the weak-* topology. Thus, by the Banach-Alaoglu theorem, $\mathcal{V}$ is metrizable and compact. Let $\mathcal{D}$ be a choice of such metric. As, by Proposition \ref{FirstVarProp}, for $\mathbf{Y}\in\mathcal{Y}^-_0(\overline{\Omega^\prime})$ the map assigning $\delta^+ V[\mathbf{Y}]$ to each $V\in\mathcal{V}$ is continuous, one has that $\mathcal{V}_s$ is a closed subset of $\mathcal{V}$ and so is compact.

\begin{prop} \label{PullTightProp}
Let $X$ be a homotopically closed set of families in $\overline{\Omega^\prime}$ parameterized by $I^k$ with $m_{rel}(X)>bM_{rel}(X)$. There is a minimizing sequence $\set{\set{U_\tau^\ell,\Sigma_\tau^\ell}_\tau}^\ell$ in $X$ so that if $\set{(U_{\tau_\ell}^\ell,\Sigma_{\tau_\ell}^\ell)}_\ell$ is a min-max sequence, then $\mathcal{D}(V_{\Sigma_{\tau_\ell}^\ell},\mathcal{V}_s)\to 0$. 
\end{prop}

\begin{proof}
Let $\mathfrak{Is}(\overline{\Omega^\prime})$ be the set of all isotopies of $\overline{\Omega^\prime}$, i.e., smooth maps $\Phi\colon [0,1]\times\overline{\Omega^\prime}\to\overline{\Omega^\prime}$ so that each $\Phi(t,\cdot)$ is a diffeomorphism of $\overline{\Omega^\prime}$. If we denote by 
$$
bX=\set{V_{\Sigma_\tau}\colon \set{\Sigma_{\tau}}\in X, \tau\in\partial I^k}.
$$
then our hypothesis on $X$ and Lemma \ref{RelEntropyBndLem} ensure that $bX\subseteq\mathcal{V}$. We now adapt the main steps of pull-tight arguments of Colding-De Lellis \cite{ColdingDeLellis} and De Lellis-Ramic \cite{DeLellisRamic} to our setting to construct a continuous map $\mathcal{V} \to \mathfrak{Is}(\overline{\Omega^\prime})$ given by $V\mapsto \Phi_V$ so that
\begin{enumerate}
\item If $V\in \mathcal{V}_s\cup bX$, then $\Phi_V$ is the identity map;
\item If $V\not \in \mathcal{V}_s\cup bX$, then $(\Phi_V)^+_\# V[\mathbf{1}]<V[\mathbf{1}]$.
\end{enumerate}

{\bf Step 1: A map from $\mathcal{V}$ to $\mathcal{Y}_0^-(\overline{\Omega^\prime})$.}
For $j\in\mathbb{Z}$, let
$$
\mathcal{V}_j=\set{V\in \mathcal{V}\colon 2^{-j+1}\geq \mathcal{D}(V, \mathcal{V}_s)\geq 2^{-j-2}}.
$$
For each $V\in\mathcal{V}_j$, as $V$ is not $E_{rel}$-minimizing to first order in $\overline{\Omega^\prime}$, there is a vector field $\bar{\mathbf{Y}}_V=\bar{\alpha}_V\mathbf{Y}_0+\bar{\mathbf{Z}}_V\in\mathcal{Y}_0^-(\overline{\Omega^\prime})$ where $\bar{\alpha}_V\in\mathbb{R}$ and $\bar{\mathbf{Z}}_V\in \mathcal{Y}_c(\overline{\Omega^\prime})$ so that $\delta^+ V[\bar{\mathbf{Y}}_V]<0$. Moreover, by linearity we may assume that, for $j\geq 1$, 
$$
\Vert\bar{\mathbf{Y}}_V\Vert_{\mathcal{Y}}+\Vert\bar{\mathbf{Z}}_V\Vert_{C^j}\leq \frac{1}{j}.
$$
By \eqref{PlusFirstVarEqn}, for such $V$ there is an open ball $B^{\mathcal{D}}_{2\rho}(V)\subset (\mathcal{V},\mathcal{D})$ so that, for any $V^\prime\in B^{\mathcal{D}}_{2\rho}(V)$,
$$
\delta^+ V^\prime [\bar{\mathbf{Y}}_V]\leq \frac{1}{2}\delta^+ V[\bar{\mathbf{Y}}_V]<0.
$$

As $\mathcal{V}_j$ is compact for the metric $\mathcal{D}$, arguing as in \cite[Proposition 4.1]{ColdingDeLellis}, one finds a locally finite covering of $\mathcal{V}\setminus\mathcal{V}_s$ by these balls so that any ball intersects at most three of $\mathcal{V}_j$. Let $\set{\varphi_i}$ be a partition of unity subordinate to this cover. Thus, as $\mathcal{Y}^-_0(\overline{\Omega^\prime})$ is a convex cone, we can define the map
$$
\mathcal{V}\setminus\mathcal{V}_s\ni V\mapsto \mathbf{Y}_V=\sum_{i} \varphi_i (V) \bar{\mathbf{Y}}_{V_i} \in \mathcal{Y}^-_0(\overline{\Omega^\prime})\cap C^\infty(\overline{\Omega^\prime}; T\overline{\Omega^\prime})
$$
where $\mathcal{Y}^-_0(\overline{\Omega^\prime})$ is endowed with the $\mathcal{Y}$ norm and $C^\infty(\overline{\Omega^\prime}; T\overline{\Omega^\prime})$ is with the usual Fr\'{e}chet topology. Our construction ensures that
\begin{enumerate}
\item $\delta^+ V[\mathbf{Y}_V]<0$ for any $V\in\mathcal{V}\setminus\mathcal{V}_s$;
\item $V\mapsto\mathbf{Y}_V$ is continuous;
\item $\Vert\mathbf{Y}_V\Vert_{\mathcal{Y}}+\Vert\mathbf{Y}_V\Vert_{C^{j-1}}\leq \frac{1}{j-1}$ if $\mathcal{D}(V,\mathcal{V}_s)\leq 2^{-j}$ and $j\geq 2$.
\end{enumerate}
Extend the map $V\mapsto\mathbf{Y}_V$ to $\mathcal{V}$ by setting it identically equal to $\mathbf{0}$ on $\mathcal{V}_s$. By Item (3) this extension is continuous in both $\mathcal{Y}$ and $C^k$ norm. That is, the map $V\mapsto\mathbf{Y}_V$ is indeed continuous in the $C^\infty$ space with its usual Fr\'{e}chet topology.

{\bf Step 2: A map from $\mathcal{V}$ to $\mathfrak{Is}(\overline{\Omega^\prime})$.}
For each $V\in\mathcal{V}$ denote by $\set{\Phi_V(t)}$ the family of diffeomorphisms generated by $\mathbf{Y}_V$. By Item (2) in Step 1 and Proposition \ref{ContPlusPushforwardProp}, the map 
$$
[0,\infty)\times \mathcal{V}\ni (t,V)\mapsto \delta^+(\Phi_{V}(t)_\#^+ V)[\mathbf{Y}_V]
$$
is continuous. Thus, for each $V\in\mathcal{V}\setminus\mathcal{V}_s$, there is a positive time $\sigma_V$ and a radius $\rho_V$ so that, for all $t \in [0,\sigma_V]$ and $V^\prime\in B^{\mathcal{D}}_{2\rho_V}(V)$,
$$
\delta^+(\Phi_{V^\prime} (t)_\#^+ V^\prime)[\mathbf{Y}_{V^\prime}]\leq \frac{1}{4} \delta^+ V[\mathbf{Y}_V]<0.
$$
Arguing as in the first step we can construct a continuous function $\sigma\colon \mathcal{V}\to [0,\infty)$ so that
\begin{enumerate}
\item[(a)] $\sigma=0$ on $\mathcal{V}_s$; 
\item[(b)] $\sigma>0$ on $\mathcal{V}\setminus\mathcal{V}_s$;
\item[(c)] $\max_{t\in [0,\sigma(V)]}\delta^+ (\Phi_V(t)_\#^+ V)[\mathbf{Y}_V]<0$ for every $V\in\mathcal{V}\setminus\mathcal{V}_s$.
\end{enumerate}

Define $b\colon\mathcal{V}\to [0,1]$ by
$$
b(V)=\min\set{\mathcal{D}(V, bX), 1}.
$$
Clearly, $b$ is continuous. Now redefine a new $\mathbf{Y}_V$ by multiplying the old one by $b(V)\sigma(V)$. This newly defined $\mathbf{Y}_V$ still continuously depends on $V$ and vanishes identically on $\mathcal{V}_s$, however, property (c) becomes 
\begin{enumerate}
\item[(c')] $\max_{t\in [0,1]} \delta^+ (\Phi_V(t)_\#^+ V)[\mathbf{Y}_V]<0$ for every $V\in\mathcal{V}\setminus(\mathcal{V}_s\cup bX)$.
\end{enumerate}
As $\mathbf{Y}_V$ is tangent to $\Gamma_-$ outside a compact set and $\Gamma_-$ is a self-expander, it is not hard to see that if $V\in\mathcal{V}\setminus(\mathcal{V}_s\cup bX)$ and $0<t\leq 1$, then 
$$
\Phi_V(t)_\#^+ V[\mathbf{1}]=V[\mathbf{1}]+\int_0^t \delta^+ (\Phi_V(s)_\#^+ V)[\mathbf{Y}_V] \, ds<V[\mathbf{1}].
$$

{\bf Step 3: Construction of the competitor and conclusion.}
It is convenient to identify $\Gamma$ with $V_\Gamma$. Take a minimizing sequence $\set{\set{(Z_\tau^\ell,\Upsilon_\tau^\ell)}_{\tau}}^\ell\subseteq X$ and consider families $\set{(W_\tau^\ell,\Xi_\tau^\ell)}_\tau$ given by 
$$
W_\tau^\ell=\Phi_{\Upsilon_\tau^\ell}(1,Z_\tau^\ell) \mbox{ and } \Xi_\tau^\ell=\Phi_{\Upsilon_\tau^\ell}(1,\Upsilon_\tau^\ell).
$$
As, by our construction, $\tau\mapsto\mathbf{Y}_{\Upsilon_\tau^\ell}$ is only guaranteed continuous, in general $\set{(W_\tau^\ell,\Xi_\tau^\ell)}_\tau$ may not be an element of $X$. Thus we need to regularize $\set{(W_\tau^\ell,\Xi_\tau^\ell)}_\tau$ in $\tau$. To achieve this, we use standard mollifier techniques to construct a smooth map $\tau\mapsto \mathbf{X}_{\Upsilon^\ell_\tau}\in \mathcal{Y}^-(\overline{\Omega^\prime})$ with the estimate
$$
\max_{\tau\in I^k}\Vert\mathbf{Y}_{\Upsilon^\ell_\tau}-\mathbf{X}_{\Upsilon^\ell_\tau}\Vert_{\mathcal{Y}} \leq \ell^{-1}.
$$
Notice that, unlike $\mathbf{Y}_{\Upsilon_\tau^\ell}$,  $\mathbf{X}_{\Upsilon_\tau^\ell}$ may not be in $\mathcal{Y}^-_0(\overline{\Omega^\prime})$ but instead lies in the bigger space $\mathcal{Y}^-(\overline{\Omega^\prime})$. Let $\set{\Psi_{\Upsilon_\tau^\ell}(t)}_{t\geq 0}$ be the family of diffeomorphisms generated by $\mathbf{X}_{\Upsilon^\ell_\tau}$. If $U_\tau^\ell=\Psi_{\Upsilon_\tau^\ell}(1,Z^\ell_\tau)$ and $\Sigma_\tau^\ell=\Psi_{\Upsilon_\tau^\ell}(1,\Upsilon_\tau^\ell)$, then the smoothness of $\mathbf{X}_{\Upsilon^\ell_\tau}$ in $\tau$ ensures that $\set{(U_\tau^\ell,\Sigma_\tau^\ell)}_\tau$ is indeed an element of $X$. We will show $\set{\set{(U_\tau^\ell,\Sigma_\tau^\ell)}_{\tau}}^\ell$ is a minimizing sequence in $X$ with the desired property.

By construction and Proposition \ref{ContPlusPushforwardProp} one has
\begin{equation} \label{DiffEqn}
\lim_{\ell\to \infty} \max_{\tau\in I^k} \mathcal{D}(\Sigma_\tau^\ell,\Xi_\tau^\ell)=0.
\end{equation}
Hence, Item (c') in Step 2 and \eqref{DiffEqn}, imply there is a $\epsilon_\ell\searrow 0$ so 
$$
E_{rel}[\Sigma_\tau^\ell]-\epsilon_\ell\leq E_{rel}[\Xi_\tau^\ell] \leq E_{rel}[\Upsilon_\tau^\ell].
$$
Thus, as $\set{\set{(Z_\tau^\ell,\Upsilon_\tau^\ell)}_\tau}^\ell$ is a minimizing sequence, so is $\set{\set{(U^\ell_\tau,\Sigma_\tau^\ell)}_\tau}^\ell$. Moreover, if $\set{(U_{\tau_\ell}^\ell,\Sigma_{\tau_\ell}^\ell)}_{\ell}$ is a min-max sequence, then so is $\set{(Z_{\tau_\ell}^\ell,\Upsilon_{\tau_\ell}^\ell)}_\ell$. We claim $\mathcal{D}(\Sigma_{\tau_\ell}^\ell,\mathcal{V}_s)\to 0$. Suppose not and there were a subsequence $\ell_i$ so that $\mathcal{D}(\Sigma_{\tau_{\ell_i}}^{\ell_i},\mathcal{V}_s)>\delta>0$. By the compactness of $\mathcal{V}$, up to passing to a further subsequence and relabelling, the $\Upsilon_{\tau_{\ell_i}}^{\ell_i}$ converges in the weak-* topology to some $V_0\in\mathcal{V}$. It is enough to show that $V_0\in\mathcal{V}_s$ and $\Sigma_{\tau_{\ell_i}}^{\ell_i}\to V_0$. This would give a contradiction.

By Lemma \ref{RegPushforwardLem} and Proposition \ref{ContPlusPushforwardProp}, $\Xi_{\tau_{\ell_i}}^{\ell_i}\to \Phi_{V_0}(1)_\#^+ V_0$ in the weak-* topology. As remarked before,  $\set{\Upsilon_{\tau_{\ell_i}}^{\ell_i}}$ is a min-max sequence while $E_{rel}[\Xi_{\tau_{\ell_i}}^{\ell_i}]\to m_{rel}(X)$. Thus, 
$$
\Phi_{V_0}(1)_\#^+ V_0[\mathbf{1}]=V_0[\mathbf{1}]=m_{rel}(X).
$$
As $m_{rel}(X)>bM_{rel}(X)$ one has $V_0\notin bX$. If $V_0\notin\mathcal{V}_s$, then by Item (c') 
$$
\Phi_{V_0}(1)^+_\# V_0[\mathbf{1}]=V_0[\mathbf{1}]+\int_0^1 \delta^+ (\Phi_{V_0}(t)_\#^+ V_0)[\mathbf{Y}_{V_0}] \, dt<V_0[\mathbf{1}].
$$
This is a contradiction, so $V_0\in\mathcal{V}_s$. By construction, $\mathbf{Y}_{V_0}=\mathbf{0}$ and $\Phi_{V_0}(1)_\#^+ V_0=V_0$. Hence, $\Xi_{\tau_{\ell_i}}^{\ell_i}\to V_0$ and hence, by \eqref{DiffEqn}, so does $\Sigma_{\tau_{\ell_i}}^{\ell_i}$. This completes the proof.
\end{proof}

\subsubsection{Almost minimizing}
We first observe that, with minor modifications, the proof from Sections 4 and 5 of \cite{DeLellisRamic} (see also \cite{ColdingDeLellis} and \cite{DeLellisTasnady}) that there is a min-max sequence produced by Proposition \ref{PullTightProp} that is almost-minimizing in appropriate annuli.

Let us first state what we mean by almost minimizing. 

\begin{defn}
Fix $\epsilon>0$, $W\subseteq \mathbb{R}^{n+1}$ an open subset -- not necessarily bounded -- and $k\in\mathbb{N}$. A boundary $\partial^* U$ for some $U\in\mathcal{C}(\Gamma_-^\prime,\Gamma_+^\prime)$ is \emph{$\epsilon$-almost minimizing in $W$} if there is {\bf no} one-parameter family $\set{\partial^* U_s}_{s\in [0,1]}$ satisfying the following properties:
\begin{enumerate}
\item All the properties of Definition \ref{ParamFamilyDef} hold for $\set{( U_s,\partial^* U_s,)}_{s\in [0,1]};$
\item $U_0=U$ and there is a bounded open subset $W^\prime$ with $\overline{W^\prime}\subset W$ so that $U_s\backslash W^\prime=U\backslash W^\prime$ for all $s\in [0,1]$;
\item $E_{rel}[\partial^* U_s, \Gamma_-]\leq E_{rel}[\partial^* U, \Gamma_-]+\frac{\epsilon}{2^{k+2}}$ for all $s\in [0,1]$;
\item $E_{rel}[\partial^* U_1, \Gamma_-]\leq E_{rel}[\partial^* U, \Gamma_-] -\epsilon$.
\end{enumerate}
A sequence $\set{\partial^* U^i}$ of hypersurfaces is called \emph{almost minimizing (or a.m.) in $W$} if each $\partial^* U^i$ is $\epsilon_i$-almost minimizing for some $\epsilon_i\to 0$.
\end{defn}

\begin{rem}
We emphasize the that the only difference with the definition in \cite{DeLellisRamic}, is the introduction of the fixed domain $W^\prime$ that is pre-compact in $W$. This is needed as we are dealing with a non-compact domain, but does not affect any of the arguments.
\end{rem}

Let 
$$
\mathcal{AN}_{\rho}(p)=\set{An(p,r_1,r_2)=B_{r_2}(p)\setminus\bar{B}_{r_1}(p)\colon 0<r_1<r_2<\rho}
$$
be the set of all (open) annuli centered at $p$ with outer radius less than $\rho$. We will show the following:

\begin{prop}\label{AMProp}
There is a function $\varrho\colon \overline{\Omega^\prime}\to (0,\infty)$, an element $V_0\in\overline{\mathfrak{Y}^*_\mathcal{C}}(\overline{\Omega^\prime};\Lambda_0)$ that is $E_{rel}$-minimizing to first order in $\overline{\Omega^\prime}$, and a min-max sequence $\set{(U^{\ell}_{\tau_\ell},\Sigma^\ell_{\tau_\ell})}_\ell$ so that
\begin{enumerate}
\item $\set{(U^\ell_{\tau_\ell},\Sigma^\ell_{\tau_\ell})}_\ell$ is a.m. in every $An\in\mathcal{AN}_{\varrho(p)}(p)$ for all $p\in \overline{\Omega^\prime}$;
\item $V_{\Sigma^\ell_{\tau_\ell}}$ converges in the metric $\mathcal{D}$ to $V_0$ as $\ell\to\infty$.
\end{enumerate}
\end{prop}

We observe that De Lellis-Ramic's version Almgren-Pitts combinatorial Lemma can be adapted to our setting. As remarked before, it is convenient to allow unbounded sets. This has almost no effect on the proof.  

\begin{defn}
Let $d\in\mathbb{N}$ and $W^1,\dots, W^d$ be open sets in $\mathbb{R}^{n+1}$. A boundary $\partial^* U$ for some $U\in\mathcal{C}(\Gamma_-^\prime,\Gamma_+^\prime)$ is said to be \emph{$\epsilon$-almost minimizing in $(W^1,\dots,W^d)$} if it is $\epsilon$-a.m. in at least one of the open sets $W^1,\dots,W^d$.  

Denote by $\mathcal{CO}_d$ the set of $d$-tuples $(W^1,\dots,W^d)$ where $W^1,\dots,W^d$ are open subsets of $\mathbb{R}^{n+1}$ with the property that, for all $i,j\in\set{1,\dots,d}$ and $i\neq j$,
$$
\mathrm{dist}(W^i,W^j) \geq 4 \min\set{\mathrm{diam}(W^i),\mathrm{diam}(W^j)}.
$$
Here
$$
\mathrm{dist}(W^i,W^j)=\inf_{p\in W^i, q\in W^j} |\mathbf{x}(p)-\mathbf{x}(q)|.
$$
\end{defn}

Observe that we do not require that the open sets $W^1,\dots,W^d$ are bounded, however, for a tuple $(W^1,\dots,W^d)$ to lie in $\mathcal{CO}_d$ all but one must be bounded.

We now state the Almgren-Pitts combinatorial lemma. The proof is identical to that in \cite[Proposition 4.8]{DeLellisRamic}.

\begin{prop}\label{CombinatorialProp}
Let $X$ be a homotopically closed set of families in $\overline{\Omega^\prime}$ parametrized by $I^k$ with $m_{rel}(X)>bM_{rel}(X)$. There is a $d\in\mathbb{N}$, an element $V_0\in\overline{\mathfrak{Y}^*_\cC}(\overline{\Omega^\prime};\Lambda_0)$ that is $E_{rel}$-minimizing to first order in $\overline{\Omega^\prime}$, and a min-max sequence $\set{(U^\ell_{\tau_\ell},\Sigma_{\tau_\ell}^\ell)}_\ell$ such that
\begin{enumerate}
\item $V_{\Sigma_{\tau_\ell}^\ell}$ converges in the metric $\mathcal{D}$ to $V_0$ as $\ell\to\infty$;
\item For any $(W^1,\dots,W^d)\in \mathcal{CO}_d$, $\Sigma_{\tau_\ell}^\ell$ is $\frac{1}{\ell}$-a.m. in $(W^1,\dots, W^d)$ for $\ell$ large enough.
\end{enumerate}
\end{prop}

We can now modify the arguments in \cite[Proposition 4.3]{DeLellisRamic} to prove Proposition \ref{AMProp}.

\begin{proof}[Proof of Proposition \ref{AMProp}]
Let $d\in\mathbb{N}$ and $\set{\Sigma_{\tau_\ell}^\ell}_{\ell}$ be the number and min-max sequence given by Proposition \ref{CombinatorialProp}. Write $\Sigma^\ell=\Sigma_{\tau_\ell}^\ell$. We will show that a subsequence of $\set{\Sigma^\ell}$ satisfies the desired properties. For any $r_1>0$ and $r_2,\dots,r_d$ with $0<r_{i}<\frac{1}{9} r_{i-1}$, set $r_i^\prime=\frac{1}{9} r_{i-1}$ and consider the tuple $(W^1_{r_1}(p),\dots,W^d_{r_d}(p))$ given by
\begin{align*}
W^1_{r_1}(p)  & =\mathbb{R}^{n+1}\setminus\bar{B}_{r_1}(p); \\
W^i_{r_i}(p) & =B_{r_i^\prime}(p)\setminus\bar{B}_{r_i}(p) \mbox{ for $2\leq i\leq d-1$}; \\
W^d_{r_d}(p) & =B_{r_d} (p).
\end{align*}
By definition $(W^1_{r_1}(p),\dots,W^d_{r_d}(p))\in\mathcal{CO}_d$ and so $\Sigma^\ell$ is $\frac{1}{\ell}$-a.m. in at least one $W^i_{r_i}(p)$. For any $r_1>0$ fixed, one of the two situations occurs:
\begin{enumerate}
\item $\Sigma^\ell$ is $\frac{1}{\ell}$-a.m. in $(W^2_{r_2}(q),\dots, W^d_{r_d}(q))$ for every $q\in\overline{\Omega^\prime}$ and every choice of $r_2,\dots,r_d$ and for $\ell$ large;
\item For every $K\in\mathbb{N}$, there is some $\ell_K\geq K$ and a point $p_{r_1}^{\ell_K}\in\overline{\Omega^\prime}$ so that $\Sigma^{\ell_K}$ is $\frac{1}{\ell_K}$-a.m. in $\mathbb{R}^{n+1}\setminus\bar{B}_{r_1}(p_{r_1}^{\ell_K})$.
\end{enumerate}

First assume there is no $r_1>0$ so that Case (1) holds. Thus, by choosing Case (2) with $r_1=1/j$ and $K=j$ for every $j\in\mathbb{N}$, we obtain a subsequence $\set{\Sigma^{\ell_j}}_{j}$ and a sequence of points $\set{p_{j}^{\ell_j}}_{j}$ in $\overline{\Omega^\prime}$ so that $\Sigma^{\ell_j}$ is $\frac{1}{\ell_j}$-a.m. in $\mathbb{R}^{n+1}\setminus \bar{B}_{\frac{1}{j}}(p_{j}^{\ell_j})$. If $p_j^{\ell_j}$ is unbounded, then $\Sigma^{\ell_j}$ is $\frac{1}{\ell_j}$-a.m. in any bounded open subset of $\Real^{n+1}$ once $j$ is large. In particular, this shows the claim with $\varrho\equiv 1$. Otherwise, up to passing to a subsequence, $p_{j}^{\ell_j}$ converges to a point $p_0\in\overline{\Omega^\prime}$. It then follows that, for every $N\in\mathbb{N}$, $\Sigma^{\ell_j}$ is $\frac{1}{\ell_j}$-a.m. in $\mathbb{R}^{n+1}\setminus\bar{B}_{\frac{1}{N}}$ for $j$ large. Thus, if $q\in\overline{\Omega^\prime}\setminus\set{p_0}$, then we can choose $\varrho(q)$ so that $\bar{B}_{\varrho(q)}(q)\subset\mathbb{R}^{n+1}\setminus\set{p_0}$ whereas $\varrho(p_0)$ can be chosen arbitrarily. Hence, $\set{\Sigma^{\ell_j}}_j$ is a.m. in any annulus of $\mathcal{AN}_{\varrho(q)}(q)$ for any $q\in\overline{\Omega^\prime}$.

Now assume there is some $r_1>0$ so that Case (1) holds. Fix any $R>1$ and, as $\overline{\Omega^\prime}\cap\bar{B}_R$ is compact, we can divide $\overline{\Omega^\prime}\cap\bar{B}_R$ into finitely many closed subsets $\Omega_1,\dots,\Omega_M$ so that $\mathrm{diam}(\Omega_i)<r_2^\prime$ for all $i$. Similar to the reasoning above, for each $\Omega_i$, starting with $\Omega_1$, consider the two mutually exclusive cases:
\begin{itemize}
\item[(a)] There is some fixed $r_{2,i}>0$ so that $\set{\Sigma^\ell}$ must be $\frac{1}{\ell}$-a.m. in $(W^3_{r_3}(q),\dots,W^d_{r_d}(q))$ for every $q\in\Omega_i$ and every $r_3,\dots,r_d$ with $r_3<\frac{1}{9} r_{2,i}$ and $r_j<\frac{1}{9} r_{j-1}$ and for $\ell$ large;
\item[(b)] There is a subsequence $\set{\Sigma^\ell}$ (not relabeled) and a sequence of points $\set{p_{i,\ell}}$ in $\Omega_i$ so that $\Sigma^\ell$ is $\frac{1}{\ell}$-a.m. in $B_{r_2^\prime}(p_{i,\ell})\setminus\bar{B}_{\frac{1}{\ell}}(p_{i,\ell})$.
\end{itemize}
If Case (b) holds, then $p_{i,\ell}\to p_i\in\Omega_i$ and we can choose $\varrho(p_i)\in (\mathrm{diam}(\Omega_i),r_2^\prime)$. For any other $q\in\Omega_i$, we can choose $\varrho(q)$ so that $\bar{B}_{\varrho(q)}(q)\subset B_{\varrho(p_i)}(p_i)\setminus \set{p_i}$. We then proceed onto $\Omega_{i+1}$, where either Case (a) is chosen, or a further subsequence is extracted and define further values of the function $\varrho$. For the subsets $\Omega_{i_1},\dots,\Omega_{i_l}$ where Case (a) holds, we define $r_2=\min\set{r_{2,i_1},\dots,r_{2,i_l}}$ and then continue iteratively by first subdividing the sets and consider the relevant cases. Note that if in the last instance of the iteration Case (a) holds, it follows that $\Sigma^\ell$ is $\frac{1}{\ell}$-a.m. in $B_{r_d}(q)$ for some $r_d>0$ and all $q$, hence we can choose $\varrho(q)=r_d$. Finally, the result follows from a standard diagonal argument.
\end{proof}

By Lemma \ref{RelEntropyBndLem} and Proposition \ref{StationaryProp}, $V_0$ has a decomposition $V_0=e^{\frac{|\mathbf{x}|^2}{4}}V_+-V_{\Gamma_-}^E$ where $V_+$ is a varifold supported in $\tilde{\Omega}\subset\Omega^\prime$. As the argument in \cite{DeLellisTasnady} is local, one can use it to obtain the following interior regularity of $V_+$ as a consequence of Proposition \ref{AMProp}.

\begin{prop} \label{InteriorRegProp}
Let $V_0$ be given by Proposition \ref{AMProp}. Then $V_0=e^{\frac{|\mathbf{x}|^2}{4}} V_+- V_{\Gamma_-}^E$ where $V_+$ is an integer multiplicity $E$-stationary varifold with codimension-$7$ singular set.
\end{prop}

\subsubsection{Proof of Proposition \ref{MinMaxThm}}
Appealing to Propositions \ref{PullTightProp}, \ref{AMProp} and \ref{InteriorRegProp} one obtains a min-max sequence $\set{(U^\ell,\Sigma^\ell)}=\set{(U^\ell_{\tau_\ell},\Sigma^\ell_{\tau_\ell})}$ and a pair $(U_0,V_0)$ for $(U_0,V_0)\in\mathcal{C}(\Gamma_-,\Gamma_+)\times\overline{\mathfrak{Y}^*_\cC}(\overline{\Omega^\prime};\Lambda_0)$ so that 
\begin{itemize}
\item $V_0=e^{\frac{|\mathbf{x}|^2}{4}}V_+-V_{\Gamma_-}^E$ for some integer multiplicity $E$-stationary varifold $V_+$  with codimension-$7$ singular set;
\item $(U^\ell,V_{\Sigma^\ell})$ converges in the sense of Corollary \ref{PairConvergeCor} to $(U_0,V_{0})$. 
\end{itemize}
Let $\Gamma_0$ be the regular part of (the support of) $V_+$. By the last item of Proposition \ref{StationaryProp}, $\Gamma_0$ is $C^2$-asymptotic to $\cC(\Gamma_-)$ and $V_+=\mathcal{H}^n\lfloor \Gamma_0$ in $\mathbb{R}^{n+1}\setminus\bar{B}_R$ for some $R>0$. As the support of $V_+$ has no compact components, the constancy theorem implies that $V_+$ has multiplicity one and so, by the first item of Proposition \ref{StationaryProp}, 
$$
E_{rel}[\Gamma_0,\Gamma_-]=V_0[\mathbf{1}]=\lim_{\ell\to\infty} V_{\Sigma^\ell_{\tau_\ell}}[\mathbf{1}]=\lim_{\ell\to\infty} E_{rel}[V_{\Sigma^\ell_{\tau_\ell}}]=m_{rel}(X).
$$
It remains only to show $\overline{\partial^* U_0}=\overline{\Gamma_0}$. By the nature of convergence, $\overline{\partial^* U_0}\subseteq\overline{\Gamma_0}$ and $\mathbf{1}_{U_0}$ is constant on each component of $\mathbb{R}^{n+1}\setminus\overline{\Gamma_0}$. As $\Sigma^\ell=\partial^* U^\ell$ converges as varifolds with multiplicity one to $\Gamma_0$ the claim follows immediately.

\section{Lower bound on the relative expander entropy of sweepouts} \label{LowerBndSec}
Continue to use the conventions of Sections \ref{ConventionSec} and \ref{FoliationThickenSec}. Using a calibration argument together with a certain observation about $E$-minimizers we show that any sweepout of $\tilde{\Omega}$ must pass through a hypersurface of uniformly larger relative entropy. We refer interested readers to \cite[Section 11]{DeLellisRamic} and \cite{MorganRos} for alternative approaches that are built on work of White \cite{WhiteMinMax}. 

This section is devoted to proving the following:

\begin{prop}\label{StrictLowBndRelEntProp} 
There is a constant $\delta_0=\delta_0(\Gamma_-,\Gamma_+)>0$ so that if $\set{(U_\tau,\Sigma_\tau)}_{\tau\in [0,1]}$ is a sweepout of $\tilde{\Omega}$, then
$$
\max_{\tau\in[0,1]} E_{rel}[\Sigma_\tau, \Gamma_-] \geq  \max \set{E_{rel}[\Gamma_-, \Gamma_-], E_{rel}[\Gamma_+, \Gamma_-]}+\delta_0\geq \delta_0.
$$
\end{prop}

To prove Proposition \ref{StrictLowBndRelEntProp} we will need several auxiliary lemmas. The first is an observation that the truncation to $\tilde{\Omega}$ decreases the relative expander entropy.

\begin{lem} \label{TruncateLem}
For $U\in\mathcal{C}(\Gamma_-^\prime,\Gamma_+^\prime)$ let
$$
\mathscr{R}[U]=\left(U\cap\Omega_-(\Gamma_+)\right)\cup \Omega_-(\Gamma_-)\in\mathcal{C}(\Gamma_-,\Gamma_+).
$$
Then we have
\begin{enumerate}
\item The map $\mathbf{1}_U\mapsto \mathbf{1}_{\mathscr{R}[U]}$ is continuous in $L^1_{loc}$;
\item $E_{rel}[\partial^* U,\Gamma_-]\geq E_{rel}[\partial^* \mathscr{R}[U],\Gamma_-]$.
\end{enumerate}
\end{lem}

\begin{proof}
As 
$$
\mathbf{1}_{\mathscr{R}[U]}=\mathbf{1}_{U}\mathbf{1}_{\Omega_-(\Gamma_+)}+\mathbf{1}_{\Omega_-(\Gamma_-)}-\mathbf{1}_U\mathbf{1}_{\Omega_-(\Gamma_+)}\mathbf{1}_{\Omega_-(\Gamma_-)}
$$
the first claim follows from this.
	
To prove the second, we may assume $E_{rel}[\partial^* U,\Gamma_-]<+\infty$ as otherwise the inequality holds trivially. By Proposition \ref{ThickeningProp} one can define $\mathbf{V}\colon\overline{\Omega^{\prime\prime}}\cap\overline{\Omega_+(\Gamma_+)}\to\mathbb{R}^{n+1}$ by
$$
\mathbf{V}(p)=\mathbf{n}_{\Gamma^+_s}(p)  \mbox{ if $p\in\Gamma^+_s$}
$$
and so
\begin{equation} \label{EMCPositiveEqn}
\Div\mathbf{V}+\frac{\mathbf{x}}{2}\cdot\mathbf{V}\geq 0.
\end{equation}
As $\overline{U\Delta\mathscr{R}[U]}\subseteq\overline{\Omega^\prime}\subseteq\overline{\Omega^{\prime\prime}}$, setting $\mathbf{V}_R=\phi_{R,1}\mathbf{V}$ and using the divergence theorem, one readily computes
\begin{align*}
& \int_{\partial^* U\cap\overline{\Omega_+(\Gamma_+)}} \phi_{R,1} e^{\frac{|\mathbf{x}|^2}{4}} \, d\mathcal{H}^n-\int_{\partial^* \mathscr{R}[U]\cap\overline{\Omega_+(\Gamma_+)}} \phi_{R,1} e^{\frac{|\mathbf{x}|^2}{4}} \, d\mathcal{H}^n \\
\geq & \int_{\partial^* U\cap\overline{\Omega_+(\Gamma_+)}} \mathbf{V}_R\cdot\mathbf{n}_{\partial^* U} e^{\frac{|\mathbf{x}|^2}{4}} \, d\mathcal{H}^n-\int_{\partial^*\mathscr{R}[U]\cap\overline{\Omega_+(\Gamma_+)}} \mathbf{V}_R\cdot\mathbf{n}_{\partial^* \mathscr{R}[U]} e^{\frac{|\mathbf{x}|^2}{4}} \, d\mathcal{H}^n  \\
=& \int_{(U\Delta\mathscr{R}[U])\cap\overline{\Omega_+(\Gamma_+)}} \left(\Div\mathbf{V}_R+\frac{\mathbf{x}}{2}\cdot\mathbf{V}_R\right) e^{\frac{|\mathbf{x}|^2}{4}} \, d\mathcal{L}^{n+1}\\
\geq & \int_{(U\Delta\mathscr{R}[U])\cap\overline{\Omega_+(\Gamma_+)}} \phi_{R,1}\left(\Div\mathbf{V}+\frac{\mathbf{x}}{2}\cdot\mathbf{V}\right) e^{\frac{|\mathbf{x}|^2}{4}} \, d\mathcal{L}^{n+1}-2\int_{\Omega^{\prime}\cap (\bar{B}_{R+1}\setminus B_R)} e^{\frac{|\mathbf{x}|^2}{4}} \, d\mathcal{L}^{n+1}.
\end{align*}
Thus, appealing to \eqref{EMCPositiveEqn} and \cite[Lemma 2.2]{BWExpanderRelEnt}, one sees
$$
E[\partial^* U,\Gamma_-;\phi_{R,1};\overline{\Omega_+(\Gamma_+)}]-E[\partial^* \mathscr{R}[U],\Gamma_-;\phi_{R,1};\overline{\Omega_+(\Gamma_+)}] \geq -CR^{-2}
$$
for some $C=C(\Omega^\prime)$. Likewise,
$$
E[\partial^* U,\Gamma_-;\phi_{R,1};\overline{\Omega_-(\Gamma_-)}]-E[\partial^* \mathscr{R}[U],\Gamma_-;\phi_{R,1};\overline{\Omega_-(\Gamma_-)}] \geq -CR^{-2}.
$$
As $\partial^* U\cap\Omega=\partial^*\mathscr{R}[U]\cap\Omega$, combining these estimates gives 
$$
E[\partial^* U,\Gamma_-;\phi_{R,1}]-E[\partial^*\mathscr{R}[U],\Gamma_-; \phi_{R,1}] \geq -2CR^{-2}.
$$
Hence, sending $R\to\infty$, it follows from Proposition \ref{RelEntropyProp} that
$$
E_{rel}[\partial^* U,\Gamma_-]-E_{rel}[\partial^*\mathscr{R}[U],\Gamma_-] \geq 0.
$$
The result follows from rearranging this inequality.
\end{proof}

For each $U\in \mathcal{C}(\Gamma_-, \Gamma_+)$ let $\Omega_U$ be the open region between $\overline{\partial^* U}$ and $\Gamma_-$, and let
$$
\mathbf{M}[U]= \int_{\Omega_U} e^{\frac{|\mathbf{x}|^2}{4}} \, d\mathcal{L}^{n+1}
$$
be the weighted volume of $\Omega_U$. We have, for $U\in \mathcal{C}(\Gamma_-, \Gamma_+)$,
$$
0\leq \mathbf{M}[U]\leq \mathbf{M}_0=\mathbf{M}[\Omega_-(\Gamma_+)]<\infty
$$
where the finiteness of $\mathbf{M}[\Omega_-(\Gamma_+)]$ follows from \cite[Lemma 2.2]{BWExpanderRelEnt}. Given $m\in [0, \mathbf{M}_0]$, we then let
$$
\mathcal{C}(\Gamma_-, \Gamma_+; m)=\set{U\in \mathcal{C}(\Gamma_-, \Gamma_+)\colon \mathbf{M}[U]=m}.
$$
Notice that, for $U\in \mathcal{C}(\Gamma_-, \Gamma_+)$, if $\mathbf{M}[U]=0$, then, up to a measure-zero set, $U=\Omega_-(\Gamma_-)$, and if $\mathbf{M}[U]=\mathbf{M}_0$, then $U=\Omega_-(\Gamma_+)$. Observe that it is possible to perturb the distance to $\overline{\Omega_-(\Gamma_-)}$ to a uniformly bounded Morse function $f$ on $\Omega$ which coincides with the distance  function near $\Gamma_-$. Thus the function $a\mapsto\mathbf{M}[\set{f<a}]$ is continuous and so, for any $0\leq m\leq \mathbf{M}_0$, $\mathcal{C}(\Gamma_-,\Gamma_+;m)$ is nonempty and we can define a number
$$
E_m=\inf\set{E_{rel}[\partial^* U,\Gamma_-]\colon U\in\mathcal{C}(\Gamma_-,\Gamma_+;m)}.
$$

\begin{lem}\label{ExistLem}
For each $m\in [0, \mathbf{M}_0]$ there is an element $U^\prime_m\in \mathcal{C}(\Gamma_-, \Gamma_+; m)$ so that $E_{rel}[\partial^* U^\prime_m, \Gamma_-]=E_m$.
\end{lem}

\begin{proof}
By Lemma \ref{RelEntropyBndLem} there is a constant $E_-=E_-(\Gamma_-,\Gamma_+) \leq 0$ so that, for any $U\in \mathcal{C}(\Gamma_-, \Gamma_+)$,
$$
E_{rel}[\partial^* U, \Gamma_-]\geq E_-.
$$
In particular, one has 
$$
E_m\geq E_->-\infty.
$$
We may assume $E_m<\infty$ as otherwise $E_{rel}[\partial^* U,\Gamma_-]=\infty$ for all such $U$ and so the claim holds trivially. Then pick a minimizing sequence $U_i\in \mathcal{C}(\Gamma_-, \Gamma_+; m)$ so that $E_{rel}[\partial^* U_i, \Gamma_-]\to E_m$. By Proposition \ref{RelEntropyProp}, up to passing to a subsequence, $\mathbf{1}_{U_i}\to \mathbf{1}_{U_\infty}$ in the weak-* topology of $BV_{loc}$ for some $U_{\infty}\in \mathcal{C}(\Gamma_-, \Gamma_+)$. Appealing to \cite[Lemma 2.2]{BWExpanderRelEnt} one has, for any $R>1$, 
$$
\int_{\Omega_{U_i}\setminus \bar{B}_R} e^{\frac{|\mathbf{x}|^2}{4}} \, d\mathcal{L}^{n+1}\leq C R^{-1}
$$
where $C=C(\Gamma_-,\Gamma_+)$. Thus it follows that $\mathbf{M}[U_{\infty}]=m$ and so $U_{\infty}\in \mathcal{C}(\Gamma_-, \Gamma_+; m)$. By the nature of convergence and Proposition \ref{RelEntropyProp}
$$
E_{rel}[\partial^* U_\infty, \Gamma_-]\leq \lim_{i\to \infty} E_{rel}[\partial^* U_i, \Gamma_-]=E_m.
$$ 
Hence $E_{rel}[\partial^* U_\infty, \Gamma_-]=E_m$ and so the claim follows with $U^\prime_m=U_\infty$.
\end{proof}

Given $R>0$ and $\epsilon>0$, let
$$
W_{R,\epsilon}=\left(B_R\cap \Omega\right)\backslash \overline{\mathcal{T}_{\epsilon}(\Gamma_-)}.
$$

\begin{lem}\label{OpenSetLem}
There are positive constants $\mathcal{R}^\prime_0=\mathcal{R}^\prime_0(\Gamma_-,\Gamma_+)$ and $\epsilon^\prime_0=\epsilon_0^\prime(\Gamma_-,\Gamma_+)$ so that if $U\in \mathcal{C}(\Gamma_-, \Gamma_+)$ satisfies $W_{\mathcal{R}^\prime_0,\epsilon_0^\prime}\cap U =\emptyset$, then $E_{rel}[\partial^* U, \Gamma_-]\geq 0$ and the inequality is strict when $\overline{\partial^* U}\neq \Gamma_-$.
\end{lem}

\begin{proof}
We may assume $E_{rel}[\partial^* U,\Gamma_-]<\infty$ as otherwise the claim holds trivially. Using Lemma \ref{MeanConvexLem} with $\Gamma=\Gamma_-$, one obtains a foliation $\set{\Gamma_s}_{s\in [0,\epsilon_0]}$ and a number $c_0>0$ so that
\begin{enumerate}
\item $\Gamma_0=\Gamma_-$ and $\Gamma_s\preceq\Gamma_{s^\prime}$ if $s\leq s^\prime$;
\item For $s>0$ the expander mean curvature of $\Gamma_s$ points towards $\Gamma_-$;
\item For $p\in\Gamma_{\epsilon_0}$, $\mathrm{dist}(p,\Gamma_-)\geq c_0\epsilon_0 (1+|\mathbf{x}(p)|^2)^{-\frac{1}{2}(n+1)+c_0} e^{-\frac{|\mathbf{x}(p)|^2}{4}}$.
\end{enumerate}
Let 
$$
\Omega_0=\bigcup_{s\in [0,\epsilon_0]}\Gamma_s
$$
and define the vector field $\mathbf{V}\colon\Omega_0\to\mathbb{R}^{n+1}$ by
$$
\mathbf{V}(p)=\mathbf{n}_{\Gamma_s}(p) \mbox{ if $p\in\Gamma_s$}.
$$
By Items (1) and (2)
\begin{equation} \label{EMCPositiveEqn2}
\Div\mathbf{V}(p)+\frac{\mathbf{x}(p)}{2}\cdot\mathbf{V}(p) \geq 0
\end{equation}
and the inequality is strict when $p\in\Gamma_s$ for some $s>0$.
	
Appealing to \cite[Proposition 2.1]{BWExpanderRelEnt} gives that 
$$
\Gamma_+\backslash B_{2\mathcal{R}}\subseteq \set{\mathbf{x}(p)+u(p)\mathbf{n}_{\Gamma_-}(p) \colon p\in \Gamma_-\backslash \bar{B}_{\mathcal{R}}} \subseteq \Gamma_+
$$
where $\mathcal{R}=\mathcal{R}(\Gamma_-,\Gamma_+)>1$ and $u$ satisfies
$$
0\leq u \leq C (1+|\mathbf{x}|^2)^{-\frac{1}{2}(n+1)}  e^{-\frac{|\mathbf{x}|^2}{4}}
$$
for some $C=C(\Gamma_-,\Gamma_+)$. Thus, by Item (3), there is a radius $\mathcal{R}_0^\prime=\mathcal{R}_0^\prime(\Gamma_-,\Gamma_+)>\mathcal{R}$ so that $\overline{\Omega}\setminus B_{\mathcal{R}_0^\prime}\subseteq\Omega_0$. Hence there is a constant $\epsilon^\prime_0=\epsilon^\prime_0(\Gamma_-,\Gamma_+)>0$ so that
$$
\overline{\Omega\setminus W_{\mathcal{R}^\prime_0,\epsilon^\prime_0}}=(\overline{\Omega}\setminus B_{\mathcal{R}^\prime_0})\cup (\overline{\Omega\cap \mathcal{T}_{\epsilon^\prime_0}(\Gamma_0)})\subseteq\Omega_0
$$
and so $\mathbf{V}$ is well defined on $\overline{\Omega\setminus W_{\mathcal{R}^\prime_0,\epsilon^\prime_0}}$.
	
As $U\in \mathcal{C}(\Gamma_-, \Gamma_+)$ satisfies $U\cap W_{\mathcal{R}^\prime_0,\epsilon^\prime_0}=\emptyset$, one has 
$$
\overline{\Omega_U}\subseteq \overline{\Omega\setminus W_{\mathcal{R}^\prime_0,\epsilon^\prime}}.
$$
Applying the divergence theorem to $\mathbf{V}_R=\phi_{R,1}\mathbf{V}$ gives
\begin{align*}
& \int_{\partial^* U} \phi_{R,1} e^{\frac{|\mathbf{x}|^2}{4}} \, d\mathcal{H}^n - \int_{\Gamma_-} \phi_{R,1} e^{\frac{|\mathbf{x}|^2}{4}} \, d\mathcal{H}^n \\
\geq & \int_{\partial^* U} \mathbf{V}_R\cdot\mathbf{n}_{\partial^* U} e^{\frac{|\mathbf{x}|^2}{4}} \, d\mathcal{H}^n - \int_{\Gamma_-}  \mathbf{V}_R\cdot\mathbf{n}_{\Gamma_-} e^{\frac{|\mathbf{x}|^2}{4}} \, d\mathcal{H}^n\\
= & \int_{\Omega_U} \left(\Div \mathbf{V}_R +\frac{\mathbf{x}}{2}\cdot \mathbf{V}_R\right) e^{\frac{|\mathbf{x}|^2}{4}} \, d\mathcal{L}^{n+1}\\
\geq & \int_{\Omega_U}\phi_{R,1} \left( \Div \mathbf{V} +\frac{\mathbf{x}}{2} \cdot \mathbf{V}\right) e^{\frac{|\mathbf{x}|^2}{4}} \, d\mathcal{L}^{n+1}-2\int_{\Omega\cap(\bar{B}_{R+1}\setminus B_R)} e^{\frac{|\mathbf{x}|^2}{4}} \, d\mathcal{L}^{n+1}.
\end{align*}
Thus, by \cite[Lemma 2.2]{BWExpanderRelEnt},
$$
E[\partial^* U,\Gamma_-; \phi_{R,1}] \geq \int_{\Omega_U}\phi_{R,1} \left( \Div \mathbf{V} +\frac{\mathbf{x}}{2} \cdot \mathbf{V}\right) e^{\frac{|\mathbf{x}|^2}{4}} \, d\mathcal{L}^{n+1}-C^\prime R^{-2}
$$
for some $C^\prime=C^\prime(\Omega)$. Hence, sending $R\to\infty$ and invoking \eqref{EMCPositiveEqn2}, it follows from the monotone convergence theorem and Proposition \ref{RelEntropyProp} that 
\begin{equation} \label{NonNegRelEntEqn}
E_{rel}[\partial^* U,\Gamma_-]\geq \int_{\Omega_U} \left( \Div\mathbf{V} +\frac{\mathbf{x}}{2} \cdot \mathbf{V}\right) e^{\frac{|\mathbf{x}|^2}{4}} \, d\mathcal{L}^{n+1} \geq 0.
\end{equation}
Finally, if $\overline{\partial^* U}\neq\Gamma_-$, then $\Omega_U\setminus\Gamma_-$ has positive measure and 
$$
\Div \mathbf{V}+\frac{\mathbf{x}}{2}\cdot\mathbf{V}>0 \mbox{ on $\Omega_U\setminus\Gamma_-$}.
$$
Thus, the second inequality in \eqref{NonNegRelEntEqn} is strict. That is, $E_{rel}[\partial^* U,\Gamma_-]>0$. This completes the proof.
\end{proof}

\begin{lem}\label{AreaEstLem}
There is a $\delta_1=\delta_1(\mathcal{R}^\prime_0, \epsilon^\prime_0)>1$ so that given $U\in \mathcal{C}(\Gamma_-, \Gamma_+)$ there are constants $\mathcal{R}^\prime_1=\mathcal{R}^\prime_1(U)\in (2\mathcal{R}^\prime_0,4\mathcal{R}^\prime_0)$ and $\epsilon^\prime_1=\epsilon^\prime_1(U)\in \left(\frac{1}{4} \epsilon^\prime_0, \frac{1}{2}\epsilon^\prime_0\right)$ so that
$$
E[U\cap \partial W_{\mathcal{R}^\prime_1, \epsilon^\prime_1}]=\int_{U\cap \partial W_{\mathcal{R}^\prime_1, \epsilon^\prime_1}} e^{\frac{|\mathbf{x}|^2}{4}} \, d\mathcal{H}^n<\delta_1 \mathbf{M}[U].
$$
\end{lem}

\begin{proof}
First, the coarea formula yields
$$
\int_{2\mathcal{R}^\prime_0}^{4\mathcal{R}^\prime_0} \int_{\partial B_\rho \cap \overline{\Omega_U}} e^{\frac{|\mathbf{x}|^2}{4}} \, d\mathcal{H}^n d\rho =\int_{\left(\bar{B}_{4\mathcal{R}^\prime_0} \backslash B_{2\mathcal{R}^\prime_0}\right)\cap \overline{\Omega_U}} e^{\frac{|\mathbf{x}|^2}{4}} \, d\mathcal{L}^{n+1}\leq \mathbf{M}[U].
$$
Hence,
$$
2\mathcal{R}^\prime_0 \inf_{2\mathcal{R}^\prime_0<\rho<4\mathcal{R}^\prime_0} \int_{\partial B_\rho \cap \overline{\Omega_U}} e^{\frac{|\mathbf{x}|^2}{4}} \, d\mathcal{H}^n \leq \mathbf{M}[U]
$$
and so there is a radius $\mathcal{R}^\prime_1\in (2\mathcal{R}^\prime_0, 4\mathcal{R}^\prime_0)$ so that
\begin{equation} \label{SphereAreaEqn}
\int_{\partial B_{\mathcal{R}^\prime_1} \cap \overline{\Omega_U}} e^{\frac{|\mathbf{x}|^2}{4}} d\mathcal{H}^n\leq \frac{1}{\mathcal{R}^\prime_0} \mathbf{M}[U].
\end{equation}
	
Next, let $f\colon \mathcal{T}_{\epsilon^\prime_0}(\Gamma_-) \backslash \Omega_-(\Gamma_-)\to [0, \epsilon^\prime_0)$ be the function given by $f(p)=\dist(p, \Gamma_-)$. The choice of $\epsilon^\prime_0$ ensures this function is Lipschitz and $|\nabla f|=1$. Thus, invoking the coarea formula again,
$$
\int_{\frac{\epsilon^\prime_0}{4}}^{\frac{\epsilon^\prime_0}{2}} \int_{ \set{f=\rho}\cap \overline{\Omega_U}\cap \bar{B}_{\mathcal{R}^\prime_1}} e^{\frac{|\mathbf{x}|^2}{4}} \, d\mathcal{H}^n d\rho=\int_{\set{\frac{\epsilon^\prime_0}{4} \leq f\leq \frac{\epsilon^\prime_0}{2} }\cap \overline{\Omega_U}\cap \bar{B}_{\mathcal{R}^\prime_1}} |\nabla f|  e^{\frac{|\mathbf{x}|^2}{4}} \, d\mathcal{L}^{n+1} \leq \mathbf{M}[U].
$$
Hence,
$$
\frac{\epsilon^\prime_0}{4} \inf_{\frac{\epsilon^\prime_0}{4} < \rho< \frac{\epsilon^\prime_0}{2} }\int_{ \set{f=\rho}\cap \overline{\Omega_U}\cap \bar{B}_{\mathcal{R}^\prime_1}} e^{\frac{|\mathbf{x}|^2}{4}} \, d\mathcal{H}^n \leq \mathbf{M}[U].
$$	
As such there is an $\epsilon^\prime_1\in (\frac{1}{4} \epsilon^\prime_0, \frac{1}{2}\epsilon^\prime_0)$ so that
\begin{equation} \label{FAreaEqn}
\int_{\set{f=\epsilon^\prime_1}\cap \overline{\Omega_U}\cap \bar{B}_{\mathcal{R}^\prime_1}} e^{\frac{|\mathbf{x}|^2}{4}} \, d\mathcal{H}^n \leq \frac{8}{\epsilon^\prime_0} \mathbf{M}[U].
\end{equation}
	
As
$$
E[U\cap \partial W_{\mathcal{R}^\prime_1, \epsilon^\prime_1}]\leq\int_{\partial B_{\mathcal{R}^\prime_1} \cap \overline{\Omega_U}} e^{\frac{|\mathbf{x}|^2}{4}} \, d\mathcal{H}^n+ \int_{\set{f=\epsilon^\prime_1}\cap\overline{\Omega_U}\cap \bar{B}_{\mathcal{R}^\prime_1}} e^{\frac{|\mathbf{x}|^2}{4}} \, d\mathcal{H}^n 
$$
the claim follows by combining \eqref{SphereAreaEqn}-\eqref{FAreaEqn} and choosing $\delta_1=\frac{1}{R^\prime_0}+\frac{8}{\epsilon^\prime_0}$.
\end{proof}

\begin{lem}\label{ELowBndLem}
There is a constant $\gamma_0=\gamma_0(\mathcal{R}^\prime_0,\epsilon^\prime_0)>0$ so that if $U\in  \mathcal{C}(\Gamma_-, \Gamma_+)$ is such that $\partial^*U$ is $E$-stationary in $B_{2\mathcal{R}^\prime_0} \backslash \overline{\mathcal{T}_{ \frac{1}{2}\epsilon^\prime_0}(\Gamma_-)}$ and $U\cap W_{\mathcal{R}^\prime_0, \epsilon^\prime_0}\neq \emptyset$, then
$$
E[ W_{2\mathcal{R}^\prime_0, \frac{1}{2}\epsilon^\prime_0}\cap \partial^* U ]=\int_{W_{2\mathcal{R}^\prime_0, \frac{1}{2}\epsilon^\prime_0}\cap \partial^* U } e^{\frac{|\mathbf{x}|^2}{4}} d\mathcal{H}^n >\gamma_0.
$$
\end{lem}

\begin{proof}
If $U\cap W_{\mathcal{R}^\prime_0, \epsilon^\prime_0}\neq \emptyset$, then there is a point $p\in \overline{W_{\mathcal{R}_0, \epsilon^\prime_0}} \cap \partial^* U $. Clearly, if $r_0=\min\set{\mathcal{R}^\prime_0, \frac{1}{2}\epsilon^\prime_0}>0$, then $\partial^* U$ is $E$-stationary in $B_{r_0}(p)$. As $B_{r_0}(p)\subset B_{2\mathcal{R}^\prime_0}$, one has that $|H_{\partial^* U}| \leq \mathcal{R}^\prime_0$. By the monotonicity formula \cite[Section 17]{SimonBook}, there is a constant $\gamma_0=\gamma_0(\mathcal{R}_0^\prime, r_0)>0$, which in turn depends on $\mathcal{R}_0^\prime$ and $\epsilon_0^\prime$, so that
$$
E[W_{2\mathcal{R}^\prime_0, \frac{1}{2}\epsilon^\prime_0}\cap \partial^* U ]\geq E[B_{r_0}(p)\cap \partial^* U] \geq \mathcal{H}^n (B_{r_0}(p)\cap \partial^* U)>\gamma_0
$$
proving the claim. 
\end{proof}

\begin{prop}\label{LowBndLowMeasProp}
There is an $m_0=m_0(\delta_1,\gamma_0)>0$ so that if $U\in \mathcal{C}(\Gamma_-, \Gamma_+)$ is such that $0<\mathbf{M}[U]<m_0$, then $E_{rel}[\partial^* U, \Gamma_-]> 0$.
\end{prop}

\begin{proof}
Given such $U$ and $\mathcal{R}^\prime_0$ and $\epsilon^\prime_0$ from Lemma \ref{OpenSetLem}, choose $\mathcal{R}^\prime_1$ and $\epsilon^\prime_1$ as in Lemma \ref{AreaEstLem} and let
$$
\mathcal{C}(\Gamma_-,\Gamma_+; U, W_{\mathcal{R}^\prime_1, \epsilon^\prime_1})=\set{W\in \mathcal{C}(\Gamma_-, \Gamma_+)\colon W\Delta U\subseteq W_{\mathcal{R}^\prime_1, \epsilon^\prime_1}}.
$$
By standard compactness results for sets of finite perimeter, there is an element $U^\prime\in \mathcal{C}(\Gamma_-,\Gamma_+; U, W_{\mathcal{R}^\prime_1, \epsilon^\prime_1})$ so that, for all $W\in\mathcal{C}(\Gamma_-,\Gamma_+; U,W_{\mathcal{R}^\prime_1, \epsilon^\prime_1})$, 
$$
E[\partial^* U^\prime \cap \overline{W_{\mathcal{R}^\prime_1, \epsilon^\prime_1}}]\leq E[\partial^* W \cap \overline{W_{\mathcal{R}^\prime_1, \epsilon^\prime_1}}].
$$
Clearly, one has
$$
E_{rel}[\partial^* U, \Gamma_-]\geq E_{rel}[\partial^* U^\prime, \Gamma_-].
$$
	
We now treat two situations. If $U^\prime\cap W_{\mathcal{R}^\prime_0, \epsilon^\prime_0}=\emptyset$, then Lemma \ref{OpenSetLem} implies that
$$
E_{rel}[\partial^* U, \Gamma_-]\geq E_{rel}[\partial^* U^\prime, \Gamma_-]\geq 0
$$
and the inequality is strict as otherwise, up to zero-measure sets, $U=U^\prime=\Omega_-(\Gamma_-)$ which contradicts $\mathbf{M}[U]>0$.
	
The other situation is $U^\prime\cap W_{\mathcal{R}^\prime_0, \epsilon^\prime_0}\neq\emptyset$. In this case, we first observe that $\partial^* U^\prime$ is $E$-stationary in $B_{2\mathcal{R}^\prime_0}\backslash \overline{\mathcal{T}_{\frac{1}{2}\epsilon^\prime_0}(\Gamma_-)}$.  Indeed, this is true by construction in $W_{2\mathcal{R}^\prime_0, \frac{1}{2}\epsilon^\prime_0}$. The only other points of $\partial^* U^\prime$ in $B_{2\mathcal{R}^\prime_0}\backslash \overline{\mathcal{T}_{\frac{1}{2}\epsilon^\prime_0}(\Gamma_-)}$ are those that lie in $\Gamma_+\cap B_{2\mathcal{R}^\prime_0}$. However, if this occurs, then the maximum principle (see Solomon-White \cite{SolomonWhite}) implies that 
$$
\partial^* U^\prime\cap (B_{2\mathcal{R}^\prime_0}\backslash W_{2\mathcal{R}^\prime_0, \frac{1}{2}\epsilon^\prime_0})\subseteq \Gamma_+\cap B_{2\mathcal{R}^\prime_0}
$$
and so $\partial^* U^\prime$ is automatically $E$-stationary at any such point. Next, by Lemma \ref{ELowBndLem},
$$
E[\partial^* U^\prime\cap W_{\mathcal{R}^\prime_1, \epsilon^\prime_1}]\geq E[\partial^* U^\prime\cap W_{2\mathcal{R}^\prime_0, \frac{1}{2}\epsilon^\prime_0}]>\gamma_0>0.
$$
Now let $U^{\prime\prime}=U\backslash \overline{W_{\mathcal{R}^\prime_1, \epsilon^\prime_1}}$. Clearly, $U^{\prime\prime} \in \mathcal{C}(\Gamma_-, \Gamma_+)$ satisfies $U^{\prime\prime}\cap W_{\mathcal{R}^\prime_0, \epsilon^\prime_0}=\emptyset$ and so, by Lemma \ref{OpenSetLem},
$$
E_{rel}[\partial^* U^{\prime\prime}, \Gamma_-]\geq 0.
$$
Finally, as
$$
\partial^* U^{\prime\prime}\subseteq (\partial^* U^\prime\backslash {W_{\mathcal{R}^\prime_1, \epsilon^\prime_1}}) \cup (U\cap \partial W_{\mathcal{R}^\prime_1, \epsilon^\prime_1})
$$
one has
$$
E_{rel}[\partial^* U^{\prime\prime}, \Gamma_-]\leq E_{rel}[\partial^* U^\prime, \Gamma_-]-E[\partial^* U^\prime\cap W_{\mathcal{R}^\prime_1, \epsilon^\prime_1}]+ E[U\cap \partial W_{\mathcal{R}^\prime_1, \epsilon^\prime_1}].
$$
Hence, by Lemmas \ref{AreaEstLem} and \ref{ELowBndLem}, 
$$
E_{rel}[\partial^* U^{\prime\prime}, \Gamma_-]\leq E_{rel}[\partial^* U^\prime, \Gamma_-]+\delta_1 \mathbf{M}[U] -\gamma_0.
$$
Picking $m_0<\delta_1^{-1} \gamma_0$ gives
$$
0\leq E_{rel}[\partial^* U^{\prime\prime}, \Gamma_-]< E_{rel}[\partial^* U^\prime, \Gamma_-]\leq E_{rel}[\partial^* U, \Gamma_-].
$$
This proves the claim.
\end{proof}

\begin{cor}\label{UniformCor}
If $0<m<m_0$, then $E_m>0$.
\end{cor}

\begin{proof}
For $m\in (0, m_0)$, let $U^\prime_m\in \mathcal{C}(\Gamma_-, \Gamma_+;m )$ be the element given by Lemma \ref{ExistLem}. Thus, by Proposition \ref{LowBndLowMeasProp}, one has $E_m=E_{rel}[\partial^* U^\prime_m, \Gamma_-]>0$ proving the claim.  
\end{proof}

We can now complete the proof of the main result of this section.

\begin{proof}[Proof of Proposition \ref{StrictLowBndRelEntProp}]
We treat two cases. The first is $E_{rel}[\Gamma_+,\Gamma_-]\leq 0$. As such, we are trying to show 
$$
\max_{\tau\in [0,1]} E_{rel}[\Sigma_\tau, \Gamma_-]\geq \delta_0>0
$$
for some $\delta_0=\delta_0(\Gamma_-,\Gamma_+)$. By definition of sweepouts and Lemma \ref{TruncateLem}, the map $\tau\mapsto\mathbf{M}[\mathscr{R}[U_\tau]]$ is a continuous map from $[0,1]$ to $[0, \mathbf{M}_0]$ and the value at $\tau=0$ is $0$ while the value at $\tau=1$ is $\mathbf{M}_0$. Moreover, for all $\tau\in [0,1]$,
$$
E_{rel}[\partial^* \mathscr{R}[U_\tau],\Gamma_-]\leq E_{rel}[\Sigma_\tau,\Gamma_-].
$$
As such, there is a $\tau_*\in (0,1)$ with $\mathbf{M}[\mathscr{R}[U_{\tau_*}]]=\frac{m_0}{2}=m_*$ where $m_0$ is given by Proposition \ref{LowBndLowMeasProp}. It then follows from Corollary \ref{UniformCor} that
$$
\max_{\tau\in [0,1]} E_{rel}[\Sigma_\tau, \Gamma_-] \geq E_{rel}[\Sigma_{\tau_*},\Gamma_-]\geq E_{rel}[\partial^*\mathscr{R}[U_{\tau_*}],\Gamma_-]\geq E_{m_*}>0.
$$
As $E_{m_*}$ depends only on $\Gamma_-$ and $\Gamma_+$,  we prove the claim by setting $\delta_0=E_{m_*}$.
	
The second is $E_{rel}[\Gamma_+,\Gamma_-]\geq 0$. As such, we need to show 
$$
\max_{\tau\in [0,1]} E_{rel}[\Sigma_\tau,\Gamma_-] \geq E_{rel}[\Gamma_+,\Gamma_-]+\delta_0 \geq \delta_0.
$$
Reversing orientation (which swaps $\Gamma_-$ with $\Gamma_+$), by what we have shown 
$$
\max_{\tau\in [0,1]} E_{rel}[\Sigma_\tau,\Gamma_+]\geq \delta_0.
$$
Thus, as 
$$
\max_{\tau\in [0,1]} E_{rel}[\Sigma_\tau,\Gamma_-]=E_{rel}[\Gamma_+,\Gamma_-]+\max_{\tau\in [0,1]} E_{rel}[\Sigma_\tau,\Gamma_+]
$$
the claim follows immediately.
\end{proof}

\section{Existence of sweepouts and proof of the results} \label{ExistSweepoutSec}
In this section we complete the proof of Theorem \ref{MainThm} and Corollaries \ref{RefinedMainCor} and \ref{TopologyCor}. This will mostly require us to show the existence of at least one sweepout of $\tilde{\Omega}$.  In order to do this, we adapt arguments of \cite[Lemma 11.1]{DeLellisRamic} and \cite{Milnor} to prove the existence of sweepouts. 

\begin{prop} \label{ExistSweepoutProp}
There exists a sweepout of $\tilde{\Omega}$.
\end{prop}

\begin{proof}
We will construct a suitable Morse function on $\overline{\Omega}$, the closure of the open region between $\Gamma_-$ and $\Gamma_+$,  and use it to obtain the desired sweepout. First, by our hypotheses on $\Gamma_-$ and $\Gamma_+$, there is a radius $\mathcal{R}=\mathcal{R}(\Gamma_-,\Gamma_+)>1$ and a foliation $\set{\Xi_s}_{s\in [0,1]}$ of $\overline{\Omega}\setminus\bar{B}_{\mathcal{R}}$ so that
\begin{enumerate}
\item $\Xi_0=(\Gamma_-\cap\overline{\Omega})\setminus\bar{B}_{\mathcal{R}}$ and $\Xi_1=(\Gamma_+\cap\overline{\Omega})\setminus\bar{B}_{\mathcal{R}}$;
\item Each $\Xi_s$ is given by a smooth graph over $\Gamma_-$ outside a compact set;
\item For every $s$, $\sup_{p\in\Xi_s} |\mathbf{x}(p)| |A_{\Xi_s}(p)|\leq K$ for some $K=K(\Gamma_-,\Gamma_+)$.
\end{enumerate}
Using this foliation one can define $f_0\colon\overline{\Omega}\setminus\bar{B}_{\mathcal{R}}\to [0,1]$ by 
$$
f_0(p)=s \mbox{ if $p\in\Xi_s$}.
$$
Clearly, $f_0^{-1}(0)=(\Gamma_-\cap\overline{\Omega})\setminus\bar{B}_{\mathcal{R}}$, $f_0^{-1}(1)=(\Gamma_+\cap\overline{\Omega})\setminus\bar{B}_{\mathcal{R}}$ and $f_0$ has no critical points.
	
Next let $\epsilon=\epsilon(\Gamma_-,\Gamma_+)>0$ be sufficiently small so that the distance to $\Gamma_\pm$, $\mathrm{dist}(\cdot,\Gamma_\pm)$ on $\mathcal{T}_{2\epsilon}(\Gamma_\pm)$ are smooth functions without critical points. Shrinking $\epsilon$ if needed, we may assume $\mathcal{T}_{2\epsilon}(\Gamma_-)$ and $\mathcal{T}_{2\epsilon}(\Gamma_+)$ are disjoint in $\overline{\Omega}\cap B_{2\mathcal{R}}$. Set $W_1=\mathcal{T}_{2\epsilon}(\Gamma_-)\cap\overline{\Omega}\cap B_{2\mathcal{R}}$ and define $f_1\colon W_1 \to \mathbb{R}$ by 
$$
f_1(p)=\mathrm{dist}(p,\Gamma_-).
$$ 
Likewise, set $W_2=\mathcal{T}_{2\epsilon}(\Gamma_+)\cap\overline{\Omega}\cap B_{2\mathcal{R}}$ and define $f_2\colon W_2\to\mathbb{R}$ by 
$$
f_2(p)=1-\mathrm{dist}(p,\Gamma_+).
$$
Thus, $0\leq f_1 <1$ with $f_1^{-1}(0)=\Gamma_-\cap\overline{\Omega}\cap B_{2\mathcal{R}}$,  and $0<f_2\leq 1$ with $f_2^{-1}(1)=\Gamma_+\cap\overline{\Omega}\cap B_{2\mathcal{R}}$. Let $W_3=(\Omega\cap B_{2\mathcal{R}})\setminus(\overline{\mathcal{T}_\epsilon(\Gamma_-)}\cup\overline{\mathcal{T}_\epsilon(\Gamma_+)})$ and $f_3\colon W_3\to\mathbb{R}$ be the constant-$\frac{1}{2}$ function.
	
Set $W_0=\mathbb{R}^{n+1}\setminus\bar{B}_{\mathcal{R}}$ and the family $\set{W_i}_{0\leq i\leq 3}$ forms an open cover of $\overline{\Omega}$. Let $\set{\varphi_i}_{0\leq i\leq 3}$ be a partition of unity subordinate to this cover. Now define $f\colon \overline{\Omega}\to\mathbb{R}$ by
$$
f(p)=\sum_{0\leq i\leq 3} \varphi_i(p) f_i(p).
$$
By definition $f^{-1}(0)=\Gamma_-\cap\overline{\Omega}$, $f^{-1}(1)=\Gamma_+\cap\overline{\Omega}$ and $f$ has no critical points in $\overline{\Omega}\setminus B_{2\mathcal{R}}$ or $\overline{W_1}\cup \overline{W_2}$. Indeed, only $\varphi_0\neq 0$ in $\overline{\Omega}\setminus B_{2\mathcal{R}}$ and so $f=f_0$, as remarked before, has no critical points. One readily computes, at $p\in\Gamma_-\cap\overline{\Omega}\cap B_{2\mathcal{R}}$,
$$
\nabla_{\mathbf{n}_{\Gamma_-}} f(p)=\sum_{0\leq i\leq 3} \nabla_{\mathbf{n}_{\Gamma_-}}\varphi_i(p) f_i(p)+\sum_{0\leq i\leq 3} \varphi_i(p) \nabla_{\mathbf{n}_{\Gamma_-}} f_i(p).
$$
First observe that the first sum vanishes. By construction $\varphi_2(p)=\varphi_3(p)=0$ and $\nabla_{\mathbf{n}_{\Gamma_-}} f_i(p)>0$ for $i\in \set{0,1}$. Thus the second sum is positive and so $\nabla_{\mathbf{n}_{\Gamma_-}} f(p)>0$. Hence, by further shrinking $\epsilon$ if necessary, it follows that $f$ has no critical points in $\overline{W_1}$. Similar arguments show the same for $f$ in $\overline{W_2}$. This proves the claim.
	
To make $f$ Morse, we need to modify $f$ in $W_3$. As the set of Morse functions is an open dense subset of $C^2$ functions, for sufficiently small $\delta>0$ (which will be chosen later) there is a Morse function $g$ on $W=(\overline{\Omega}\cap B_{4\mathcal{R}})\setminus (\overline{\mathcal{T}_{\frac{1}{2}\epsilon}(\Gamma_-)}\cup\overline{\mathcal{T}_{\frac{1}{2}\epsilon}(\Gamma_+)})$ with 
$$
\Vert g-f \Vert_{C^2}<\delta.
$$
Let $\phi\colon\overline{\Omega}\to [0,1]$ be a cutoff function so that $\phi=1$ in $W_3$ and $\phi=0$ outside $B_{4\mathcal{R}}$ and in $\frac{\epsilon}{2}$-tubular neighborhood of $\Gamma_\pm$. Define $h\colon\overline{\Omega}\to\mathbb{R}$ by 
$$
h(p)=\phi(p) g(p)+\left(1-\phi(p)\right) f(p).
$$
Clearly, $h^{-1}(0)=\Gamma_-\cap\overline{\Omega}$ and $h^{-1}(1)=\Gamma_+\cap\overline{\Omega}$. By what we have shown, $h$ has no degenerate critical points in the region where $\phi=0$ or $\phi=1$. In the intermediate region, i.e., $0<\phi<1$, applying the chain rule gives
$$
D h=Df+D\phi (g-f)+\phi D(g-f).
$$
As $f$ has no critical points outside $B_{2\mathcal{R}}$ or in $\epsilon$-tubular neighborhood of $\Gamma_\pm$, one has $|Df|>\eta>0$ in the region where $0<\phi<1$. Thus, by the triangle inequality,
$$
|Dh|\geq |Df|-|D\phi| |g-f|-|D(g-f)| \geq \eta-C\Vert g-f \Vert_{C^2} \geq \eta-C\delta
$$
where $C>0$ depends only on $\phi$ (which in turn depends only on $\epsilon$). Hence, choosing $\delta<\eta/C$, it follows that $|Dh|>0$ in the region where $0<\phi<1$. Therefore we have shown $h$ is a Morse function. 
	
Finally,  for $\tau\in [0,1]$, let $U_\tau=\set{h<\tau}\cup\Omega_-(\Gamma_-)$ and $\Sigma_\tau=\partial^* U_\tau$. We show that $\set{(U_\tau,\Sigma_\tau)}_{\tau\in [0,1]}$ is a sweepout of $\overline{\Omega}$. By construction, except for the first and third all other properties in Definition \ref{ParamFamilyDef} are satisfied. To show the rest, as $\Sigma_\tau\setminus B_{4\mathcal{R}}=\Xi_\tau\setminus B_{4\mathcal{R}}$, invoking Items (2) and (3), it follows from Proposition \ref{RelEntropyAnnuliProp} that, for all $R>\bar{\mathcal{R}}_1>4\mathcal{R}$,
$$
E_{rel}[\Sigma_\tau,\Gamma_-; \mathbb{R}^{n+1}\setminus B_R]<\bar{K}_1 R^{-2}
$$
where $\bar{K}_1$ and $\bar{\mathcal{R}}_1$ depend on $\Gamma_-,\Gamma_+,\mathcal{R}$ and $K$. Thus it follows that the map $\tau\mapsto\Sigma_\tau$ is continuous in the weak-* topology of $\mathfrak{Y}^*(\overline{\Omega^\prime})$ and $E_{rel}[\Sigma_\tau,\Gamma_-]<\infty$ for every $\tau\in [0,1]$. This proves the claim and completes the proof.
\end{proof}

We now prove Theorem \ref{MainThm}.

\begin{proof}[Proof of Theorem \ref{MainThm}]
By Proposition \ref{ExistSweepoutProp}, there is a sweepout of $\tilde{\Omega}$. Let $X$ be the homotopically closed set of families in $\overline{\Omega^\prime}$ generated by this sweepout. As any element of $X$ is a sweepout of $\tilde{\Omega}$, invoking Proposition \ref{StrictLowBndRelEntProp} gives
$$
bM_{rel}(X)=\max\set{E_{rel}[\Gamma_+,\Gamma_-],0}< m_{rel}(X).
$$
Hence, appealing to Theorem \ref{MinMaxThm}, one obtains a $C^2$-asymptotically conical (possibly singular) self-expander $\Gamma_0$ with $\cC(\Gamma_0)=\cC$ and $\Gamma_-\preceq\Gamma_0\preceq\Gamma_+$ and that has codimension-$7$ singular set and $E_{rel}[\Gamma_0,\Gamma_-]=m_{rel}(X)$, in particular, $\Gamma_0\neq\Gamma_\pm$.
\end{proof}

We now use Theorem \ref{MainThm} to prove Corollaries \ref{RefinedMainCor} and \ref{TopologyCor}.

\begin{proof}[Proof of Corollary \ref{RefinedMainCor}]
Let $\Gamma_0^1$ be the self-expander produced by applying Theorem \ref{MainThm} to $\Gamma_\pm$. If $\Gamma_0^1\cap\Omega\neq\emptyset$, then we are done with $\Gamma_0=\Gamma_0^1$. Otherwise, $\Gamma_0^1$ is the union of some components of $\Gamma_-$ and $\Gamma_+$. Thus, $\Gamma_0^1$ is a strictly stable $C^2$-asymptotically conical self-expander with $\mathcal{C}(\Gamma_0^1)=\cC$ and $\Gamma_-\preceq\Gamma_0^1\preceq\Gamma_+$. As common components of $\Gamma_-$ and $\Gamma_+$ are also components of $\Gamma_0^1$ and $\Gamma_0^1\neq\Gamma_\pm$, $\Gamma_-$ shares strictly more common components with $\Gamma_0^1$ than $\Gamma_+$. Apply Theorem \ref{MainThm} to $\Gamma_-$ and $\Gamma_0^1$ and iterate the previous arguments. After iterating $l$ times, we end up with two situations. The first situation is that $\Gamma_0^l\cap\Omega\neq\emptyset$. Thus we stop and set $\Gamma_0=\Gamma_0^l$. The second is that except one component of $\Gamma_-$ all others are also components of $\Gamma_0^l$. Again, applying Theorem \ref{MainThm} to $\Gamma_-$ and $\Gamma_0^l$ gives a $C^2$-asymptotically conical self-expander $\Gamma_0$ with $\mathcal{C}(\Gamma_0)=\cC$ and so $\Gamma_-\preceq\Gamma_0\preceq\Gamma_0^l\preceq\Gamma_+$ and 
$$
\Gamma_0\cap\Omega\supseteq\Gamma_0\cap \left(\Omega_-(\Gamma_0^l)\setminus\overline{\Omega_-(\Gamma_-)}\right)\neq\emptyset.
$$
This proves the first claim.

To see the second, suppose $\cC$ is generic. By Schoen-Simon's \cite{SchoenSimon} compactness and uniform asymptotic estimates \cite[Proposition 3.3]{BWProper}, the set of stable self-expanders $C^2$-asymptotic to $\cC$ is a finite set and consists only of strictly stable elements. Hence one can apply Theorem \ref{MainThm} finitely many times and obtain an unstable self-expander trapped between $\Gamma_-$ and $\Gamma_+$.
\end{proof}

\begin{proof}[Proof of Corollary \ref{TopologyCor}]
Let $\mathcal{C}_0$ be a generic $C^3$-regular cone that is very close to a rotationally symmetric double cone with sufficiently large apertures. If $\cC$ is a cone $C^3$-close to $\cC_0$, then the results of \cite{BWBanach} ensure that $\cC$ is generic as well. Thus, one appeals to \cite{EHAnn} to get a unique disconnected smooth self-expander $\Sigma_0$ asymptotic to $\cC$. In particular, $\Sigma_0$ is strictly stable. Using an appropriate rotationally symmetric self-expander as a barrier and a minimizing argument for the expander functional (which is sketched by Ilmanen \cite{IlmanenNotes} and carried out by Ding \cite{Ding}), one obtains a connected strictly stable self-expander $\Sigma_1$ asymptotic to $\cC$ -- See \cite[Lemma 8.2]{BWDegree} for details. By the strong maximum principle we may assume $\Sigma_0\preceq\Sigma_1$. Hence, by the uniqueness of disconnected self-expanders asymptotic to $\cC$, Theorem \ref{MainThm} implies that there is a third connected smooth self-expander $\Sigma$ asymptotic to $\cC$ and $\Sigma_0\preceq\Sigma\preceq\Sigma_1$. This completes the proof.
\end{proof}

\appendix

\section{Estimates on the flow of vector fields} \label{FlowAppendix}
We collect here various estimates and properties of the vector fields and associated flows used in Section \ref{FlowSec}. Throughout this appendix, let $\mathbf{Y}=\alpha\mathbf{Y}_0+\mathbf{Y}_1\in\mathcal{Y}^-(\overline{\Omega^\prime})$ and $\set{\Phi(t)}_{t\geq 0}$ be the vector field and the family of diffeomorphisms in $\overline{\Omega^\prime}$, respectively, as given in Proposition \ref{FirstVarProp}. For convenience, we define the following Banach space 
$$
C^0_d(Y)=\set{f\in C^0_{loc}(Y)\colon \Vert f\Vert_{C^0_d}<\infty}
$$
with the weighted norm
$$
\Vert f\Vert_{C^0_d}=\sup_{p\in Y} (1+|\mathbf{x}(p)|)^d |f(p)|.
$$

First is an estimate on the asymptotic properties of the vector field $\mathbf{Y}$.

\begin{lem} \label{DecayVectorLem}
There is a constant $\tilde{C}_0=\tilde{C}_0(\Omega^\prime,\Gamma_-,M_0)>0$ so that 
$$
\Vert \mathbf{Y}-\alpha\chi|\mathbf{x}|^{-2}\mathbf{x}\Vert_{C^0_{-3}}+\sum_{l=1}^3\Vert\nabla^l\mathbf{Y}\Vert_{C^0_{-2}} \leq \tilde{C}_0.
$$
\end{lem}

\begin{proof}
Write
$$
\mathbf{Y}-\alpha\chi|\mathbf{x}|^{-2}\mathbf{x}=\mathbf{Y}_1-\alpha\chi|\mathbf{x}|^{-2}(\mathbf{x}\cdot\mathbf{N})\mathbf{N}.
$$
The hypotheses on $\mathbf{Y}_1$ and Lemma \ref{VectorFieldLem} ensure that there is a constant $\tilde{C}_0^\prime=\tilde{C}_0^\prime(\Omega^\prime,\Gamma_-, M_0)>1$ so that
$$
\Vert\mathbf{Y}-\alpha\chi|\mathbf{x}|^{-2}\mathbf{x}\Vert_{C^0_{-3}} \leq \tilde{C}_0^\prime.
$$

Next, we compute, on $\overline{\Omega^\prime}\setminus\bar{B}_{\mathcal{R}_1+1}$,
$$
\nabla_i\mathbf{Y}_0=|\mathbf{x}|^{-2}\left(\mathbf{e}_i-2\mathbf{x}_i\mathbf{Y}_0\right)+|\mathbf{x}|^{-2}\mathbf{P}_{i},
$$
\begin{align*}
\nabla_j\nabla_i\mathbf{Y}_0 &= |\mathbf{x}|^{2} \nabla_i \mathbf{Y}_0 \nabla_j |\mathbf{x}|^{-2}+|\mathbf{x}|^{-2} \nabla_j \left(\mathbf{e}_i-2\mathbf{x}_i\mathbf{Y}_0\right) +|\mathbf{x}|^{-2}  \nabla_j \mathbf{P}_i 
\\ &=-2|\mathbf{x}|^{-2}\left(\delta_{ij}\mathbf{Y}_0+\mathbf{x}_j\nabla_i\mathbf{Y}_0+\mathbf{x}_i\nabla_j\mathbf{Y}_0 \right)+|\mathbf{x}|^{-2}\nabla_j\mathbf{P}_{i},
\end{align*}
and
\begin{multline*}
\nabla_k\nabla_j\nabla_i\mathbf{Y}_0=-2|\mathbf{x}|^{-2} \left(\mathbf{x}_k\nabla_j\nabla_i\mathbf{Y}_0+\mathbf{x}_i\nabla_j\nabla_k\mathbf{Y}_0+\mathbf{x}_j\nabla_k\nabla_i\mathbf{Y}_0\right)\\
-2|\mathbf{x}|^{-2}\left(\delta_{ij}\nabla_k\mathbf{Y}_0+\delta_{jk}\nabla_i\mathbf{Y}_0+\delta_{ki}\nabla_j\mathbf{Y}_0\right)+|\mathbf{x}|^{-2}\nabla_k\nabla_j\mathbf{P}_i
\end{multline*}
where $\mathbf{e}_i$ is the $i$-th coordinate vector, $\mathbf{x}_i=\mathbf{x}\cdot\mathbf{e}_i$, and
$$
\mathbf{P}_i=-(\mathbf{e}_i\cdot\mathbf{N})\mathbf{N}-(\mathbf{x}\cdot\nabla_i\mathbf{N})\mathbf{N}-(\mathbf{x}\cdot\mathbf{N})\nabla_i\mathbf{N}.
$$
It is readily checked that there is a constant $C=C(\Omega^\prime,\Gamma_-)>0$ so that, for all $i$,
$$
\sum_{l=0}^2\Vert\nabla^l\mathbf{P}_i\Vert_{C^0} \leq C.
$$
Thus, inductively, it follows that there is a constant $C^\prime=C^\prime(\Omega^\prime,\Gamma_-)>0$ so that
$$
\sum_{l=1}^3 \Vert\nabla^l\mathbf{Y}_0\Vert_{C^0_{-2}} \leq C^\prime,
$$
and this together with the hypotheses on $\mathbf{Y}_1$ implies the desired estimates on $\nabla^l\mathbf{Y}$ for $1\leq l \leq 3$.
\end{proof}

\begin{lem} \label{FlowEstLem}
There is a constant $\tilde{C}_1=\tilde{C}_1(\Omega^\prime, \Gamma_-,M_0,T)>0$ so that, for all $0\leq t \leq T$, 
$$
\Vert|\Phi(t)|^2-|\mathbf{x}|^2-2\alpha t\Vert_{C^0_{-2}}+\Vert\nabla\Phi(t)-\mathbf{I}_{n+1}\Vert_{C^0_{-2}}+\sum_{l=2,3}\Vert\nabla^l\Phi(t)\Vert_{C^0_{-2}} \leq \tilde{C}_1
$$
where $\mathbf{I}_{n+1}$ is the $(n+1)\times(n+1)$ identity matrix.
\end{lem}

\begin{proof}
As $\frac{\partial}{\partial t}\Phi(t,p)=\mathbf{Y}(\Phi(t,p))$ and $\Phi(0,p)=p$, it follows that
\begin{equation} \label{DistanceDiffEqn}
\frac{\partial}{\partial t}|\Phi(t,p)|^2=2\Phi(t,p)\cdot\mathbf{Y}(\Phi(t,p)) \mbox{ and } |\Phi(0,p)|^2=|\mathbf{x}(p)|^2.
\end{equation}
By the triangle inequality and Lemma \ref{DecayVectorLem},
\begin{align*}
|\mathbf{x}\cdot\mathbf{Y}-\alpha| & \leq |\mathbf{x}\cdot (\mathbf{Y}-\alpha\chi |\mathbf{x}|^{-2}\mathbf{x})|+|\alpha(1-\chi)| \\
& \leq \left(\tilde{C}_0+M_0(2+\mathcal{R}_1)^2\right)(1+|\mathbf{x}|)^{-2}.
\end{align*}
Thus, integrating \eqref{DistanceDiffEqn} gives that, for all $(t,p)\in [0,T]\times\overline{\Omega^\prime}$,
$$
|\Phi(t,p)|^2 \geq |\mathbf{x}(p)|^2-2\left(\tilde{C}_0+M_0\left(1+(2+\mathcal{R}_1)^2\right)\right)T
$$
and so there is a constant $C=C(\Omega^\prime, \Gamma_-, M_0, T)>1$ so that
\begin{equation} \label{DistanceEqn}
C(1+|\Phi(t,p)|)\geq (1+|\mathbf{x}(p)|).
\end{equation}
Thus one readily computes, for $(t,p)\in [0,T]\times\overline{\Omega^\prime}$,
\begin{align*}
\left||\Phi(t,p)|^2-|\mathbf{x}(p)|^2-2\alpha t \right| & \leq \int_0^{t} \left|\frac{\partial}{\partial \tau}|\Phi(\tau,p)|^2-2\alpha\right| d\tau \\
& =2\int_0^{t} |\mathbf{x}\cdot\mathbf{Y}-\alpha|(\Phi(\tau,p)) \, d\tau \\
& \leq 2\left(\tilde{C}_0+M_0(2+\mathcal{R}_1)^2\right)\int_0^{t} (1+|\Phi(\tau,p)|)^{-2}\, d\tau \\
& \leq 2\left(\tilde{C}_0+M_0(2+\mathcal{R}_1)^2\right) T C^2 (1+|\mathbf{x}(p)|)^{-2}.
\end{align*}
Hence, as long as $\tilde{C}_1\geq 2(\tilde{C}_0+M_0(2+\mathcal{R}_1)^2)TC^2$,
$$
\Vert|\Phi(t)|^2-|\mathbf{x}|^2-2\alpha t\Vert_{C^0_{-2}} \leq \tilde{C}_1.
$$
	
Similarly, differentiating the equation for $\Phi$ up to three times gives that
$$
\frac{\partial}{\partial t} \nabla_i\Phi=(\nabla\mathbf{Y}\circ\Phi)\cdot\nabla_i\Phi,
$$
$$
\frac{\partial}{\partial t} \nabla_j\nabla_i\Phi(t,p) = (\nabla\mathbf{Y}\circ\Phi)\cdot\nabla_j\nabla_i\Phi+(\nabla^2\mathbf{Y}\circ\Phi)(\nabla_i\Phi,\nabla_j\Phi),
$$
and
\begin{align*}
\frac{\partial}{\partial t}\nabla_k\nabla_j\nabla_i\Phi=& (\nabla\mathbf{Y}\circ\Phi)\cdot\nabla_k\nabla_j\nabla_i\Phi+(\nabla^2\mathbf{Y}\circ\Phi)(\nabla_k\Phi,\nabla_j\nabla_i\Phi) \\
&+(\nabla^2\mathbf{Y}\circ\Phi)(\nabla_i\Phi,\nabla_k\nabla_j\Phi)+(\nabla^2\mathbf{Y}\circ\Phi)(\nabla_j\Phi,\nabla_k\nabla_i\Phi)\\
& +(\nabla^3\mathbf{Y}\circ\Phi)(\nabla_i\Phi,\nabla_j\Phi,\nabla_k\Phi).
\end{align*}
Observe that by our hypotheses on $\mathbf{Y}$ and the standard ODE theory
$$
\sup_{0\leq t\leq T}\Vert\Phi(t)\Vert_{C^3} \leq C^\prime
$$
where $C^\prime=C^\prime(\Omega^\prime, \Gamma_-,T)>0$. Hence, as $\nabla\Phi(0,p)=\mathbf{I}_{n+1}$ and $\nabla^l\Phi(0,p)=\mathbf{0}$ for $l=2,3$, integrating the above equations and appealing to Lemma \ref{DecayVectorLem} and \eqref{DistanceEqn}, gives the desired estimates for the covariant derivatives of $\Phi(t)$.
\end{proof}

\begin{cor} \label{WeightCor}
There is a constant $\tilde{C}_2=\tilde{C}_2(\Omega^\prime,\Gamma_-, M_0,T)>0$ so that, for any $0\leq a,t \leq T$,
\begin{multline*}
e^{\frac{1}{4}(|\Phi(t,p)|^2-|\mathbf{x}(p)|^2)}= e^{\frac{1}{4}(|\Phi(a,p)|^2-|\mathbf{x}(p)|^2)} \left(1+\frac{1}{2}(t-a)(\mathbf{x}\cdot\mathbf{Y})\circ\Phi(a,p)\right) \\
+(t-a)^2 Q_0(a,t,p) \in \mathfrak{Y}(\overline{\Omega^\prime})
\end{multline*}
where $Q_0(a,t)=Q_0(a,t,\cdot)$ satisfies 
$$
\left\Vert Q_0(a,t)-\frac{1}{4}\alpha^2\int_0^1 e^{\frac{1}{2}\alpha \left(a+(t-a)\rho\right)} (1-\rho)\, d\rho \right\Vert_{C^0_{-2}}+\Vert\nabla Q_0(a,t)\Vert_{C^0_{-1}} \leq \tilde{C}_2.
$$
\end{cor}

\begin{proof}
It is convenient to set
$$
f(t,p)=e^{\frac{1}{4}(|\Phi(t,p)|^2-|\mathbf{x}(p)|^2)}.
$$
By the Taylor expansion and the chain rule,
$$
f(t,p)=f(a,p)\left(1+\frac{1}{2}(t-a)(\mathbf{x}\cdot\mathbf{Y})\circ\Phi(a,p)\right)+(t-a)^2 Q_0(a,t,p)
$$
where
$$
Q_0(a,t,p)= \frac{1}{4}\int_0^1 h(a+(t-a)\rho,p) f(a+(t-a)\rho,p)(1-\rho) \, d\rho
$$
and
$$
h(t,p)=\left((\mathbf{x}\cdot\mathbf{Y})^2+2|\mathbf{Y}|^2+2\mathbf{x}\cdot \nabla_{\mathbf{Y}}\mathbf{Y} \right)\circ\Phi(t,p).
$$
By Lemma \ref{DecayVectorLem} there is a constant $K=K(\Omega^\prime,\Gamma_-,M_0)>0$ so that
\begin{multline*}
\Vert(\mathbf{x}\cdot\mathbf{Y})^2-\alpha^2\Vert_{C^0_{-2}}+\Vert\nabla(\mathbf{x}\cdot\mathbf{Y})^2\Vert_{C^0_{-1}}+\Vert |\mathbf{Y}|^2\Vert_{C^0_{-2}}+\Vert\nabla |\mathbf{Y}|^2 \Vert_{C^0_{-3}} \\
+\Vert\mathbf{x}\cdot \nabla_{\mathbf{Y}}\mathbf{Y}\Vert_{C^0_{-2}} +\Vert\nabla(\mathbf{x}\cdot\nabla_{\mathbf{Y}}\mathbf{Y})\Vert_{C^0_{-2}}\leq K.
\end{multline*}
Together with estimates on $\Phi(t,p)$, Lemma \ref{FlowEstLem}, and the chain rule, it follows that
$$
\sup_{0\leq t\leq T} \left(\Vert h(t,\cdot)-\alpha^2\Vert_{C^0_{-2}} +\Vert\nabla h(t,\cdot)\Vert_{C^0_{-1}}\right)\leq C
$$
where $C=C(\Omega^\prime,\Gamma_-,M_0,T)>0$. Likewise, one uses Lemmas \ref{DecayVectorLem} and \ref{FlowEstLem} to obtain the estimates for $f$: 
$$
\sup_{0\leq t\leq T} \left(\Vert f(t,\cdot)-e^{\frac{1}{2}\alpha t}\Vert_{C^0_{-2}}+\Vert\nabla f(t,\cdot)\Vert_{C^0_{-1}}\right) \leq C^\prime
$$
where $C^\prime=C^\prime(\Omega^\prime,\Gamma_-,M_0,T)>0$. Hence, the desired estimate for $Q_0$ follows from combining these estimates. In particular, for all $0\leq a,t\leq T$, $Q_0(a,t)\in\mathfrak{Y}(\overline{\Omega^\prime})$. Again, appealing to Lemma \ref{DecayVectorLem} gives $\mathbf{x}\cdot\mathbf{Y}\in\mathfrak{Y}(\overline{\Omega^\prime})$. Thus, fixing $a=0$, it follows that $f(t,\cdot)\in\mathfrak{Y}(\overline{\Omega^\prime})$ for all $0\leq t\leq T$, completing the proof.
\end{proof}

\begin{cor} \label{JacobiCor}
There is a constant $\tilde{C}_3=\tilde{C}_3(\Omega^\prime,\Gamma_-,M_0,T)>0$ so that, for any $0\leq a,t\leq T$,
$$
J\Phi(t,p,\mathbf{v})=J\Phi(a,p,\mathbf{v})+(t-a)\frac{\partial}{\partial t}J\Phi(a,p,\mathbf{v})+(t-a)^2 Q_1(a,t,p,\mathbf{v})
$$
where $J\Phi(t,\cdot,\cdot),\frac{\partial}{\partial t}J\Phi(t,\cdot,\cdot)\in\mathfrak{Y}(\overline{\Omega^\prime})$ satisfy $J\Phi(0,p,\mathbf{v})=1$ and 
$$
\frac{\partial}{\partial t}J\Phi(0,p,\mathbf{v})=\Div \mathbf{Y}(p)-Q_{\nabla\mathbf{Y}}(p,\mathbf{v})=\Div \mathbf{Y}(p)-\nabla_{\mathbf{v}}\mathbf{Y}(p)\cdot\mathbf{v},
$$
and $Q_1(a,t)=Q_1(a,t,\cdot,\cdot)$ satisfies, for any $R\geq 0$,
$$
\Vert Q_1(a,t)\Vert_{\mathfrak{X}(\overline{\Omega^\prime}\setminus B_R)} \leq \tilde{C}_3(R+1)^{-2},
$$
so $Q_1(a,t)\in\mathfrak{Y}(\overline{\Omega^\prime})$.
\end{cor}

\begin{proof}
By the Taylor expansion,
\begin{equation} \label{TaylorEqn}
J\Phi(t,p,\mathbf{v})=J\Phi(a,p,\mathbf{v})+(t-a)\frac{\partial}{\partial t}J\Phi(a,p,\mathbf{v})+(t-a)^2 Q_1(a,t,p,\mathbf{v})
\end{equation}
where
$$
Q_1(a,t,p,\mathbf{v})=\int_0^1\frac{\partial^2}{\partial t^2} J\Phi(a+(t-a)\rho, p,\mathbf{v})(1-\rho)\, d\rho.
$$
For $\mathbf{v}\in\mathbb{S}^n$, choose an orthonormal basis $\mathbf{T}_{\mathbf{v}}=\set{\tau_l\colon 1\leq l \leq n}$ for $T_{\mathbf{v}}\mathbb{S}^n$. Notice that, for any $\mathbf{v}_0\in\mathbb{S}^n$ fixed, one may choose $\mathbf{T}_{\mathbf{v}}$ so it smoothly depends on $\mathbf{v}$ in a neighborhood of $\mathbf{v}_0$. Define $\mathbf{b}(t,p,\mathbf{T}_\mathbf{v})=(b_{ij})$ to be a matrix-valued function given by
$$
b_{ij}=\nabla_{\tau_i}\Phi\cdot\nabla_{\tau_j}\Phi.
$$
Then
$$
J\Phi(t,p,\mathbf{v})=\sqrt{\det(\mathbf{b}(t,p,\mathbf{T}_{\mathbf{v}}))}
$$
of which the right side is independent of the choice of orthonormal basis for $T_{\mathbf{v}}\mathbb{S}^n$. By Jacobi's formula
\begin{equation} \label{FirstDerDetEqn}
\frac{\partial}{\partial t}J\Phi=\frac{1}{2}(J\Phi)^{-1} \mathrm{tr}\left(\mathrm{adj}(\mathbf{b})\frac{\partial\mathbf{b}}{\partial t}\right).
\end{equation}
Appealing to the equation for $\Phi$ gives
\begin{equation}  \label{FirstDerMatrixEqn}
\frac{\partial}{\partial t} b_{ij}=((\nabla\mathbf{Y}\circ\Phi)\cdot\nabla_{\tau_i}\Phi)\cdot\nabla_{\tau_j}\Phi+((\nabla\mathbf{Y}\circ\Phi)\cdot\nabla_{\tau_j}\Phi)\cdot\nabla_{\tau_i}\Phi.
\end{equation}
As $\Phi(0,p)=p$ and $\nabla\Phi(0,p)=\mathbf{I}_{n+1}$, one has
$$
J\Phi(0,p,\mathbf{v})=1 \mbox{ and } \frac{\partial }{\partial t}\Phi(0,p,\mathbf{v})=\Div\mathbf{Y}(p)-Q_{\nabla\mathbf{Y}}(p,\mathbf{v}).
$$
Moreover, by Lemmas \ref{DecayVectorLem} and \ref{FlowEstLem}, there is a constant $C=C(\Omega^\prime,\Gamma_-,M_0,T)>0$ so that 
$$
\Vert (J\Phi)^{-1}(t,\cdot,\cdot)\Vert_{C^2}+\Vert\mathrm{adj}(\mathbf{b}(t,\cdot,\cdot))\Vert_{C^2} \leq C
$$
and, for all $R \geq 0$,
$$
\left\Vert\frac{\partial }{\partial t}\mathbf{b}(t,\cdot,\cdot)\right\Vert_{\mathfrak{X}(\overline{\Omega^\prime}\setminus B_R)} \leq \frac{C}{R+1}.
$$
This together with the algebra property of $\mathfrak{X}$ implies $\frac{\partial}{\partial t}J\Phi(t,\cdot,\cdot)\in\mathfrak{Y}(\overline{\Omega^\prime})$.
	
Next we estimate $Q_1$. Differentiating \eqref{FirstDerDetEqn} gives
\begin{align*}
\frac{\partial^2}{\partial t^2}J\Phi=&\frac{1}{2}(J\Phi)^{-1}\mathrm{tr}\left(\frac{\partial}{\partial t}\mathrm{adj}(\mathbf{b})\frac{\partial\mathbf{b}}{\partial t}+\mathrm{adj}(\mathbf{b})\frac{\partial^2\mathbf{b}}{\partial t^2}\right)\\
&-\frac{1}{4}(J\Phi)^{-3}\left(\mathrm{tr}\left(\mathrm{adj}(\mathbf{b})\frac{\partial\mathbf{b}}{\partial t}\right)\right)^2.
\end{align*}
Differentiating \eqref{FirstDerMatrixEqn} and appealing to the equation of $\Phi$ give 
\begin{align*}
&\frac{\partial^2}{\partial t^2} b_{ij}= (\nabla^2\mathbf{Y}\circ\Phi)(\mathbf{Y}\circ\Phi,\nabla_{\tau_i}\Phi)\cdot\nabla_{\tau_j}\Phi+(\nabla^2\mathbf{Y}\circ\Phi)(\mathbf{Y}\circ\Phi,\nabla_{\tau_j}\Phi)\cdot\nabla_{\tau_i}\Phi\\
&+((\nabla\mathbf{Y}\circ\Phi)\cdot((\nabla\mathbf{Y}\circ\Phi)\cdot\nabla_{\tau_j}\Phi))\cdot\nabla_{\tau_i}\Phi+((\nabla\mathbf{Y}\circ\Phi)\cdot((\nabla\mathbf{Y}\circ\Phi)\cdot\nabla_{\tau_i}\Phi))\cdot\nabla_{\tau_j}\Phi\\
&+2((\nabla\mathbf{Y}\circ\Phi)\cdot\nabla_{\tau_i}\Phi)\cdot((\nabla\mathbf{Y}\circ\Phi)\cdot\nabla_{\tau_j}\Phi).
\end{align*}
Invoking Lemmas \ref{DecayVectorLem} and \ref{FlowEstLem} again, there is a constant $C^\prime=C^\prime(\Omega^\prime,\Gamma_-,M_0,T)>C$ so that, for any $R \geq 0$,
$$
\left\Vert\frac{\partial}{\partial t}\mathrm{adj}(\mathbf{b}(t,\cdot,\cdot))\right\Vert_{\mathfrak{X}(\overline{\Omega^\prime}\setminus B_R)}+(R+1)\left\Vert\frac{\partial^2}{\partial t^2}\mathbf{b}(t,\cdot,\cdot)\right\Vert_{\mathfrak{X}(\overline{\Omega^\prime}\setminus B_R)} \leq \frac{C^\prime}{R+1}.
$$
Combining this with the previous estimates on $(J\Phi)^{-1}$, $\mathrm{adj}(\mathbf{b})$ and $\frac{\partial\mathbf{b}}{\partial t}$, it follows from the algebra property of $\mathfrak{X}$ that
$$
\left\Vert\frac{\partial^2}{\partial t^2} J\Phi\right\Vert_{\mathfrak{X}(\overline{\Omega^\prime}\setminus B_R)} \leq c(n)(C^\prime+1)^7 (R+1)^{-2}.
$$
Hence, the claimed estimate for $Q_1$ follows by choosing $\tilde{C}_3\geq c(n)(C^\prime+1)^7$. In particular, $Q_1(a,t)\in\mathfrak{Y}(\overline{\Omega^\prime})$. 
	
Finally, by Lemma \ref{DecayVectorLem} one has $\Div\mathbf{Y}-Q_{\nabla\mathbf{Y}}\in\mathfrak{Y}(\overline{\Omega^\prime})$. As, fixing $a=0$, \eqref{TaylorEqn} gives
$$
J\Phi(t,p,\mathbf{v})=1+t\left(\Div\mathbf{Y}(p)-Q_{\nabla\mathbf{Y}}(p,\mathbf{v})\right)+t^2Q_1(0,t,p,\mathbf{v}),
$$
by what we have shown $J\Phi(t,\cdot,\cdot)$ is an element of $\mathfrak{Y}(\overline{\Omega^\prime})$. This completes the proof.
\end{proof}

\begin{lem} \label{ComposeLem}
Let $\Phi\colon\overline{\Omega^\prime}\to\overline{\Omega^\prime}$ be a local diffeomorphism which satisfies, for some constant $\tilde{M}_0>0$, 
$$
\Vert \Phi \Vert_{C^2}+\Vert |\Phi|^2-|\mathbf{x}|^2 \Vert_{C^0}\leq \tilde{M}_0<\infty.
$$
There is a constant $\tilde{C}_4=\tilde{C}_4(n,\tilde{M}_0)>0$ so that if $\psi$ is an element of $\mathfrak{Y}(\overline{\Omega^\prime})$, then so is $\Phi^{\#}\psi$ and, moreover,
$$
\Vert\Phi^{\#}\psi\Vert_{\mathfrak{Y}}\leq \tilde{C}_4 \Vert\psi\Vert_{\mathfrak{Y}}.
$$
\end{lem}

\begin{proof}
By definition $\psi=c\mathbf{1}+\psi_0$ where $c\in \Real $ and $\psi_0\in \mathfrak{Y}_0(\overline{\Omega^\prime})$. Fix any $\epsilon>0$ and there is a pair $(\xi,\zeta)\in\mathfrak{W}(\overline{\Omega^\prime})\times\mathfrak{X}(\overline{\Omega^\prime})$ so that $\psi_0=\xi+\zeta$ and 
$$
\Vert\xi\Vert_{\mathfrak{W}}+\Vert\zeta\Vert_{\mathfrak{X}}\leq \Vert\psi_0\Vert_{\mathfrak{Y}}+\epsilon.
$$
By our hypotheses, $|\mathbf{x}(p)|\leq \tilde{M}_0^{1/2} +|\Phi(p)|$ and so
$$
\sup_{p\in\overline{\Omega^\prime}}\frac{1+|\mathbf{x}(p)|}{1+|\Phi(p)|} \leq 4\left(1+\tilde{M}_0^{1/2}\right).
$$
One readily checks that
\begin{align*}
(1+|\mathbf{x}(p)|)^{n+1} e^{\frac{|\mathbf{x}(p)|^2}{4}} |\Phi^\#\xi(p,\mathbf{v})| & \leq \left(\frac{1+|\mathbf{x}(p)|}{1+|\Phi(p)|}\right)^{n+1} e^{\frac{1}{4}(|\mathbf{x}(p)|^2-|\Phi(p)|^2)} \Vert\xi\Vert_{\mathfrak{W}} \\
& \leq 4^{n+1} \left(1+\tilde{M}_0^{1/2}\right)^{n+1} e^{\frac{\tilde{M}_0}{4}}\Vert\xi\Vert_{\mathfrak{W}}.
\end{align*}
Thus, choosing $K_0=4^{n+1}(1+\tilde{M}_0^{1/2})^{n+1} e^{\tilde{M}_0/4}$, one has
$$
\Vert\Phi^\#\xi\Vert_{\mathfrak{W}} \leq K_0 \Vert\xi\Vert_{\mathfrak{W}}.
$$
	
Next, by the chain rule and the fact that $\overline{\Omega^\prime}$ is quasi-convex, 
$$
\Vert\Phi^{\#}\zeta\Vert_{Lip}+\Vert\nabla_{\mathbb{S}^n}(\Phi^{\#}\zeta)\Vert_{Lip}\leq c(n) (1+\tilde{M}_0)^2 \Vert\zeta\Vert_{\mathfrak{Y}}.
$$
Furthermore, one evaluates that
\begin{align*}
(1+|\mathbf{x}(p)|) |\nabla_{\mathbb{S}^n} (\Phi^{\#}\zeta)|(p,\mathbf{v}) & \leq (1+|\mathbf{x}(p)|) |\nabla_{\mathbb{S}^n}\zeta|(\Phi(p),\nabla_{\mathbf{v}}\Phi(p))|\nabla\Phi|(p)\\
& \leq \frac{1+|\mathbf{x}(p)|}{1+|\Phi(p)|}\Vert\zeta\Vert_{\mathfrak{X}} \tilde{M}_0 \leq 4 \left(1+\tilde{M}_0^{1/2}\right)\tilde{M}_0\Vert\zeta\Vert_{\mathfrak{X}}.
\end{align*}
Thus, by choosing $K_1=c(n)(1+\tilde{M}_0)^2+4(1+\tilde{M}_0^{1/2}) \tilde{M}_0+1\geq 1$, one has
$$
\Vert\Phi^\#\zeta\Vert_{\mathfrak{X}}\leq K_1 \Vert\zeta\Vert_{\mathfrak{X}}.
$$
	
Hence, combining these estimates gives that, for $\tilde{C}^\prime_4=\max\set{K_0,K_1}$,
$$
\Vert\Phi^\#\psi_0\Vert_{\mathfrak{Y}} \leq \Vert\Phi^\#\xi\Vert_{\mathfrak{W}}+\Vert\Phi^\#\zeta\Vert_{\mathfrak{X}} \leq  \tilde{C}_4^{\prime} (\Vert\xi\Vert_{\mathfrak{W}}+\Vert\zeta\Vert_{\mathfrak{X}}) \leq  \tilde{C}_4^{\prime} (\Vert\psi_0\Vert_{\mathfrak{Y}}+\epsilon)
$$
Sending $\epsilon\to 0$, it follows that
$$
\Vert\Phi^\#\psi_0\Vert_{\mathfrak{Y}} \leq  \tilde{C}_4^{\prime} \Vert\psi_0\Vert_{\mathfrak{Y}}.
$$
	
As $\psi_0\in \mathfrak{Y}_0(\overline{\Omega^\prime})$, Item (6) of Proposition \ref{YSpaceProp} ensures there is a sequence $\psi_i\in C^\infty_c(\overline{\Omega^\prime}\times\mathbb{S}^n)$ that converges in the $\mathfrak{Y}$ norm to $\psi_0$, the hypotheses on $\Phi$ ensure that each $\Phi^\#\psi_i$ has compact support, and by what we have shown these elements converge in the $\mathfrak{Y}$ norm to $\Phi^\# \psi_0$.  As such, $\Phi^\# \psi_0\in \mathfrak{Y}_0(\overline{\Omega^\prime})$.  Clearly, $\Phi^\# \mathbf{1}=\mathbf{1}$ and so $ \Phi^\# \psi=c\mathbf{1}+ \Phi^\# \psi_0 \in\mathfrak{Y}(\overline{\Omega^\prime})$.  Finally, appealing to Item (4) of Proposition \ref{YSpaceProp} gives
$$
\Vert \Phi^\# \psi \Vert_{\mathfrak{Y}} =\Vert c \mathbf{1} +  \Phi^\# \psi_0 \Vert_{\mathfrak{Y}}\leq |c|+  \tilde{C}_4^{\prime}\Vert\psi_0\Vert_{\mathfrak{Y}}\leq 4 \tilde{C}_4^{\prime} \Vert \psi \Vert_{\mathfrak{Y}}.
$$
This completes the proof with $\tilde{C}_4=4 \tilde{C}_4^{\prime}$.
\end{proof}

\end{document}